\documentclass{amsart}

\usepackage{amsmath,amsthm,amssymb,amsfonts,enumerate,color}
\usepackage{mathrsfs}
\usepackage{amstext,amsxtra}
\usepackage[margin=1.2in]{geometry}
\usepackage{graphicx}
\usepackage{float}

\usepackage{tikz}
\usetikzlibrary{matrix,arrows.meta,calc,intersections,fit}
\usetikzlibrary{decorations.markings}
\usepackage{tikz-cd}
\usepackage{textcomp}

\usepackage[colorlinks,urlcolor=black,linkcolor=blue,citecolor=blue,hypertexnames=false]{hyperref}
\allowdisplaybreaks

{\theoremstyle{plain}
\newtheorem{theorem}{Theorem}[section]
\newtheorem{proposition}[theorem]{Proposition}
\newtheorem{corollary}[theorem]{Corollary}
\newtheorem{lemma}[theorem]{Lemma}}
{\theoremstyle{definition}
\newtheorem{example}[theorem]{Example}
\newtheorem{definition}[theorem]{Definition}
}
{\theoremstyle{remark}
\newtheorem{remark}[theorem]{Remark}}

\newcommand{\mult}{\operatorname{mult}}
\newcommand{\aut}{\operatorname{Aut}}

\newcommand{\vdim}{\operatorname{vdim}}

\newcommand{\di}{\operatorname{div}}
\newcommand{\area}{\operatorname{Area}}
\newcommand{\floor}{\operatorname{floor}}

\newcommand{\rl}{\mathcal{R}}

\newcommand{\cl}{\mathcal{C}}
\newcommand{\pl}{\mathcal{P}}
\newcommand{\ml}{\mathcal{M}}
\newcommand{\ql}{\mathcal{Q}}
\newcommand{\xk}{\mathfrak{X}}
\newcommand{\ol}{\mathcal{O}}

\newcommand{\dl}{\mathcal{D}}

\newcommand{\cf}{\mathbf{c}}
\newcommand{\df}{\mathbf{d}}
\newcommand{\lf}{\mathbf{l}}
\newcommand{\rf}{\mathbf{r}}
\newcommand{\sff}{\mathbf{s}}
\newcommand{\fdf}{\mathbf{FD}}
\newcommand{\af}{\mathbf{a}}
\newcommand{\bbf}{\mathbf{b}}
\newcommand{\Sff}{\mathbf{S}}

\newcommand{\rb}{\mathbb{R}}

\newcommand{\cb}{\mathbb{C}}

\newcommand{\zb}{\mathbb{Z}}
\newcommand{\pb}{\mathbb{P}}

\newcommand{\qb}{\mathbb{Q}}
\newcommand{\eb}{\mathbb{E}}

\newcommand{\wt}{\widetilde}
\newcommand{\undl}{\underline}

\title[Universal polynomials for Gromov-Witten invariants]{Universal polynomials for Gromov-Witten invariants on toric surfaces}

\author{Yanqiao Ding}
\author{Dadi Ni}

\address{School of Mathematics and Statistics, Zhengzhou University, Zhengzhou, 450001, China}

\email{yqding@zzu.edu.cn}

\address{School of Mathematics and Statistics, Henan University, Kaifeng, 475004, China}

\email{nidd@henu.edu.cn}



\date{\today}

\begin{document}

\begin{abstract}
We study the polynomiality of combined logarithmic and ordinary 
Gromov-Witten invariants with a Lambda class insertion on toric surfaces $\Delta(\cf,\df)$.
We show that the combined logarithmic Gromov-Witten invariants can be 
expressed in terms of universal relative node polynomials when the direction vector $\cf$ 
and length vector $\df$ of the $h$-transverse lattice
polygon $\Delta(\cf,\df)$ are sufficiently large and spread out.
In particular, when the normal fan of the $h$-transverse lattice polygon $\Delta(\cf,\df)$ is smooth and its edges satisfy a length condition,
we show that the combined Gromov-Witten invariants are equal to a universal node polynomial. 
We also show that the higher genus Block-G\"ottsche polynomials of smooth toric surfaces 
corresponding to these $h$-transverse lattice polygons, including toric blowups beyond Hirzebruch and toric del Pezzo surfaces, 
are equal to a generating function of higher genus Gromov-Witten invariants.
\end{abstract}

\maketitle
\tableofcontents

\section{Introduction}
Gromov--Witten invariants are virtual counts of curves in algebraic varieties.
By using the virtual fundamental classes, we often study these counts together
with descendant and Lambda class insertions \cite{beh1997,BF-1997,li-tian-1998,fp-2000}. On toric
surfaces, Mikhalkin's correspondence theorem relates complex curve counts and Welschinger invariants 
to weighted counts of tropical curves \cite{mikhalkin05,wel2005a}. 
Refined enumerative geometry builds a bridge between the complex curve counts and their real analogues. 
G\"ottsche--Shende introduced refined curve counting on complex surfaces 
using invariants of relative Hilbert schemes \cite{gs14}. 
Block--G\"ottsche defined tropical refined invariants by 
replacing tropical vertex multiplicities in Mikhalkin's correspondence theorem 
by quantum integers \cite{bg2016}.
The resulting Laurent polynomials contain both complex and real enumerative information.
Their invariance under variation of the generic point conditions was proved by Itenberg--Mikhalkin \cite{im13}. 
Motivic integration provides a further approach to their geometric interpretation
\cite{nps18}. 
These developments make it possible to study the algebraic structure of the Gromov--Witten invariants
via the structure of tropical refined counts.

Two geometric interpretations of the tropical refined invariants are particularly relevant to our study. 
Mikhalkin used the quantum index to give a
refinement of real rational curve counts with prescribed toric boundary conditions \cite{mik17}.
Blomme extended this approach to more general boundary configurations \cite{blomme24-jag}. Bousseau
established a different correspondence: after the change of variables $q=e^{iu}$, refined tropical
counts of a toric surface determine generating series of higher genus logarithmic Gromov--Witten
invariants with a Lambda class insertion \cite{bou2019}. 
Bousseau also used higher genus logarithmic invariants to
construct quantum mirrors of log Calabi--Yau surfaces \cite{bousseau-quantum-mirrors}.
Refined floor diagram formulas for higher genus relative Gromov--Witten invariants of the projective plane and
Hirzebruch surfaces were established in \cite{bou2021}. 
Kennedy-Hunt--Shafi--Urundolil Kumaran extended the tropical refined correspondence to descendant
insertions in \cite{ksu24}. The surface BPS polynomial constructions of
Arg\"uz--Bousseau and the higher genus relative floor diagram
correspondence explore the generalization of tropical refined counts to non-toric surfaces \cite{arguz-bousseau-2025,dh2026}.

Floor diagrams provide an effective way to compute enumerative invariants. 
Floor diagrams were introduced by Brugall\'e--Mikhalkin through the floor decomposition of tropical curves in \cite{bm2007,bm2008}. By using the template decomposition of floor diagrams,
Fomin--Mikhalkin \cite{FM-2010} proved the polynomiality of plane Severi degrees $N^{d,\delta}$ at fixed cogenus $\delta$,
proving the conjecture proposed by Di Francesco and Itzykson in \cite{dfi-95}. 
Ardila--Block extended the polynomiality of Severi degrees to toric surfaces
corresponding to $h$-transverse polygons $\Delta(\cf,\df)$ under additional 
hypotheses \cite{AB-2013}. 
Block obtained the relative node polynomials for plane curves in \cite{Block-2012}.
The above polynomialities of Severi degrees were generalized to the refined cases in \cite{bg2016}.
By applying the combinatorial tool of long edge graphs and their logarithmic generating functions in
\cite{BCK-2014,liu2016,liu-osserman2018}, G\"ottsche--Kikwai \cite{gk16} proved that
the refined Severi degrees of $\pb^2$ and rational ruled surfaces 
are multiplicative in the topological invariants of the surfaces and the line bundle.
When allowing different data to vary, one can obtain other polynomiality results for enumerative invariants.
Ardila--Brugall\'e proved piecewise polynomiality in the contact orders of
double Gromov--Witten invariants of Hirzebruch surfaces \cite{ardila-brugalle2017}.
Hahn--Reda generalized this piecewise polynomiality to the curve 
counts in toric surfaces corresponding to $h$-transverse polygons \cite{hr-2024}.
Corey--Markwig--Ranganathan obtained polynomiality and piecewise polynomiality results for rational
logarithmic Gromov--Witten invariants on $\pb^1\times\pb^1$ 
with specified boundary tangency conditions \cite{corey-markwig-ranganathan2025}.
When the genus is fixed, Brugall\'e--Jaramillo Puentes proved polynomiality of coefficients of small codegree
of the tropical refined invariants \cite{brugalle-jaramillo2022}.
M\'evel obtained the universal formulas for these coefficients of the genus zero tropical refined invariants \cite{mevel2026}.
Blomme--M\'evel computed the associated asymptotic refined series in genera zero and one \cite{blomme-mevel2025}. 
For ruled surfaces over an elliptic curve, Blomme used floor diagrams to prove piecewise polynomiality of
relative invariants and quasimodularity of generating series \cite{blomme23-compos}. 
When comparing these polynomiality results, it is important to note the difference between 
fixed cogenus, fixed genus and fixed coefficient codegree.

Logarithmic Gromov--Witten theory provides a framework for 
studying curve counting problems on algebraic varieties with boundary divisors. 
Following relative stable maps and degeneration methods in
\cite{lijun-2001}, the theory of stable logarithmic maps was constructed by Gross--Siebert, Chen and
Abramovich--Chen \cite{gs-2013,Chen-2014,AbraChen-2014}. 
Nishinou--Siebert used toric degenerations
to establish algebraic--tropical correspondences in \cite{ns-2006}.
The decomposition formula of
Abramovich--Chen--Gross--Siebert \cite{acgs-2020} and the expanded theory of Ranganathan
\cite{ranganathan2022} give systematic tools for decomposing virtual classes under degenerations. 
Abramovich--Wise proved logarithmic Gromov--Witten invariants with primary insertions are
invariant under logarithmic modifications \cite{aw2018}.

Our main goal is to study the polynomial structure of logarithmic and ordinary
Gromov--Witten invariants of toric surfaces through Bousseau's tropical refined correspondence theorem. 
G\"ottsche conjecture asserted that the Severi degrees of any smooth projective surface $X$
with a sufficiently ample line bundle $L$ are given by universal polynomials \cite{gottsche-1998}.
Tzeng proved the G\"ottsche--Yau--Zaslow formula in \cite{tzeng2012}, 
and Kool--Shende--Thomas gave another proof of G\"ottsche conjecture using 
Hilbert schemes of points in \cite{kool-shende-thomas2011}. 
Janda--Pandharipande--Pixton--Zvonkine showed that the double ramification cycle exhibits polynomial dependence on ramification data through
tautological classes \cite{jppz2017}.
Tseng--You showed that higher genus orbifold Gromov--Witten invariants of root stacks are polynomials in $r$ when $r$ is sufficiently large in \cite{ty-2020},
and they studied the degrees and coefficients of these polynomials in \cite{tseng-you2022}.
Fan--Wu--You generalized this polynomial result to orbifold invariants of root stacks with large ages in \cite{fwy-2020}. 
Ranganathan--Urundolil Kumaran expressed logarithmic Gromov--Witten cycles for toric pairs in terms of logarithmic double ramification cycles and piecewise polynomial classes in \cite{ranganathan-kumaran2024}. 
These polynomiality results reflect the algebraic structure of the enumerative invariants.

\subsection*{Main results}
We use the notation in \cite{AB-2013} to denote the $h$-transverse polygon.
A lattice polygon $\Delta$ in $\rb^2$ is the convex hull of a finite subset of $\zb^2$. 
A lattice polygon $\Delta$ is called \textit{$h$-transverse} provided each edge of $\Delta$ has slope $0,\infty$ or $\frac{1}{k}$ for some integer $k$.
Given an $h$-transverse lattice polygon $\Delta$, we use $d_b$ and $d_t$ to denote the lengths of the bottom and top edges of $\Delta$, if they exist.
Otherwise, let $d_b=0$ or $d_t=0$.
Let $p_r,p_l$ be the numbers of edges on the right and left sides of $\Delta$, respectively.
List the edges on the right side of the polygon $\Delta$ clockwise from top to bottom, and assume that these edges have direction vectors $(c^1_r,-1),\dots,(c^{p_r}_r,-1)$ with respective lattice lengths $d^1_r,\dots,d^{p_r}_r$, so the integers $c_r^1,\ldots,c_r^{p_r}$ satisfy $c^1_r>\cdots>c^{p_r}_r$.
List the edges on the left side of $\Delta$ in counterclockwise order from top to bottom, whose direction vectors are $(c^1_l,-1),\dots,(c^{p_l}_l,-1)$ with lattice lengths $d^1_l,\dots,d^{p_l}_l$; then $c^1_l<\cdots<c^{p_l}_l$.
Note that
$d_t+\sum_{i}c^i_r d^i_r-d_b-\sum_j c^j_l d^j_l=0$.
Let
$\cf_r=(c^1_r,\dots,c^{p_r}_r)$, $\df_r=(d^1_r,\dots,d^{p_r}_r)$, $\cf_l=(c^1_l,\dots,c^{p_l}_l)$, $\df_l=(d^1_l,\dots,d^{p_l}_l)$,
and denote $\cf=(\cf_r;\cf_l)$, $\df=(d_t;\df_r;\df_l)$.
The data $\cf,\df$ determine $\Delta$ up to translation, and we have $\sum_i d_r^i=\sum_jd_l^j$.
We use $\Delta(\cf,\df)$ to denote such an $h$-transverse polygon.
The entries of $\cf$ are the slopes associated to all non-vertical primitive rays in the normal fan of $\Delta(\cf,\df)$.
Let $d_r(\Delta(\cf,\df))$ be the multiset of \textit{right directions} of the $h$-transverse polygon $\Delta(\cf,\df)$ containing each right direction $c_r^i$ repeated $d_r^i$ times. The multiset of  \textit{left directions} $d_l(\Delta(\cf,\df))$ is defined similarly. The \textit{height} of $\Delta(\cf,\df)$ is the cardinality of $d_l(\Delta(\cf,\df))$ or $d_r(\Delta(\cf,\df))$.

The normal fan of $\Delta(\cf,\df)$ determines a normal projective toric surface $X_\cf$ and an ample toric Cartier divisor $L_\Delta$. Let $\rho$ be a ray in the normal fan of $\Delta(\cf,\df)$. Denote by $D_\rho$ the corresponding toric boundary divisor.
The class $\beta_\Delta=\beta_\cf(\df)$ of the divisor $L_\Delta$ satisfies $\beta_\Delta\cdot D_\rho=d_\rho$, where $d_\rho$ is the lattice length of the edge of $\Delta$ with outward normal ray $\rho$.
For fixed $\cf$, the class varies linearly with $\df$ on a stratum with fixed presence or absence of the horizontal edges.

We also use the following notations.
For a partition $\eta$, let $|\eta|$ be the sum of its parts, $l(\eta)$ its number of parts, and $m_w(\eta)$ the multiplicity of $w$.
Unions of partitions are multiset unions.
Let $\Delta(\cf,\df)$ be an $h$-transverse polygon with bottom edge length $d_b>0$. Let $\eta,\mu$ be two partitions such that $|\eta|+|\mu|=d_b$.
The combined logarithmic Gromov--Witten invariants 
$N_{g,\beta_\Delta,n}^{X_\cf|\partial X_\cf,m}(\eta,\mu)$ are defined in Definition \ref{def:comb-log-gw}. They share the same combinatorial structure with the refined counts of floor diagrams $N_{\floor}^{\Delta,g}(\eta,\mu;q)$.
For $\delta\geq1$, we say $\cf,\df$ are {\it sufficiently large and spread out} if $\cf$ and $\df$ satisfy the following relations.
\begin{equation}\label{eq:poly-range}
\begin{aligned}
&d_t,d_b\geq \delta,\qquad
d_t+c_r^1-c_l^1\geq2\delta,\qquad d_b+c_r^{p_r}-c_l^{p_l}\geq2\delta,\\
&d_r^i,d_l^j\geq \delta+1\quad(1\leq i\leq p_r,\ 1\leq j\leq p_l),\\
&|\sum_{v=1}^i d_r^v-\sum_{v=1}^j d_l^v|\geq \delta+2\quad(1\leq i<p_r,\ 1\leq j<p_l).
\end{aligned}
\end{equation}
These are the numerical hypotheses of \cite[Theorem 1.2]{AB-2013}.
For the general varying $\cf$ statements below we also impose
\begin{equation}\label{eq:poly-gap-range}
c_r^i-c_r^{i+1}\geq \delta+1\quad(i<p_r),\qquad
c_l^{j+1}-c_l^j\geq \delta+1\quad(j<p_l).
\end{equation}
These are the additional hypotheses of \cite[Theorem 1.3]{AB-2013}.
Our first result is the relative logarithmic polynomiality theorem, and we give a brief version as follows.

\begin{theorem}[{Theorem~\ref{thm:log-poly-fix}}]
\label{thm:intro-log-polynomiality}
Fix $\delta,m\geq0$ and $p_r,p_l\geq1$. For any sufficiently large and spread out 
$\cf,\df$ and any $\eta,\mu$ with $|\eta|+|\mu|=d_b$ and $l(\mu)\geq\delta$, the combined logarithmic invariant $N_{g,\beta_\Delta,n}^{X_\cf|\partial X_\cf,m}(\eta,\mu)$ is a product of contacts, a factorial factor, and a universal polynomial,
where $g=|\Delta^\circ\cap\zb^2|-\delta+m$ and $n=2M+d_t+l(\mu)+|\Delta^\circ\cap\zb^2|-\delta-1$. Moreover,
this polynomial has length degree $2(\delta+m)$. When $\cf$ also varies in the stated range, the polynomial can be chosen with direction degree $\delta+m$ and joint degree $3(\delta+m)$.
\end{theorem}

If $(\eta,\mu)=(\emptyset,1^{d_b})$, we use $N_{g,\beta_\Delta,n}^{X_\cf|\partial X_\cf,m}$ to denote $N_{g,\beta_\Delta,n}^{X_\cf|\partial X_\cf,m}(\eta,\mu)$.
\begin{corollary}[Corollary~\ref{thm:log-simple-fix}]
\label{cor:intro-log-simple}
Under the corresponding hypotheses of Theorem~\ref{thm:intro-log-polynomiality}, 
the combined logarithmic invariant $N_{g,\beta_\Delta,n}^{X_\cf|\partial X_\cf,m}$ is the polynomial $\ql_{\delta,m}^\cf(\df)$, or
$\ql_{\delta,m}(\cf,\df)$ when directions vary.
\end{corollary}

The construction of the combined invariants $N_{g,\beta_\Delta,n}^{X_\cf|\partial X_\cf,m}(\eta,\mu)$ is motivated by the relation between the refined counts of floor diagrams and the refined counts of connected floor diagrams. Indeed, the formula expressing the combined invariants $N_{g,\beta_\Delta,n}^{X_\cf|\partial X_\cf,m}(\eta,\mu)$ in terms of connected invariants is the same as the formula expressing the refined counts of floor diagrams in terms of connected floor diagrams. We use the combinatorial tools of long edge graphs and the logarithm of generating functions in
\cite{BCK-2014,liu2016,liu-osserman2018} to obtain the refined relative polynomial for the refined counts of floor diagrams of degree $\Delta(\cf,\df)$, generalizing the results in \cite[Theorem 1]{Block-2012} and \cite[Theorem 7.8]{bg2016}. Bousseau's correspondence theorem then expresses the logarithmic series in terms of the refined relative polynomial after the change of variables $q=e^{iu}$.
Extracting the coefficient of $u^{2m}$ gives the required polynomial.

The surface $X_\cf$ is smooth precisely when adjacent primitive rays of the normal fan span $\zb^2$.
At internal corners, this is equivalent to
$$
c_r^i-c_r^{i+1}=1,\qquad c_l^{j+1}-c_l^j=1.
$$
If the top edge is absent, one must also require $c_r^1-c_l^1=1$; if the bottom edge is absent, one must require $c_l^{p_l}-c_r^{p_r}=1$.

When the normal fan of the polygon $\Delta(\cf,\df)$ is smooth, we use $N_{g,\beta_\Delta,n}^{\cf,m}$ to denote the combined Gromov--Witten invariants (see equation \eqref{eq:comb-gw}).
The following is a concise version of our second result.
\begin{theorem}[Theorem~\ref{thm:ordinary-poly-fixed}]
\label{thm:intro-ordinary-polynomiality}
Fix $\cf$ with smooth normal fan and $\delta,m\geq0$. For $d_b>0$ and a sufficiently large and spread out $h$-transverse polygon $\Delta(\cf,\df)$, the combined ordinary invariant satisfies
$$
N^{\cf,m}_{g(\df),\beta(\df),n(\df)}
=\ql_{\delta,m}^\cf(\df).
$$
\end{theorem}

The main geometric ingredient of the proof of Theorem \ref{thm:intro-ordinary-polynomiality} is the comparison of ordinary and full-boundary logarithmic invariants (see Proposition \ref{prop:poly-absolute-log}). With this comparison proposition, we obtain the equality of combined logarithmic invariants and the combined ordinary invariants. Then, Theorem \ref{thm:intro-ordinary-polynomiality} follows from Corollary \ref{cor:intro-log-simple}.
Note that the projective plane is not included in Theorem \ref{thm:intro-ordinary-polynomiality}.
One can obtain a similar result for $\pb^2$, by using the polynomiality of refined Severi degrees of $\pb^2$ in \cite{bg2016} and the comparison in Proposition \ref{prop:poly-absolute-log}.

The combined polynomial also determines connected invariants when the mixed intersections rule out
decompositions.
\begin{corollary}[Corollary~\ref{cor:ordinary-connected-asymptotics}]
\label{cor:intro-connected-asymptotics}
Fix $\cf$ with smooth normal fan and $\delta,m\geq0$. If every Minkowski decomposition of
$\Delta$ into polygons $\Delta_1,\ldots,\Delta_t$, where $t>1$, satisfies
$\delta<\sum_{i<j}\ml(\Delta_i,\Delta_j)$,
then, for $d_b>0$ and a sufficiently large and spread out $h$-transverse polygon $\Delta(\cf,\df)$, the combined invariant equals
the connected invariant $N^\cf_{g(\df),\beta(\df),n(\df)}(m)$. Along any fixed dilation $k\Delta$
that eventually satisfies that length range, we have the asymptotics:
$$
N^\cf_{g(k\df),\beta(k\df),n(k\df)}(m)
\sim\frac{(-1)^m3^\delta}{24^m m!\delta!}
\area(\Delta)^{\delta+m}k^{2(\delta+m)}
\qquad(k\longrightarrow\infty).
$$
\end{corollary}

This asymptotic concerns fixed cogenus and fixed Hodge degree, with genus growing quadratically
under dilation. It is therefore different from the fixed-genus asymptotics of plane Gromov--Witten
invariants. Tian--Wei obtained large-degree expansions in genera zero and one \cite{tian-wei2018},
and Guo--Tian--Xu established complete expansions at every fixed genus in their recent preprint
\cite{guo-tian-xu2026}. 
Our asymptotic follows from an exact polynomial valid for sufficiently large degree.

Our final result concerns the integrality, after the change of variables $q=e^{iu}$, of the coefficients of the series
$P^{X_\cf}_{g',\beta_\Delta}(u)$ in equation \eqref{eq:ding-hu-series}.
\begin{theorem}[Theorem~\ref{thm:ding-hu-bps}]
\label{thm:intro-bps}
Let $\Delta$ be $h$-transverse with smooth normal fan and $d_b>0$. Let
$0\leq\delta<|\Delta^\circ\cap\zb^2|$, and take $g'=|\Delta^\circ\cap\zb^2|-\delta$. If every boundary edge of $\Delta(\cf,\df)$ has lattice length at least $\delta$, then the series
$P^{X_\cf}_{g',\beta_\Delta}(u)$ in equation \eqref{eq:ding-hu-series} is an integral Laurent
polynomial in $q$, after the change of variables $q=e^{iu}$. It agrees with the connected refined tropical count of degree $\Delta$ and
genus $g'$ with simple free contacts.
\end{theorem}
When $g'=0$, Arg\"uz--Bousseau
\cite{arguz-bousseau-2025} showed that the series $P^{X}_{g',\beta_\Delta}(u)$ is an integral Laurent
polynomial in $q$ for any smooth projective surface $X$, after the change of variables $q=e^{iu}$.
When $g'>0$, the corresponding identity for toric del Pezzo surfaces and Hirzebruch surfaces is proved in \cite{dh2026} for nef classes, without our additional bound $d_\rho\geq\delta$. Theorem~\ref{thm:intro-bps} also applies to smooth $h$-transverse toric blowups outside those families.

\subsection*{Relation with previous results}\label{subsec:intro-comparison}

The starting point for our polynomiality results is the node polynomial results of Ardila--Block and
Block--G\"ottsche \cite{AB-2013,bg2016}. Ardila--Block obtain universal polynomials in polygon data
for unrefined Severi degrees, and Block--G\"ottsche establish refined tropical and floor diagram
formulas and polynomiality results. 
We generalize the relative node polynomial of (refined) floor diagrams 
in \cite[Theorem 1]{Block-2012} and \cite[Theorem 7.8]{bg2016} to $h$-transverse polygons.
Then we apply Bousseau's tropical refined correspondence theorem to transfer the polynomial
structures of refined floor diagrams to logarithmic Gromov--Witten invariants.
At $m=0$ and with simple free contacts, the resulting polynomial recovers the unrefined
node polynomial in the common numerical range. 
The passage to ordinary invariants on smooth toric surfaces requires the geometric comparison proved in
Section \ref{sec:poly-ordinary-comparison}; it is not a consequence of the combinatorial
polynomiality theorem alone. Here universality refers to the specified polygon variables and fixed
combinatorial data, rather than a claim that all the formulas depend only on the four topological
numbers in the G\"ottsche conjecture.

The variables and the genus regime also distinguish these results from other polynomiality theorems.
In the double Gromov--Witten and boundary-contact problems of
\cite{ardila-brugalle2017,corey-markwig-ranganathan2025,blomme23-compos}, contact orders vary and
chamber structures arise. In \cite{brugalle-jaramillo2022,mevel2026,blomme-mevel2025}, one fixes the
genus and studies coefficients at a fixed distance from the highest power of the refinement
variable. Our polynomiality results are obtained by fixing the cogenus and the degree of the Lambda classes. 
Root-order polynomiality \cite{fwy-2020,ty-2020,tseng-you2022} and tautological
descriptions through double ramification cycles \cite{jppz2017,ranganathan-kumaran2024} provide
other forms of structure, but do not directly supply the degree polynomial in
Theorem~\ref{thm:ordinary-poly-fixed}. 

\subsection*{Organization of the paper}\label{subsec:intro-organization}

Section~\ref{sec:diagrams} recalls relative floor diagrams and refined tropical counts 
and introduces the combined invariants, including their connected and disconnected 
correspondences and the logarithmic correspondence. 
Section~\ref{sec:log-polynomiality} proves polynomiality for relative floor diagrams. 
We extract the logarithmic Gromov--Witten polynomials in Section~\ref{sec:ordinary-polynomiality}, 
where we also establish the ordinary/logarithmic comparison, 
derive ordinary polynomiality and the Laurent-polynomial result, 
and conclude with the connected asymptotic corollary.

\subsection*{AI disclosure}
AI tools assisted with the combinatorial computations involved in establishing the polynomiality of the refined floor diagrams at an early stage, and with the final proofreading of the manuscript.

\section{Floor diagrams, refined relative Severi degrees and correspondence theorems}\label{sec:diagrams}

In this section, we first briefly review marked floor diagrams and the refined relative Severi degrees, then we introduce the combined invariants.

\subsection{Floor diagrams}\label{sec:floor-diagrams}

Floor diagrams were introduced by Brugall\'e and Mikhalkin in \cite{bm2007,bm2008} and then generalized to the relative case by Fomin and Mikhalkin in \cite{FM-2010}.
We use relative floor diagrams, and our notation follows closely \cite{AB-2013,Block-2012,bg2016,FM-2010}. 
We use partitions to represent the weights of the tangency conditions as in \cite{FM-2010}.

\begin{definition}[{\cite[Definition 2.2]{AB-2013}}]\label{def:Delta-floor diagram}
    Let $\Delta(\cf,\df)$ be an $h$-transverse polygon with height $M$. A \textit{$\Delta$-floor diagram} $\dl$ consists of the following data:
    \begin{enumerate}
        \item a directed multigraph on the vertex set $\{1,\ldots,M\}$, with positive integral edge weights $w(e)$ and edges directed $i\to j$ for $i<j$;
        \item two permutations $\lf=(l_1,\ldots,l_M)$ and $\rf=(r_1,\ldots,r_M)$ of the multisets $d_l(\Delta(\cf,\df))$ and $d_r(\Delta(\cf,\df))$ of the left and right directions of $\Delta(\cf,\df)$, respectively, and a sequence $\sff=(s_1,\ldots,s_M)$ of non-negative integers such that $s_1+\cdots+s_M=d_t$;

        \noindent subject to the divergence condition: for each vertex $j$, we have
        $$
        \di(j):=\sum_{\substack{\text{edges }e\\j\xrightarrow{e}k}} w(e)-\sum_{\substack{\text{edges }e\\i\xrightarrow{e}j}} w(e)\leq r_j-l_j+s_j,
        $$
        where $w(e)$ is the weight of $e$.
    \end{enumerate}
\end{definition}

We call the $\Delta$-floor diagram $\dl$ an $(\lf,\rf,\sff)$-floor diagram and call $\Delta$ the \textit{degree} of the floor diagram.

\begin{example}\label{exap:floor-diagram-weighted}
    Figure \ref{fig:htrans_normalfan} shows an $h$-transverse polygon and its normal fan. Its data are
    $$
    d_r(\Delta)=\{1,1,1,0\},\qquad d_l(\Delta)=\{0,0,0,0\},\qquad d_b=4,\quad d_t=1.
    $$
    \begin{figure}[htbp]
  \centering
  \begin{minipage}{0.48\textwidth}
  \centering
  \begin{tikzpicture}[scale=0.6, every node/.style={font=\small,inner sep=0.5pt}]
    \coordinate (P0)  at (0.00, 0.00);
    \coordinate (P1)  at (1.00, 0.00);
    \coordinate (P2)  at (2.00, 0.00);
    \coordinate (P3)  at (3.00, 0.00);
    \coordinate (P4)  at (4.00, 0.00);
    \coordinate (P5)  at (4.00, 1.00);
    \coordinate (P6)  at (3.00, 2.00);
    \coordinate (P7)  at (2.00, 3.00);
    \coordinate (P8)  at (1.00, 4.00);
    \coordinate (P9)  at (0.00, 4.00);
    \coordinate (P10) at (0.00, 3.00);
    \coordinate (P11)  at (0.00, 2.00);
    \coordinate (P12)  at (0.00, 1.00);
    \draw[thick] (P0)--(P1)--(P2)--(P3)--(P4)--(P5)--(P6)--(P7)--(P8)--(P9)--(P10)--(P11)--(P12)--cycle;
    \foreach \v in {P0,P1,P2,P3,P4,P5,P6,P7,P8,P9,P10,P11,P12}{
      \fill (\v) circle (2.2pt);
    }
  \end{tikzpicture}
  \end{minipage}
  \hfill
  \begin{minipage}{0.48\textwidth}
  \centering
  \begin{tikzpicture}[scale=0.8]
    \fill (0,0) circle (1.5pt);
    \draw (0,0) -- (2.0, 0.0);
    \draw (0,0) -- (1.4, 1.4);
    \draw (0,0) -- (0.0, 2.0);
    \draw (0,0) -- (-2.0, 0.0);
    \draw (0,0) -- (0.0, -2.0);
  \end{tikzpicture}
  \end{minipage}
  \caption{An $h$-transverse polygon and its normal fan.}
  \label{fig:htrans_normalfan}
\end{figure}
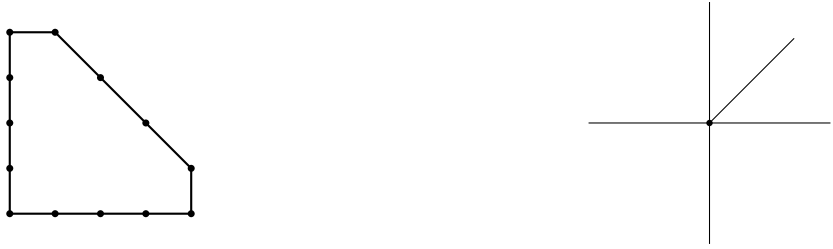
For the polygon in Figure \ref{fig:htrans_normalfan}, take
$\lf=(0,0,0,0)$, $\rf=(1,1,0,1)$, $\sff=(0,1,0,0)$.
An $(\lf,\rf,\sff)$-floor diagram is depicted in Figure \ref{fig:floor-diagram}. It has edges $1\to2$, $2\to3$, $3\to4$, and $2\to4$, with weights $1,2,1,1$, respectively.
Its divergences are $(1,2,-1,-2)$. 
Unlabelled edges in the figure have weight one.
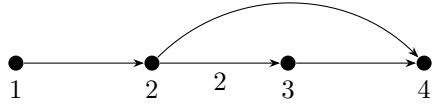
\begin{figure}[htbp]
\centering
\begin{tikzpicture}[x=1.8cm,y=0.7cm,
floor/.style={circle,fill,inner sep=2pt},
edge/.style={-{Stealth[length=4pt]}}]
\foreach \i in {1,...,4} {\node[floor,label=below:{$\i$}] (v\i) at (\i,0) {};}
\draw[edge] (v1)--(v2);
\draw[edge] (v2)--node[below] {$2$}(v3);
\draw[edge] (v3)--(v4);
\draw[edge] (v2) to[out=45,in=135] (v4);
\end{tikzpicture}
\caption{A $\Delta$-floor diagram.}
\label{fig:floor-diagram}
\end{figure}
\end{example}

A floor diagram is \textit{connected} if its underlying graph is connected. The \textit{genus} $g_\dl$ of a connected floor diagram $\dl$ is the genus of the underlying graph, \textit{i.e.}, the first Betti number of the underlying graph. The \textit{cogenus} $\delta(\dl)$ of a connected floor diagram $\dl$ is equal to $|\Delta^\circ\cap\zb^2|-g_\dl$, where $\Delta^\circ$ is the interior of $\Delta$. We denote by $\fdf_c(\Delta,g)$ the set of connected $\Delta$-floor diagrams of genus $g$.

When $\dl$ is not connected, there exist finitely many $h$-transverse lattice polygons $\Delta_1$, $\Delta_2$, $\ldots,\Delta_t$,
possibly including primitive positive length lattice segments, with Minkowski sum $\Delta_1+\Delta_2+\cdots+\Delta_t=\Delta$, up to lattice translation, such that the polygons $\Delta_i$ are degrees of the connected components of $\dl$.
Note that $t$ is not fixed.
Let $g_1,\ldots,g_t$ and $\delta_1,\ldots,\delta_t$ be the genera and cogenera of the connected components. The genus of $\dl$ is $g_\dl=\sum_ig_i+1-t$, and the cogenus is defined as
$$
\delta_\dl=\sum_i\delta_i+\sum_{i<j}\ml(\Delta_i,\Delta_j),
$$
where $\ml(\Delta_i,\Delta_j):=\frac{1}{2}(\area(\Delta_i+\Delta_j)-\area(\Delta_i)-\area(\Delta_j))$ is the mixed area of $\Delta_i$ and $\Delta_j$. 
Note that the normalized area $\area(\Delta_i)$ is given by twice the Euclidean area of $\Delta_i$ in $\rb^2$.
From Bernstein's theorem \cite{bernstein-1975}, one obtains $|\Delta^\circ\cap\zb^2|-g_\dl=\delta_\dl$ for a disconnected floor diagram $\dl$. 
Let $\fdf(\Delta,g)$ denote the set of $\Delta$-floor diagrams of genus $g$.

For $w\geq1$, the quantum integer is
$$
[w]_q=\frac{q^{w/2}-q^{-w/2}}{q^{1/2}-q^{-1/2}}
=q^{(w-1)/2}+q^{(w-3)/2}+\cdots+q^{-(w-1)/2};
$$
we set $[0]_q=0$ and $[w]_1=w$.

\begin{definition}[{\cite[Definition 5.2]{bg2016}}]\label{def:mult_q}
The \textit{refined multiplicity} $\mult_q(\dl)$ of a floor diagram $\dl$ is defined as
$$
\mult_q(\dl)=\prod_{\text{edges }e}([w(e)]_q)^2.
$$
\end{definition}

Now we introduce the definition of relative markings of floor diagrams. Our notation is a modification of that in \cite{Block-2012,bg2016,FM-2010}.
In accordance with the notation for relative Gromov–Witten invariants, we define relative markings via partitions as in \cite{FM-2010}.

\begin{definition}[{\cite[Definition 4.5]{bg16}}]\label{def:relative-marking}
    Let $\dl$ be a $\Delta$-floor diagram, where $\Delta=\Delta(\cf,\df)$ is an $h$-transverse polygon. Let $\eta,\mu$ be two partitions with $|\eta|+|\mu|=d_b$. An \textit{$(\eta,\mu)$-marking} of $\dl$ is defined by the following four step process.
    \begin{itemize}
        \item[{\bf Step 1:}] Fix a pair of collections of partitions $(\{\eta^i\},\{\mu^i\})_{i=1}^M$, where $M$ is the height of the polygon $\Delta$, such that:
        \begin{enumerate}
            \item $\eta=\cup_i\eta^i$ and $\mu=\cup_i\mu^i$;
            \item for any vertex $j$ of $\dl$, we have $|\eta^j|+|\mu^j|=r_j-l_j+s_j-\di(j)$.
        \end{enumerate}
        The second condition requires that the relative degree $|\eta^j|+|\mu^j|$ 
        be compatible with the divergence at the vertex $j$. 
        Such a pair $(\{\eta^i\},\{\mu^i\})_{i=1}^M$ is called \textit{compatible} with $\dl$ and $(\eta,\mu)$.
        \item[{\bf Step 2:}] For each vertex $j$ of $\dl$, create $s_j$ new indistinguishable vertices and connect them to $j$ with a new edge of weight one directed towards $j$. For each vertex $j$ of $\dl$, create $l(\eta^j)$ (resp. $l(\mu^j)$) indistinguishable vertices, called the \textit{$\eta$-vertices} (resp. \textit{$\mu$-vertices}), and connect them to $j$ with a new edge of weight $\eta^j_k$ (resp. $\mu_k^j$) directed away from $j$, with $1\leq k\leq l(\eta^j)$ (resp. $1\leq k\leq l(\mu^j)$).
        This makes the divergence of vertex $j$ in the extended graph equal to $r_j-l_j$ for $1\leq j\leq M$. See Figure \ref{fig:marking-step2} for the illustration in the case of Example \ref{exap:floor-diagram-weighted}.
        \item[{\bf Step 3:}] Subdivide each edge of the original floor diagram $\dl$ into two directed edges by introducing a new vertex for each edge. The new edges inherit their weights and orientations (see Figure \ref{fig:marking-step3}). We use $\wt\dl$ to denote the resulting graph.
        \item[{\bf Step 4:}] Extend the linear ordering on the vertices of $\dl$ to the set of vertices of $\wt\dl$ such that each edge in $\wt\dl$ is directed from a smaller vertex to a larger vertex. We also require that the $\eta$-vertices are largest among all vertices, and for every pair of $\eta$-vertices $i'>i$, the weight of the $i'$-adjacent edge is larger than or equal to the weight of the $i$-adjacent edge. See Figure \ref{fig:marking-step4} for an example.
    \end{itemize}
    The extended graph $\wt\dl$ together with the linear order on its vertices is called an \textit{$(\eta,\mu)$-marked floor diagram} or an \textit{$(\eta,\mu)$-marking} of the floor diagram $\dl$.
\end{definition}

\begin{example}
For the $\Delta$-floor diagram in Example \ref{exap:floor-diagram-weighted}, choose $\eta^3=(1)$ and $\mu^4=(1,2)$, with all other profiles empty.
Thus $\eta=(1)$ and $\mu=(1,2)$.
Figures \ref{fig:marking-step2}--\ref{fig:marking-step4} illustrate the marking construction.
\begin{figure}[htbp]
\centering
\begin{tikzpicture}[x=1.3cm,y=0.6cm,
floor/.style={circle,fill,inner sep=2pt},
added vertex/.style={circle,draw,fill=white,inner sep=1.8pt},
edge/.style={-{Stealth[length=4pt]}}]
\foreach \i in {1,...,4} {\node[floor,label=below:{$\i$}] (v\i) at (1.2*\i,0) {};}
\node[added vertex,label=below:{$t$}] (t) at (1.6,-1) {};
\node[added vertex,label=right:{$a$}] (a) at (5.8,1.5) {};
\node[added vertex,label=right:{$b$}] (b) at (5.8,0.4) {};
\node[added vertex,label=right:{$c$}] (c) at (5.8,-1) {};
\draw[edge] (v1)--(v2);
\draw[edge] (v2)--node[below] {$2$}(v3);
\draw[edge] (v3)--(v4);
\draw[edge] (v2) to[out=65,in=135] (v4);
\draw[edge] (t)--(v2);
\draw[edge] (v3)--(a);
\draw[edge] (v4)--(b);
\draw[edge] (v4)--node[below] {$2$}(c);
\end{tikzpicture}
\caption{Steps 1--2: attaching the top and relative contacts.}
\label{fig:marking-step2}
\end{figure}
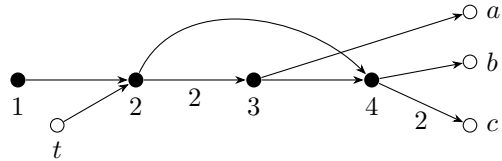

\begin{figure}[htbp]
\centering
\begin{tikzpicture}[x=1.3cm,y=0.6cm,
floor/.style={circle,fill,inner sep=2pt},
added vertex/.style={circle,draw,fill=white,inner sep=1.8pt},
edge/.style={-{Stealth[length=4pt]}}]
\foreach \i in {1,...,4} {\node[floor,label=below:{$\i$}] (v\i) at (1.2*\i,0) {};}
\node[added vertex,label=below:{$t$}] (t) at (1.6,-1) {};
\node[added vertex,label=right:{$a$}] (a) at (5.8,1.5) {};
\node[added vertex,label=right:{$b$}] (b) at (5.8,0.4) {};
\node[added vertex,label=right:{$c$}] (c) at (5.8,-1) {};
\node[added vertex] (x12) at (1.8,0) {};
\node[added vertex] (x23) at (3,0) {};
\node[added vertex] (x34) at (4.2,0) {};
\node[added vertex] (x24) at (3.6,1.35) {};
\draw[edge] (v1)--(x12); \draw[edge] (x12)--(v2);
\draw[edge] (v2)--node[below] {$2$}(x23); \draw[edge] (x23)--node[below] {$2$}(v3);
\draw[edge] (v3)--(x34); \draw[edge] (x34)--(v4);
\draw[edge] (v2) to[out=65,in=180] (x24); \draw[edge] (x24) to[out=0,in=135] (v4);
\draw[edge] (t)--(v2); \draw[edge] (v3)--(a);
\draw[edge] (v4)--(b); \draw[edge] (v4)--node[below] {$2$}(c);
\end{tikzpicture}
\caption{Step 3: subdividing all four original edges.}
\label{fig:marking-step3}
\end{figure}
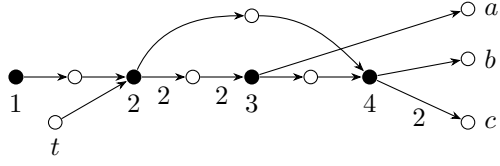

\begin{figure}[htbp]
\centering
\begin{tikzpicture}[x=0.75cm,y=0.65cm,
floor/.style={circle,fill,inner sep=2pt},
added vertex/.style={circle,draw,fill=white,inner sep=1.6pt},
edge/.style={-{Stealth[length=3pt]}}]
\foreach \i in {1,...,12} {\node[added vertex,label=below:{\scriptsize $\i$}] (p\i) at (\i,0) {};}
\foreach \i in {1,4,7,9} {\node[floor] at (p\i) {};}
\draw[edge] (p1)--(p2);
\draw[edge] (p2) to[out=55,in=125] (p4);
\draw[edge] (p3)--(p4);
\draw[edge] (p4)--node[above] {\scriptsize $2$}(p5);
\draw[edge] (p5) to[out=50,in=130] node[above] {\scriptsize $2$}(p7);
\draw[edge] (p4) to[out=80,in=100] (p6);
\draw[edge] (p6) to[out=75,in=105] (p9);
\draw[edge] (p7)--(p8); \draw[edge] (p8)--(p9);
\draw[edge] (p9)--(p10);
\draw[edge] (p9) to[out=55,in=125] node[above] {\scriptsize $2$}(p11);
\draw[edge] (p7) to[out=-50,in=-130] (p12);
\end{tikzpicture}
\caption{Step 4: a compatible total order. The original floors occupy positions $1,4,7,9$; the fixed contact is the last vertex, at position $12$.}
\label{fig:marking-step4}
\end{figure}
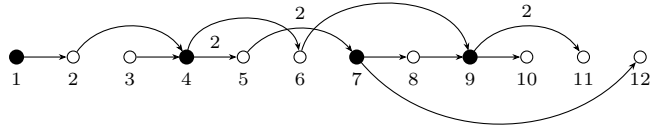
\end{example}

Two $(\eta,\mu)$-markings $\wt\dl_1,\wt\dl_2$ of a floor diagram $\dl$ are \textit{equivalent} if there exists an isomorphism of weighted directed graphs $\varphi:\wt\dl_1\to\wt\dl_2$ fixing the vertices of $\dl$ and preserving the linear order and the types of the added vertices.
The \textit{number of markings $\nu_{\eta,\mu}(\dl)$} is the number of $(\eta,\mu)$-markings of $\dl$, up to equivalence.

The \textit{refined $(\eta,\mu)$-multiplicity} of a $\Delta$-floor diagram $\dl$ is defined as
$$
\mult_q^{\eta,\mu}(\dl)=\prod_{i=1}^{l(\mu)}[\mu_i]_q\mult_q(\dl).
$$

\begin{definition}\label{def:refined-count}
    The \textit{count with refined $(\eta,\mu)$-multiplicity} of $\Delta$-floor diagrams of genus $g$ is
    $$
        N_{\floor}^{\Delta,g}(\eta,\mu;q):=\sum_{\dl\in\fdf(\Delta,g)}\mult_q^{\eta,\mu}(\dl)\cdot\nu_{\eta,\mu}(\dl).
    $$
    The \textit{count with refined $(\eta,\mu)$-multiplicity} of connected $\Delta$-floor diagrams of genus $g$ is defined as
    $$
    N_{c,\floor}^{\Delta,g}(\eta,\mu;q):=\sum_{\dl\in\fdf_c(\Delta,g)}\mult_q^{\eta,\mu}(\dl)\cdot\nu_{\eta,\mu}(\dl).
    $$
\end{definition}

Let $\dl$ be a $\Delta$-floor diagram with $M$ vertices and $E$ edges, and let $\dl_1,\ldots,\dl_t$ be its connected components whose degrees are $\Delta_1,\ldots,\Delta_t$, respectively. 
The polygons $\Delta_i$ are determined by $(\lf_i,\rf_i,\sff_i)$,
where $\lf_i,\rf_i,\sff_i$ are the sequences determined by $\dl_i$.
Note that $\Delta$, $\Delta_1,\ldots,\Delta_t$ are $h$-transverse lattice polygons.
Denote by $M_i$, $E_i$, and $g_i$ the number of vertices, the number of edges, and the genus of $\dl_i$, respectively.
Since $g_i=E_i-M_i+1$, we have $g=\sum_{i=1}^{t}g_i+1-t=E-M+1$.
The marking construction adds $E$ subdivision vertices, $d_t$ top vertices, and $l(\mu)$ free bottom vertices to $\dl$.
Thus the number of nonfixed marking positions is
$$
n:=M+E+d_t+l(\mu)=2M+d_t+l(\mu)+g-1.
$$
Label the nonfixed marking positions by $\Lambda=\{1,\ldots,n\}$.
Choose a set $\varPi$ of $l(\eta)$ distinct points equipped with a weight map $\omega:\varPi\to\zb_{>0}$ whose multiset of values is $\eta$.
For $I\subseteq\varPi$, denote the partition of weights on $I$ by $\eta(I)$.
Restricting the total order of a marked diagram to each connected component gives a nonempty collection $\Lambda_i$ of nonfixed positions. 
Its fixed contact positions form a subset $\varPi_i$, with profile $\eta_i=\eta(\varPi_i)$, and its free bottom ends determine a profile $\mu_i$.
The component data recover the total degree and contact profiles by
$\Delta=\Delta_1+\cdots+\Delta_t, ~
\eta=\bigsqcup_{i=1}^{t}\eta_i,~
\mu=\bigsqcup_{i=1}^{t}\mu_i$.
The same marking count on $\dl_i$ gives
$|\Lambda_i|=2M_i+d_{t,i}+l(\mu_i)+g_i-1$.
Conversely, connected marked diagrams with such data determine a marked diagram of their disjoint union.
The refined multiplicity separates over the components as follows:
$$
\mult_q^{\eta,\mu}(\dl)
=\prod_{i=1}^{t}\mult_q^{\eta_i,\mu_i}(\dl_i).
$$

\begin{remark}\label{rem:floor-split}
With the preceding notation,
$$
\begin{aligned}
N_{\floor}^{\Delta,g}(\eta,\mu;q)
={}&\sum_{\Lambda=\bigsqcup_{i=1}^{t}\Lambda_i}
    \ \sum_{\varPi=\bigsqcup_{i=1}^{t}\varPi_i}
    \ \sum_{(\Delta_i,g_i,\mu_i)}
    \prod_{i=1}^{t}
    N_{c,\floor}^{\Delta_i,g_i}(\eta_i,\mu_i;q),
\end{aligned}
$$
where $\eta_i=\eta(\varPi_i)$.
The first sum is over unordered partitions of $\Lambda$ into nonempty collections, with $t$ varying. 
The second sum is taken over disjoint unions of $\varPi$, allowing $\varPi_i=\varnothing$.
The last sum is over component degrees, genera $g_i\geq0$, and free profiles $\mu_i$ satisfying
$$
\begin{aligned}
\Delta&=\Delta_1+\cdots+\Delta_t,~
\mu=\bigsqcup_{i=1}^{t}\mu_i,~
|\eta_i|+|\mu_i|=d_{b,i},~
g-1=\sum_{i=1}^{t}(g_i-1)\\
|\Lambda_i|&=2M_i+d_{t,i}+l(\mu_i)+g_i-1, \text{ for }
1\leq i\leq t.
\end{aligned}
$$
Here $M_i>0$ and $d_{b,i},d_{t,i}$ are the height and the bottom and top lengths of $\Delta_i$.
\end{remark}

\subsection{Refined tropical relative Severi degrees}

We use the definitions of nodal or simple parametrized tropical curves from \cite[Section 3]{bg2016} directly, and we refer the reader to that section for more details.
We use partitions to describe the contact orders.

Let $\Delta$ be a lattice polygon with a horizontal bottom edge of lattice length $d_b>0$.
Let $\eta,\mu$ be partitions satisfying $|\eta|+|\mu|=d_b$.
Fix an integer $g$ such that $|\Delta^\circ\cap\zb^2|-g\geq0$, and put
$n=|\partial\Delta\cap\zb^2|+g-1-d_b+l(\mu)$.
We assume $n\geq0$.
Choose a tropical boundary line $L_b$, parallel to the bottom edge of $\Delta$ and sufficiently far below.
Choose $n$ generic points $\Pi$ above $L_b$ and $l(\eta)$ distinct prescribed points
$\Pi_\eta=\{p_1,\ldots,p_{l(\eta)}\}\subset L_b$,
where $p_i$ is assigned contact order $\eta_i$.

\begin{definition}[{\cite[Definitions 7.1--7.2]{bg2016}}]
A tropical curve $C$ is \textit{$(\eta,\mu)$-relative} to $L_b$, with respect to $\Pi_\eta$, if its downward unbounded edges consist of the following:
\begin{enumerate}
    \item one edge of weight $\eta_i$ passing through $p_i$, for each $1\leq i\leq l(\eta)$;
    \item $l(\mu)$ edges whose weights form the partition $\mu$, with unspecified intersections with $L_b$.
\end{enumerate}
All other unbounded edges have weight one.
The contacts specified by $\eta$ are called \textit{fixed contacts}, and those specified by $\mu$ are called \textit{free contacts}.
The free contacts are not additionally labeled.
\end{definition}

Let $e_1,e_2$ be any two edges adjacent to a $3$-valent vertex $v$ in $C$ with weights $w(e_1), w(e_2)$, respectively. 
Let $u_{v,e_1},u_{v,e_2}$ be the outgoing primitive directions of $e_1,e_2$ from $v$.
For the inner vertex $v$, define
$m_C(v)=\bigl|\det(w(e_1)u_{v,e_1},w(e_2)u_{v,e_2})\bigr|.$
Note that the balancing condition implies that this definition is independent of the chosen pair.
The \textit{refined $(\eta,\mu)$-multiplicity} of such a simple tropical curve is
$$
\mult_q^{\eta,\mu}(C)
=\prod_{i=1}^{l(\eta)}\frac{1}{[\eta_i]_q}\mult_q(C).
$$
The multiplicity is defined by $\mult_q(C)=\prod_v[m_C(v)]_q$, where the product runs over its finite trivalent vertices.

\begin{definition}[{\cite[Definition 7.2]{bg2016}}]\label{def:trop-count}
For the chosen configurations $(\Pi,\Pi_\eta)$, 
let $\mathcal T(\Delta,g;\eta,\mu)$ denote the set of 
isomorphism classes of nodal parametrized tropical curves of degree $\Delta$ 
and genus $g$, passing through $\Pi$ and satisfying the prescribed relative conditions.
Let $\mathcal T_c(\Delta,g;\eta,\mu)$ be the subset consisting of irreducible curves.
If the configurations $(\Pi,\Pi_\eta)$ are tropically generic, these sets are finite and their curves are simple.
The \textit{tropical refined relative Severi degree} is
$$
N_{\mathrm{trop}}^{\Delta,g}(\eta,\mu;q)
:=\sum_{C\in\mathcal T(\Delta,g;\eta,\mu)}\mult_q^{\eta,\mu}(C).
$$
The \textit{irreducible tropical refined relative Severi degree} is
$$
N_{c,\mathrm{trop}}^{\Delta,g}(\eta,\mu;q)
:=\sum_{C\in\mathcal T_c(\Delta,g;\eta,\mu)}\mult_q^{\eta,\mu}(C).
$$
By \cite[Theorem 7.3]{bg2016}, these counts are independent of the generic configuration.
\end{definition}

For an $h$-transverse polygon $\Delta=\Delta(\cf,\df)$ of height $M$, we have $|\partial\Delta\cap\zb^2|=2M+d_t+d_b$.
Consequently, $n=2M+d_t+g-1+l(\mu)$,
which agrees with the number of nonfixed marking positions of a $\Delta$-floor diagram.

\begin{theorem}[{\cite[Proposition 5.9]{bm2008}, \cite[Theorem 4.6]{bg16}}]
\label{thm:floor-trop}
Let $\Delta=\Delta(\cf,\df)$ have height $M>0$ and $d_b>0$, and let $|\eta|+|\mu|=d_b$.
Let $0\leq g\leq |\Delta^\circ\cap\zb^2|$, and put $\delta=|\Delta^\circ\cap\zb^2|-g$ and $n=2M+d_t+g-1+l(\mu)\geq0$.
Then, we have
$$
N_{c,\floor}^{\Delta,g}(\eta,\mu;q)
=N_{c,\mathrm{trop}}^{\Delta,g}(\eta,\mu;q),
$$
and, if $\delta<M$,
$$
N_{\floor}^{\Delta,g}(\eta,\mu;q)
=N_{\mathrm{trop}}^{\Delta,g}(\eta,\mu;q).
$$
\end{theorem}

\begin{proof}
The first equality follows from the floor-tropical correspondence \cite[Proposition 5.9]{bm2008} and the multiplicity computation in \cite[Proposition 7.7]{bg2016}. 
The second equality is given by \cite[Theorem 4.6]{bg16}.

Note that we allow the occurrence of segments in the Minkowski sum, so we first exclude the existence of a height zero component.
Let $\Delta=\Delta_1+\cdots+\Delta_t$ be the Minkowski sum. 
Suppose that there is a component with height zero, say $\Delta_1$. 
The tropical curve of degree $\Delta_1$ that corresponds to an irreducible component of a tropical curve in $\mathcal T(\Delta,g;\eta,\mu)$ is a vertical line of weight one. 
Let $M_i$ be the height of $\Delta_i$ for $1<i\leq t$, then $M=\sum_{i=2}^tM_i$. 
Note that $\area(\Delta_1+\Delta_i)-\area(\Delta_i)=2M_i$, 
so $\ml(\Delta_1,\Delta_i)=M_i$.
Then, we have
$$
\delta
=\sum_{i=1}^{t}\delta_i
 +\sum_{j=2}^{t}\ml(\Delta_1,\Delta_j)
 +\sum_{2\leq i<j\leq t}
      \ml(\Delta_i,\Delta_j)
=\sum_{i=2}^tM_i+\sum_{i=1}^{t}\delta_i
 +\sum_{2\leq i<j\leq t}
      \ml(\Delta_i,\Delta_j)
\geq M.
$$
This contradicts the hypothesis $\delta<M$.
Hence, every component of $\Delta$ has positive height and admits a floor decomposition, including primitive segment components, whose count and multiplicity are one. 
Decompose $\Lambda,\varPi$ as in Remark \ref{rem:floor-split},
then the second equality follows from the equality of the connected invariants.
\end{proof}

\subsection{Gromov-Witten invariants of smooth toric surfaces}
\label{sec:corr-gw}

Let $X_\cf$ be a smooth projective toric surface from an $h$-transverse polygon $\Delta:=\Delta(\cf,\df)$, and let $\beta\in H_2(X_\cf;\zb)$ be a curve class.
Let $\overline M_{g,n}(X_\cf,\beta)$ be
the moduli space of stable maps of genus $g$ and degree $\beta$ to $X_\cf$ with $n$ marked points.
From \cite{beh1997,li-tian-1998,fp97}, the moduli space $\overline M_{g,n}(X_\cf,\beta)$
is a proper and separated Deligne-Mumford stack.
The moduli space $\overline M_{g,n}(X_\cf,\beta)$ has a virtual fundamental class $[\overline M_{g,n}(X_\cf,\beta)]^{vir}$ \cite{li-tian-1998,BF-1997}.
Its virtual dimension is
$$
\vdim \overline M_{g,n}(X_\cf,\beta)=c_1(X_\cf)\cdot\beta+g-1+n,
$$
where $c_1(X_\cf)$ is the first Chern class of $X_\cf$.

Denote by $\pi:\cl\to\overline M_{g,n}(X_\cf,\beta)$ the universal curve,
and let $\omega_\pi$ be the relative dualizing sheaf of $\cl$.
The rank $g$ vector bundle $\eb_\pi:=\pi_*\omega_\pi$ on $\overline M_{g,n}(X_\cf,\beta)$ is called the Hodge bundle.
Lambda classes are the Chern classes of the Hodge bundle $\eb_\pi$: 
$
\lambda_{m,\pi}:=c_m(\eb_\pi)\in H^{2m}(\overline M_{g,n}(X_\cf,\beta);\qb), m=0,1,\ldots,g.
$
Let $ev_i:\overline M_{g,n}(X_\cf,\beta)\to X_\cf$ be the evaluation map,
and let $pt\in H^{*}(X_\cf;\qb)$ be the class Poincar\'e dual to a point in $X_\cf$. We put
$$
N_{g,\beta,n}^{\cf}(m):=
\int_{[\overline M_{g,n}(X_\cf,\beta)]^{vir}}(-1)^m\lambda_m\cdot\prod_{i=1}^n ev_i^*(pt).
$$
Note that $N_{g,\beta,n}^{\cf}(m)=0$, unless $n=c_1(X_\cf)\cdot\beta+g-m-1$. When $m=g$, we use $N_{g,\beta,n}^{\cf}$ to denote $N_{g,\beta,n}^{\cf}(g)$.

We introduce combined Gromov-Witten invariants as follows.
Let $\Delta$ be an $h$-transverse lattice polygon with smooth normal fan.
Decompose $\Delta$ into a Minkowski sum $\Delta=\Delta_1+\cdots+\Delta_t$ of $h$-transverse lattice polygons $\Delta_1,\ldots,\Delta_t$.
Then the normal fan of $\Delta$ is the common refinement of the normal fans of $\Delta_1,\ldots,\Delta_t$.
Hence, each polygon $\Delta_i$ defines a curve class $\beta_i\in H_2(X_\cf;\zb)$ such that $\beta_\Delta=\beta_1+\cdots+\beta_t$, where $\beta_\Delta=\beta_\cf(\df)$ is the class of the divisor $L_\Delta$.
Fix a set $\Lambda$ of $n=c_1(X_\cf)\cdot\beta+g-m-1$ generic points in the surface $X_\cf$.
Let $\Lambda_1,\ldots,\Lambda_t$ be a partition of $\Lambda$ into some number of nonempty subsets, and let $n_i=|\Lambda_i|$ for $i=1,\ldots,t$.
Define the combined Gromov-Witten invariants as
\begin{equation}\label{eq:comb-gw}
N_{g,\beta_\Delta,n}^{\cf,m}:=\sum_{\Lambda=\cup\Lambda_i}\sum_{(\Delta_i,g_i,m_i)}\prod_{i}N_{g_i,\beta_{i},n_i}^{\cf}(m_i),
\end{equation}
where the first sum is taken over all partitions of $\Lambda$ into nonempty subsets,
and the second sum is over all triples $(\Delta_i,g_i,m_i)$ which satisfy
\begin{equation}\label{eq:comb-gw-cond}
\begin{aligned}
&\Delta=\Delta_1+\cdots+\Delta_t,
g=\sum_{i=1}^tg_i+1-t,~
m=\sum_{i=1}^tm_i,\\
n_i=|\Lambda_i|=&c_1(X_\cf)\cdot\beta_i+g_i-m_i-1, ~0\leq m_i\leq g_i,~\forall i\in\{1,\ldots,t\},\\
\end{aligned}
\end{equation}
Note that the relation $\beta_\Delta=\beta_1+\cdots+\beta_t$ of classes is included in the Minkowski sum $\Delta=\Delta_1+\cdots+\Delta_t$.

\subsection{Logarithmic Gromov-Witten invariants and refined tropical counts}
\label{sec:corr-log-gw}

Logarithmic geometry was used to handle the problem of non-proper intersections in the study of curve counting theory. The theory of logarithmic stable maps has been developed in \cite{AbraChen-2014,Chen-2014,gs-2013}. 
In the following, we recall the definition of logarithmic Gromov-Witten invariants \cite{gs-2013} and the tropical refined correspondence \cite{bou2019}.

Let $X_\cf$ be the normal projective toric surface corresponding to an $h$-transverse lattice polygon $\Delta=\Delta(\cf,\df)$.
The toric boundary of $X_\cf$ is $\partial X_\cf=\bigcup_\rho D_\rho$.
We also use $D_b$ to denote the toric divisor corresponding to the bottom edge of $\Delta(\cf,\df)$.
Let $\beta_\Delta$ be the curve class of the divisor $L_\Delta$, and put $d_\rho=\beta_\Delta\cdot D_\rho$.
We write $c_1(X_\cf)=-K_{X_\cf}=\sum_\rho D_\rho$ for the anticanonical divisor class,
so $c_1(X_\cf)\cdot\beta_\Delta=\sum_\rho d_\rho$.
Equip $X_\cf$ with the divisorial logarithmic structure determined by $\partial X_\cf$.

Let $\eta,\mu$ be two partitions such that $|\eta|+|\mu|=d_b>0$.
Along $D_b$ we impose the fixed contacts $\eta$ and the free contacts $\mu$; along every other $D_\rho$, we impose $d_\rho$ free simple contacts.
The total number of boundary marked points is
$$
\begin{aligned}
\ell
=\sum_{\rho\ne b}d_\rho+l(\eta)+l(\mu)
=|\partial\Delta\cap\zb^2|-d_b+l(\eta)+l(\mu)
 =c_1(X_\cf)\cdot\beta_\Delta-d_b+l(\eta)+l(\mu).
\end{aligned}
$$
Let $\overline M^{\log}_{g,n}
(X_\cf|\partial X_\cf,\beta_\Delta,\vec\eta,\vec\mu)$
be the moduli space of basic stable logarithmic maps of genus $g$ and  class $\beta_\Delta$ to $X_\cf$ with $n$ ordinary marked points and $\ell$ relative marked points contacting the boundary $\partial X_\cf$ as specified above.
The moduli space $\overline M^{\log}_{g,n}
(X_\cf|\partial X_\cf,\beta_\Delta,\vec\eta,\vec\mu)$ is a proper Deligne-Mumford stack with a virtual fundamental class \cite{gs-2013}.
Its virtual dimension is
$$
\vdim\overline M^{\log}_{g,n}
(X_\cf|\partial X_\cf,\beta_\Delta,\vec\eta,\vec\mu)
=g-1+n+\ell.
$$

Let $ev_i$ be the evaluation maps to $X_\cf$, and let $ev_j^{D_b}$ be the evaluation maps at the fixed boundary markings.
Let $pt$ be the class of a general point in the dense torus of $X_\cf$, and let $pt_{D_b}$ be the class of a general point  of $D_b$.
As in Section \ref{sec:corr-gw}, the universal curve of the moduli space $\overline M^{\log}_{g,n}
(X_\cf|\partial X_\cf,\beta_\Delta,\vec\eta,\vec\mu)$ defines a Hodge bundle $\eb$ of rank $g$, and we put $\lambda_m=c_m(\eb)$.
For $0\leq m\leq g$, the logarithmic Gromov-Witten invariant is defined as
$$
\begin{aligned}
N_{g,\beta_\Delta,n}^{X_\cf|\partial X_\cf}(m,\eta,\mu)
:=\frac{1}{|\aut(\mu)|\prod_{\rho\ne b}d_\rho!}
\int_{[\overline M^{\log}_{g,n}]^{vir}}
(-1)^m\lambda_m
\prod_{i=1}^n ev_i^*(pt)
\prod_{j=1}^{l(\eta)}(ev_j^{D_b})^*(pt_{D_b}),
\end{aligned}
$$
where $\overline M^{\log}_{g,n}:=\overline M^{\log}_{g,n}
(X_\cf|\partial X_\cf,\beta_\Delta,\vec\eta,\vec\mu)$.
The normalizing factor forgets the labels of the free contacts of equal order on each boundary divisor, with $|\aut(\mu)|=\prod_{w\geq1}m_w(\mu)!$.
Note that this invariant vanishes unless
$$
n=\ell-l(\eta)+g-m-1
 =c_1(X_\cf)\cdot\beta_\Delta-d_b+l(\mu)+g-m-1.
$$

\begin{theorem}[{\cite[Theorems 5--6]{bou2019}}]\label{thm:log-trop}
Let $X_\cf$ be the normal projective toric surface corresponding to an $h$-transverse lattice polygon $\Delta=\Delta(\cf,\df)$ with $d_b>0$.
Let $\eta,\mu$ be two partitions such that $|\eta|+|\mu|=d_b$.
Fix an integer $0\leq g'\leq|\Delta^\circ\cap\zb^2|$, and let
$n=c_1(X_\cf)\cdot\beta_\Delta-d_b+l(\mu)+g'-1\geq0$.
Then, we have the equality
$$
\begin{aligned}
\sum_{g\geq g'}&
N_{g,\beta_\Delta,n}^{X_\cf|\partial X_\cf}
(g-g',\eta,\mu)u^{2g-2+\ell}\\
&=
\left(\prod_{j=1}^{l(\eta)}\frac{[\eta_j]_q}{\eta_j}\right)
N_{c,\mathrm{trop}}^{\Delta,g'}(\eta,\mu;q)
\left((-i)(q^{1/2}-q^{-1/2})\right)^{2g'-2+\ell}
\end{aligned}
$$
of power series in $u$ with rational coefficients, where $q=e^{iu}$ and $\ell=|\partial\Delta\cap\zb^2|-d_b+l(\eta)+l(\mu)$.
\end{theorem}

We introduce a combined logarithmic invariant as follows.
Decompose the polygon $\Delta=\Delta(\cf,\df)$ into a finite Minkowski sum $\Delta=\Delta_1+\cdots+\Delta_t$. 
Each summand $\Delta_i$ defines a nef toric divisor on $X_\cf$, and we let $\beta_i$ be its associated curve class.
Then we have $\beta_\Delta=\beta_1+\cdots+\beta_t$.
Let $d_{\rho,i}=\beta_i\cdot D_\rho$. When $\eta_i=\mu_i=\varnothing$, we allow $d_{b,i}=0$.

The normal fan of a two-dimensional $\Delta_i$ may be coarser than that of $X_\cf$. We use its full {boundary} logarithmic invariant on $X_\cf$. 
Applying the connected correspondence on a common toric refinement gives the same invariant by logarithmic birational invariance \cite{aw2018}. 
For segment summands, we use the collinear balanced degree as in \cite[Section 2.1, footnote 4]{bou2019}, and the same refinement argument. 
Label the ordinary point conditions by $\Lambda=\{1,\ldots,n\}$.
Let $\varPi$ be a disjoint set of $l(\eta)$ fixed-contact labels, with
weight map $\omega:\varPi\to\zb_{>0}$ having multiset of values $\eta$.

\begin{definition}\label{def:comb-log-gw}
For $g\in\zb$ and $m,n\geq0$, define the combined logarithmic invariant by
\[
N_{g,\beta_\Delta,n}^{X_\cf|\partial X_\cf,m}(\eta,\mu)
=\sum_{\Lambda=\bigsqcup_{i=1}^{t}\Lambda_i}
    \ \sum_{\varPi=\bigsqcup_{i=1}^{t}\varPi_i}
\ \sum_{(\Delta_i,g_i,m_i,\mu_i)}
\prod_{i=1}^{t}
N_{g_i,\beta_i,n_i}^{X_\cf|\partial X_\cf}(m_i,\eta_i,\mu_i).
\]
The first sum is over partitions of $\Lambda$ into nonempty subsets, with
$t$ varying. The second sum is taken {over disjoint unions of} $\varPi$, allowing $\varPi_i=\emptyset$. Let $n_i=|\Lambda_i|$, {$\eta_i=\eta(\varPi_i)$, which} is the partition of weights on $\varPi_i$.
The last sum is over nonzero component polygon degrees, and integers $g_i\geq0$, $0\leq m_i\leq g_i$,
together with free-contact partitions $\mu_i$, such that
\[
\begin{gathered}
\Delta=\Delta_1+\cdots+\Delta_t,\qquad
\mu=\bigsqcup_{i=1}^{t}\mu_i,\qquad
g-1=\sum_{i=1}^{t}(g_i-1),\qquad m=\sum_{i=1}^{t}m_i,\\
|\eta_i|+|\mu_i|=\beta_i\cdot D_b,\qquad
n_i=c_1(X_\cf)\cdot\beta_i-\beta_i\cdot D_b
+l(\mu_i)+g_i-m_i-1.
\end{gathered}
\]
For $(\eta,\mu)=(\emptyset,1^{d_b})$, we use $N_{g,\beta_\Delta,n}^{X_\cf|\partial X_\cf,m}$ {as an abbreviation for} $N_{g,\beta_\Delta,n}^{X_\cf|\partial X_\cf,m}(\eta,\mu)$.
\end{definition}

The argument in the proof of Theorem \ref{thm:floor-trop} shows that, 
for $\delta=|\Delta^\circ\cap\zb^2|-g'<M$, 
every contributing tropical component has positive height and at least one ordinary point condition. 
The admissible components in the definition of $N_{g,\beta_\Delta,n}^{X_\cf|\partial X_\cf,m}(\eta,\mu)$
and $N_{\mathrm{trop}}^{\Delta,g'}(\eta,\mu;q)$ coincide: 
for a summand in that definition, put $g_i'=g_i-m_i\geq0$ and 
apply the connected correspondence on the ambient target.
Components with no corresponding tropical curve have zero series. 
The ordinary and fixed labels split as in Remark~\ref{rem:floor-split}. Finally,
$$
\sum_i(2g_i-2+\ell_i)=2g-2+\ell,
\qquad
\sum_i(2g_i'-2+\ell_i)=2g'-2+\ell,
$$
and the fixed contact factors multiply. The correspondence below follows by applying \cite[Theorems 5--6]{bou2019} componentwise.

\begin{theorem}\label{thm:log-trop-total}
Let $X_\cf$ be the normal projective toric surface corresponding to an $h$-transverse lattice polygon $\Delta=\Delta(\cf,\df)$ with $d_b>0$.
Let $\eta,\mu$ be two partitions such that $|\eta|+|\mu|=d_b$.
Fix $g'\in\zb$ such that $0\leq\delta:=|\Delta^\circ\cap\zb^2|-g'<M$, where $M$ is the height of $\Delta$, and
$n=c_1(X_\cf)\cdot\beta_\Delta-d_b+l(\mu)+g'-1\geq0$.
Then, we have the equality
$$
\begin{aligned}
&\sum_{g\geq g'}
N_{g,\beta_\Delta,n}^{X_\cf|\partial X_\cf,g-g'}(\eta,\mu)
u^{2g-2+\ell}\\
&\quad=
\left(\prod_{j=1}^{l(\eta)}\frac{[\eta_j]_q}{\eta_j}\right)
N_{\mathrm{trop}}^{\Delta,g'}(\eta,\mu;q)
\left((-i)(q^{1/2}-q^{-1/2})\right)^{2g'-2+\ell}
\end{aligned}
$$
of power series in $u$ with rational coefficients, where $q=e^{iu}$ and $\ell=|\partial\Delta\cap\zb^2|-d_b+l(\eta)+l(\mu)$.
\end{theorem}

\section{Relative node polynomial for refined floor diagrams}\label{sec:log-polynomiality}

In this section, we establish polynomiality for the count \(N_{\floor}^{\Delta,g}(\eta,\mu;q)\) from Definition \ref{def:refined-count}.
Our proof combines the template construction of Ardila--Block \cite[Section 3]{AB-2013}, the relative terminal data of Block \cite[Section 3]{Block-2012}, and the combinatorial linearity theorem for long edge graphs recalled below.

We specify the degree conventions used throughout this section.
In $\deg_{\df}$ and $\deg_{\cf}$, every variable in the indicated group has degree one; $\deg_{(\cf,\df)}$ is the total degree with both groups assigned degree one.
The length relations are imposed by choosing independent coordinates on $\sum_i d_r^i=\sum_jd_l^j$, and $d_b$ is always replaced by its expression in $\cf,\df$.
For a polynomial $F$, write $[F]_{\df,k}$ for its homogeneous part of degree $k$ in the length variables.

\subsection{Long edge graphs and templates}
\label{subsec:range}

Let $\Delta$ be an $h$-transverse lattice polygon with $d_b>0$.
We introduce some notations that will be used in the following.
Suppose $L_j$ is a horizontal line at distance $M-j$ above the bottom edge of $\Delta$, where $j\in\{0,\ldots,M\}$.
Then the endpoints of the segment $L_j\cap\Delta$ have integral coordinates.
The length $W_j$ of the intersection $L_j\cap\Delta$ can be computed as follows.
Let $\rf^0=(r_1^0,\ldots,r_M^0)$ and $\lf^0=(l_1^0,\ldots,l_M^0)$ be two permutations of the multisets $d_r(\Delta)$ and $d_l(\Delta)$ such that
$r_1^0\geq r_2^0\geq\cdots\geq r_M^0$ and $l_1^0\leq l_2^0\leq\cdots\leq l_M^0$.
We call these the {\it natural orders} of the right directions and left directions.
From the notations in the introduction, we know that $W_j-W_{j-1}=r_j^0-l_j^0$ for $j\in\{1,\ldots,M\}$. Hence, we have
$$
W_j=W_j(\lf^0,\rf^0):=d_t+\sum_{v=1}^j(r_v^0-l_v^0),
W_0=W_0(\lf^0,\rf^0):=d_t,  W_M=W_M(\lf^0,\rf^0):=d_b.
$$
Let $\lf=(l_1,\ldots,l_M),\rf=(r_1,\ldots,r_M)$ be two arbitrary permutations of $d_l(\Delta)$ and $d_r(\Delta)$, respectively.
We define
$$
W_j(\lf,\rf):=d_t+\sum_{v=1}^j(r_v-l_v),
$$
and define the reversal weight  of the pair $(\lf,\rf)$ by
$$
\delta(\lf,\rf)=
\sum_{\substack{i<j\\r_i<r_j}}(r_j-r_i)
+\sum_{\substack{i<j\\l_i>l_j}}(l_i-l_j).
$$
A {\it reversal pair} of $\lf$ (resp. $\rf$) is a pair of entries $(l_i,l_j)$ in $\lf$ (resp. $(r_i,r_j)$ in $\rf$) such that $i<j$ and $l_i>l_j$ (resp. $i<j$ and $r_i<r_j$).

We collect some known estimates on $W_j, W_j(\lf,\rf)$.
\begin{lemma}[{\cite[equation (6.3)]{liu-osserman2018}}]
\label{lem: width1}
Let $\Delta$ be an $h$-transverse lattice polygon with height $M$, {and let $\lf,\rf$ be} two permutations of $d_l(\Delta)$ and $d_r(\Delta)$, respectively.
For $1\leq j<M$,
$
0\leq W_j-W_j(\lf,\rf)\leq\delta(\lf,\rf),
\sum_{j=1}^{M-1}\bigl(W_j-W_j(\lf,\rf)\bigr)=\delta(\lf,\rf)$.
\end{lemma}

\begin{definition}[{\cite[Definitions 11--12]{gk16}}]\label{def:poly-long edge}
A \textit{long edge graph} $G$ is a graph $(V,E)$ with a weight function $w:E\to\zb_{>0}$ satisfying:
\begin{itemize}
\item the vertex set is $V=\zb_{\geq0}$, and the edge set $E$ is finite;
\item $G$ may have multiple edges, but no loops;
\item $G$ has no \textit{short edges}, namely edges of weight $1$ connecting two adjacent vertices.
\end{itemize}
The {\it refined multiplicity} and {\it cogenus} of a long edge graph $G$ are
$$
\mult_q(G)=\prod_{e\in E}[w(e)]_q^2,~
\delta(G)=\sum_{e\in E}\bigl(\ell(e)w(e)-1\bigr).
$$
\end{definition}

An edge joining $i,j$ with $i<j$ is denoted by $e:i\to j$ and has length $\ell(e)=j-i$.
A vertex $v$ is \textit{strictly covered} if there exists some edge $i\to j$ satisfying $i<v<j$.
For a nonempty graph $G$, let $\operatorname{minv}(G)$ and $\operatorname{maxv}(G)$ be the smallest and largest vertices incident to an edge of $G$. The \textit{length} of $G$ is
$
\ell(G)=\operatorname{maxv}(G)-\operatorname{minv}(G).
$
Let $G^{(k)}$ denote the shifted graph obtained by increasing every vertex of $G$ by $k$.

\begin{definition}[{\cite[Definitions 14 and 23]{gk16}}]\label{def:poly-template}
A nonempty long edge graph $\Gamma$ with $\operatorname{minv}(\Gamma)=0$ is a \textit{template} if, for any vertex $1\leq i\leq\ell(\Gamma)-1$, there exists at least one edge $j\to k$ with $j<i<k$.
A long edge graph $G$ is a \textit{shifted template} if $G=\Gamma^{(k)}$ for a template $\Gamma$ and an integer $k\geq0$.

For a nonempty long edge graph $G$, set $\epsilon_0(G)=1$ if all edges incident to $\operatorname{minv}(G)$ have weight $1$, and set $\epsilon_0(G)=0$ otherwise.
Similarly, set $\epsilon_1(G)=1$ if all edges incident to $\operatorname{maxv}(G)$ have weight $1$, and set $\epsilon_1(G)=0$ otherwise.
\end{definition}

\begin{lemma}[{\cite[Lemma 4.2]{BCK-2014}}]
\label{lem:poly-template-decomposition}
Every nonempty long edge graph $G$ has a unique expression as an ordered union of translated templates.
\end{lemma}

\begin{definition}[{\cite[Definition 4]{Block-2012}}]
\label{def:poly-terminal}
An \textit{extended template} is a tuple $(\Lambda,A,B)$ of length $\ell=l(\Lambda)=l(\Lambda,A,B)$, where $\ell\geq0$, with the following properties.
The graph $\Lambda$ is a finite directed graph on $\{0,\ldots,\ell\}$, possibly with multiple edges, whose edge weights are positive integers and satisfy:
\begin{enumerate}[(i)]
\item If $i\to j$ is an edge, then $i<j$.
\item $\Lambda$ has no short edges.
\end{enumerate}
The matrices $A=(a_{ij})_{i,j\geq1}$ and $B=(b_{ij})_{i,j\geq1}$ have nonnegative integral entries and finite support.
Let $l(A)$ and $l(B)$ denote their largest nonzero row indices, with length $0$ for the zero matrix.
We require
$
l(\Lambda)\geq\max\{l(A),l(B)\},
$
and, for every $1\leq j\leq l(\Lambda)-\max\{l(A),l(B)\}$, there exists an edge $i\to k$ of $\Lambda$ with $i<j<k$.
\end{definition}
The cogenus of an extended template is
$\delta(\Lambda,A,B)=\delta(\Lambda)+\delta(A)+\delta(B)$, where
$$
\delta(\Lambda)=\sum_{e:i\to j\text{ in }\Lambda}\bigl((j-i)w(e)-1\bigr),~
\delta(A)=\sum_{i,j\geq1}ij\,a_{ij},~
\delta(B)=\sum_{i,j\geq1}ij\,b_{ij}.
$$

If $l(\Lambda)>0$, the covering condition forces $A$ and $B$ not to be both zero, since no edge can strictly cover the last vertex.
If $l(\Lambda)=0$, both matrices vanish and $\Lambda$ has a single vertex and no edges; this is the empty extended template, of cogenus $0$.
The covering condition gives
\begin{equation}\label{eq:poly-term-bound}
l(\Lambda)\leq\delta(\Lambda,A,B).
\end{equation}
Indeed, we have
$
l(\Lambda)-\max\{l(A),l(B)\}
\leq\sum_{e:i\to j\text{ in }\Lambda}(j-i-1)\leq\delta(\Lambda),
\max\{l(A),l(B)\}\leq\delta(A)+\delta(B).
$

\begin{lemma}[{{\cite[Sections 3.3.2 and 5]{AB-2013}}}]\label{lem:poly-defect}
Let $\wt\dl$ be an $(\eta,\mu)$-marked $\Delta$-floor diagram of genus $g$, and let $\delta=\delta(\dl)=|\Delta^\circ\cap\zb^2|-g$.
Then the following two statements hold.
\begin{enumerate}
\item The cogenus is given by
\begin{equation}\label{eq:poly-defect}
\delta={}\delta(\lf,\rf)
+\sum_{e:i\to j\text{ in }\dl}\bigl((j-i)w(e)-1\bigr)
+\sum_{v=1}^M(v-1)s_v
+\sum_{v=1}^M(M-v)(|\eta^v|+|\mu^v|).
\end{equation}
\item If $\cf,\df$ are sufficiently large and spread out, the width $W_j$ satisfies $W_j\geq2\delta$ for $1\leq j<M$.
\end{enumerate}
\end{lemma}

\begin{proof}
The second statement {follows from the estimates in Step 2} of the proof of Theorem 1.2 in \cite{AB-2013} under the hypotheses that $\cf,\df$ are sufficiently large and spread out.
The first statement is a relative version of the cogenus formula in \cite[Section~3.3.2]{AB-2013}. Here we provide a combinatorial proof of it.

We compute the genus of the corresponding unmarked $\Delta$-floor diagram $\dl$ first.
Suppose that the connected components of $\dl$ are  $\dl_1,\ldots,\dl_t$. Then $g=\sum_{i=1}^tg(\dl_i)+1-t$.
Note that $g(\dl_i)=|E(\dl_i)|-|V(\dl_i)|+1$, so we have
$g=|E(\dl)|-M+1$.
From the definition of $W_j$, we have $|\Delta^\circ\cap\zb^2|=\sum_{j=1}^{M-1}(W_j-1)$. Hence, the cogenus $\delta$ is given by $\delta=|\Delta^\circ\cap\zb^2|-g=\sum_{j=1}^{M-1}W_j-|E(\dl)|$.
From Lemma \ref{lem: width1}, we know $\sum_{j=1}^{M-1}\bigl(W_j-W_j(\lf,\rf)\bigr)=\delta(\lf,\rf)$.
Therefore, $\delta=\delta(\lf,\rf)+\sum_{j=1}^{M-1}W_j(\lf,\rf)-|E(\dl)|$.

In the following we compute $\sum_{j=1}^{M-1}W_j(\lf,\rf)$.
From the definition of divergence, we have
$$
\sum_{v=1}^j\di(v)
=\sum_{\substack{v\xrightarrow{e}b\\ v\leq j}}w(e)
-\sum_{\substack{a\xrightarrow{e}v\\ v\leq j}}w(e).
$$
We split the outgoing sum according to whether the other endpoint $b$ is at most $j$ or greater than $j$,
so we have
$$
\sum_{\substack{v\xrightarrow{e}b\\ v\leq j}}w(e)=
\sum_{\substack{a\xrightarrow{e}b\\ a<b\leq j}}w(e)+\sum_{\substack{a\xrightarrow{e}b\\ a\leq j<b}}w(e).
$$
Therefore, $\sum_{\substack{a\xrightarrow{e}b\\ a\leq j<b}}w(e)=\sum_{v=1}^j\di(v)$.
Since $\di(v)=r_v-l_v+s_v-(|\eta^v|+|\mu^v|)$, we get
$$
\sum_{e:a\leq j<b}w(e)
=\sum_{v=1}^j(r_v-l_v)+\sum_{v=1}^j s_v
-\sum_{v=1}^j(|\eta^v|+|\mu^v|).
$$
Note that $d_t=\sum_{i=1}^Ms_i$ and $W_j(\lf,\rf):=d_t+\sum_{v=1}^j(r_v-l_v)$, so we know
$$
\sum_{e:a\leq j<b}w(e)
=W_j(\lf,\rf)-\sum_{v=j+1}^M s_v-\sum_{v=1}^j(|\eta^v|+|\mu^v|).
$$
We now sum this equality over all $j\in\{1,\ldots,M-1\}$.
The sum of the {left-hand side} of the equality is
$$\sum_{j=1}^{M-1}\sum_{e:a\leq j<b}w(e)
=\sum_{e:a\to b}w(e)\sum_{j=a}^{b-1}1
=\sum_{e:a\to b}(b-a)w(e).
$$
Hence, we have
$$
\sum_{e:a\to b}(b-a)w(e)
=\sum_{j=1}^{M-1}W_j(\lf,\rf)-\sum_{v=1}^M(v-1)s_v
-\sum_{v=1}^M(M-v)(|\eta^v|+|\mu^v|).
$$
Since $|E(\dl)|=\sum_{e\in E(\dl)}1$, combining the value of $\sum_{j=1}^{M-1}W_j(\lf,\rf)$ with the equality $\delta=\delta(\lf,\rf)+\sum_{j=1}^{M-1}W_j(\lf,\rf)-|E(\dl)|$, we obtain equation (\ref{eq:poly-defect}).
\end{proof}

\begin{definition}
Let $\Delta=\Delta(\cf,\df)$ be a lattice polygon with height $M$ and let $\delta\geq0$.
Let $\eta=\emptyset, \mu=(1^{d_b})$,
{and let $(\{\eta^i\},\{\mu^i\})_{i=1}^M$ be} a compatible pair with $\eta^i=\emptyset$ for any $i\in\{1,\ldots,M\}$,
and {$\mu^j=\emptyset$ for every $j<M$, and $\mu^M=\mu$}.
An \textit{interior marked $\Delta$-floor diagram} is a marked diagram obtained from a $(\emptyset,1^{d_b})$-marked floor diagram with the compatible pair $(\{\eta^i\},\{\mu^i\})_{i=1}^M$ by deleting the $\mu$-vertices and their incident edges and preserving the induced linear order on the remaining vertices.
Define
$$
K_\delta(\Delta;q)=\sum_{\wt\dl}\mult_q(\dl),
$$
where the sum is over interior marked $\Delta$-floor diagrams of cogenus $\delta$ up to order-preserving isomorphisms fixing the original floor vertices and preserving weights and vertex types.
\end{definition}

\begin{definition}[{\cite[Definitions 13, 15]{gk16}}]\label{def:poly-ordering}
Let $G$ be a long edge graph. For any $j\geq1$, we define
$\lambda_j(G)=\sum_{e}w(e)$, where $e$ {runs over the edges} $e:i\to v$ with $i<j\leq v$.
Let $\beta=(\beta_0,\ldots,\beta_M)\in\zb_{\geq0}^{M+1}$.
The graph $G$ is called \textit{$\beta$-allowable} if $\operatorname{maxv}(G)\leq M+1$ and $\beta_{j-1}\geq\lambda_j(G)$ for $1\leq j\leq M+1$.

For a $\beta$-allowable graph $G$, define a new graph $\operatorname{ext}_\beta(G)$ by adding $\beta_{j-1}-\lambda_j(G)$ edges of weight $1$ joining $j-1$ and $j$ for $1\leq j\leq M+1$.
A \textit{$\beta$-extended ordering} of $G$ is a total order on the vertices and edges of $\operatorname{ext}_\beta(G)$ such that
\begin{itemize}
\item it extends the natural ordering of the vertices,
\item if $e$ is an edge connecting $i$ and $j$, then $e$ is between $i$ and $j$.
\end{itemize}
Two such orders are equivalent if one is obtained from the other by permuting edges with the same endpoints and weight.
\end{definition}
Given a long edge graph $G$, let $P_\beta(G)$ be the number of $\beta$-extended orderings of $G$ up to equivalence. Set $P_\beta(G)=0$ when $G$ is not $\beta$-allowable.
The empty graph is $\beta$-allowable, with $P_\beta(\emptyset)=1$.

\begin{definition}[{\cite[Definition 22]{gk16}}]
A \textit{partition} of a long edge graph $G=(V,E,w)$ is a tuple $(G_1,\ldots,G_n)$ of nonempty long edge graphs such that the disjoint union of the (weighted) edge sets of $G_1,\ldots,G_n$ is the (weighted) edge set of $G$.
For a long edge graph $G$, define
$$
\Phi_\beta(G)=\sum_{n\geq1}\frac{(-1)^{n+1}}n
\sum_{G_1,\ldots,G_n}\prod_{j=1}^nP_\beta(G_j),
$$
where the inner sums run over the partitions $(G_1,\ldots,G_n)$ of $G$.
\end{definition}

Let $\beta=(\beta_0,\ldots,\beta_M)\in\zb_{\geq0}^{M+1}$.
Define
$$
N_{\beta,\delta}(q)=\sum_{\delta(G)=\delta}\mult_q(G)P_\beta(G), \text{ and } N_{\beta,\delta}^0(q)=\sum_{\substack{\delta(G)=\delta\\ weight~ conditions}}\mult_q(G)P_\beta(G),
$$
where the first sum is taken over all long edge graphs $G$ of cogenus $\delta$, and the second sum is over all cogenus $\delta$ long edge graphs whose edges adjacent to 0 have weight 1. Define $Q_{\beta,\delta}(q)$ and $Q_{\beta,\delta}^0(q)$ via
$$
\log(1+\sum_{\delta>0}N_{\beta,\delta}(q)z^\delta)
=\sum_{\delta>0}Q_{\beta,\delta}(q)z^\delta, \text{ and }
\log(1+\sum_{\delta>0}N_{\beta,\delta}^0(q)z^\delta)
=\sum_{\delta>0}Q_{\beta,\delta}^0(q)z^\delta.
$$
The arguments giving \cite[equation (2.8)]{liu2016}, \cite[equation (3.9)]{liu-osserman2018}, 
and \cite[equation (2.4)]{gk16} yield
\begin{equation}\label{eq:poly-formal-log}
Q_{\beta,\delta}(q)=\sum_{\delta(G)=\delta}\mult_q(G)\Phi_\beta(G), \text{ and }
Q_{\beta,\delta}^0(q)=\sum_{\substack{\delta(G)=\delta\\ weight~ conditions}}\mult_q(G)\Phi_\beta(G).
\end{equation}

\begin{lemma}[{\cite[Corollary 3.1 and Lemma 5.9]{liu2016}}]\label{lem:poly-log}
Let $\beta=(\beta_0,\ldots,\beta_M)\in\zb_{\geq0}^{M+1}$. Then we have the following statements.
\begin{enumerate}
\item For a nonempty long edge graph $G$, we have $\Phi_\beta(G)=0$ unless $G$ is a shifted template.
\item For a template $\Gamma$ of length $\ell\leq M+1$, if $\beta_{j-1}\geq\lambda_j(\Gamma)$ for all  {$1\leq j\leq\ell$, we have}
\begin{equation}\label{eq:poly-affine}
\Phi_\beta(\Gamma)=\eta_0(\Gamma)+\sum_{j=1}^{\ell}\eta_j(\Gamma)\beta_{j-1},
\end{equation}
where $\eta_j(\Gamma)$ are constants.
\end{enumerate}
\end{lemma}

\subsection{Polynomiality for refined counts of relative floor diagrams}\label{subsec:poly}

\subsubsection{The interior polynomial}
The number of pairs $(\lf,\rf)$ with $\delta(\lf,\rf)=\delta_0\leq \delta$ is determined in
 \cite[Proposition~6.8]{liu-osserman2018} and \cite[proof of Corollaries~6.4--6.5]{liu-osserman2018}, and we summarize it as follows.
For an integer bound $\delta\geq0$, let $m_\cf(\delta_0)$, $0\leq\delta_0\leq \delta$, be the coefficients determined by
\begin{equation}\label{eq:poly-swaps}
\sum_{\delta_0=0}^\delta m_\cf(\delta_0)z^{\delta_0}
\equiv
\prod_{i<p_r}\prod_{v\geq1}(1-z^{v(c_r^i-c_r^{i+1})})^{-1}
\prod_{j<p_l}\prod_{v\geq1}(1-z^{v(c_l^{j+1}-c_l^j)})^{-1}
\mod z^{\delta+1}.
\end{equation}
The congruence modulo $z^{\delta+1}$ means that the coefficients in degrees $0,\ldots,\delta$ agree.

\begin{lemma}[{\cite[Proposition 6.8, proof of Corollaries~6.4--6.5]{liu-osserman2018}}]
\label{lem:poly-swaps}
Let $\Delta(\cf,\df)$ be an $h$-transverse lattice polygon. Suppose that $\cf,\df$ are sufficiently large and spread out for $\delta\geq1$.
Then the number of pairs $(\lf,\rf)$ with $\delta(\lf,\rf)=\delta_0\leq \delta$ is $m_\cf(\delta_0)$.
This number is independent of $\df$.
If condition \eqref{eq:poly-gap-range} holds, $m_\cf(0)=1$ and $m_\cf(\delta_0)=0$ for $1\leq\delta_0\leq \delta$.
\end{lemma}

\begin{lemma}\label{lem:poly-interior}
The interior count
	$K_e(\Delta(\cf,\df);q)$ has the following properties.
	\begin{enumerate}
		\item[(1)] 
		For fixed $\cf$ and $e\ge1$, there exists a polynomial
		$\mathcal K_e^{\cf}(\df;q)\in\qb[q^{\pm1}][\df]$
		of degree $2e$ in $\df$ that equals 
		$K_e\bigl(\Delta(\cf,\df);q\bigr)$,
		whenever equation \eqref{eq:poly-range} holds with $\delta=e$.
		
		\item[(2)] 
		For fixed $p_r,p_l$, there exists a polynomial $\mathcal K_e$ in both $\cf$ and
		$\df$, of degree $e$ in $\cf$, degree $2e$ in $\df$, and total degree $3e$, that
		agrees with the count, that is,
		$\mathcal K_e(\cf,\df;q)=K_e\bigl(\Delta(\cf,\df);q\bigr)$,
		whenever equations \eqref{eq:poly-range} and \eqref{eq:poly-gap-range} hold
		simultaneously with $\delta=e$.

		\item[(3)]
		For either construction and every $e\geq0$, the homogeneous part of
		length degree $2e$ is
		\begin{equation}\label{eq:poly-interior-leading}
		\bigl[\mathcal K_e^\cf\bigr]_{\df,2e}
		=\bigl[\mathcal K_e\bigr]_{\df,2e}
		=\frac1{e!}\left((q+4+q^{-1})\frac{\area(\Delta)}2\right)^e.
		\end{equation}
		The right-hand side is understood as a polynomial in the polygon coordinates.

	\end{enumerate}
	
	We set $\mathcal K_0^\cf=\mathcal K_0=K_0=1$, of degree zero.
\end{lemma}

\begin{proof}
(1).
Suppose that $\cf,\df$ are sufficiently large and spread out for $\delta=e$.
Let $M$ be the height of the $h$-transverse polygon $\Delta(\cf,\df)$.

The proof of the correspondence follows the argument in the proof of \cite[Theorem 20]{gk16}.
Fix a pair of side
permutations $(\lf,\rf)$ with reversal weight
$\delta_0=\delta(\mathbf l,\mathbf r)\in[0,e]$.
The construction below gives a bijection between the interior marked floor diagrams contributing to $K_e(\Delta(\cf,\df);q)$ with permutations $(\lf,\rf)$ and {pairs $(G, ord)$ consisting of a $\beta$-allowable graph $G\in\mathcal G(\lf,\rf)$ of cogenus $e-\delta_0$ and a $\beta$-extended ordering $ord$ of $G$, up to equivalence}. Here $\mathcal G(\lf,\rf)$ {denotes all long edge graphs} supported on $\{0,\ldots,M\}$ whose edges incident to $0$ have weight one and $\beta$ is defined by 
$
\beta=(W_0(\mathbf l,\mathbf r),W_1(\mathbf l,\mathbf r),\dots,W_{M-1}(\mathbf l,\mathbf r))$.

Let $\mathcal{D}$ be a $(\emptyset,1^{d_b})$-marked floor diagram and let $\widetilde{\mathcal{D}}$ be the graph obtained from
the interior marked floor diagram of $\dl$ by {forgetting the subdivision vertices}.
Add a vertex $0$ before all floors of $\wt\dl$. Each indistinguishable vertex corresponding to the sequence $\sff$ of $\wt\dl$ and attached to a floor $v$ becomes the subdivision vertex of the edge $0\to v$.
Delete all short edges of the augmented graph, and one can obtain a long edge graph $G$. 
Since the sum of the number of deleted short edges adjacent to $j-1$, $j$ and $\lambda_j(G)$ is $d_t+\sum_{i=1}^{j-1}(r_i-l_i)$, we know the number of deleted short edges adjacent to $j-1$ and $j$
is equal to $W_{j-1}(\lf,\rf)-\lambda_j(G)$.
Hence $G$ is a $\beta$-allowable long edge graph on
$\{0,\dots,M\}$ with
all edges incident to $0$ of weight one.
Note that 
$\delta(G)=\sum_{a:i\to j}\bigl((j-i)w(a)-1\bigr)
+\sum_{v=1}^M(v-1)s_v$, 
where the first sum is over edges $a:i\to j$ with $i>0$. From equation (\ref{eq:poly-defect}), we know that
the long graph $G$ has cogenus $e-\delta(\lf,\rf)=e-\delta_0$.
The ordering $ord$ is determined by the total order on vertices of $\dl$.

Let $G$ be a $\beta$-allowable long edge graph with a $\beta$-extended ordering. 
We glue $W_{j-1}(\mathbf l,\mathbf r)-\lambda_j(G)$ short edges to $G$ adjacent to the vertices $j-1$ and $j$. Then, we replace the subdivision vertices of the edges emanating from $0$ with top marking vertices, delete $0$, and add new subdivision vertices. This gives the inverse construction.
Markings are counted by $P_{\beta}(G)$, and the multiplicity is unchanged.

Fix $(\lf,\rf)$ with $\delta(\lf,\rf)=\delta_0$ and $\beta$, {and then consider} the formal series
\begin{equation}\label{eq:S}
S(\lf,\rf;z)=1+\sum_{\delta>0}N_{\beta,\delta}^0(q)z^\delta=\sum_{G}\mult_q(G)P_\beta(G)z^{\delta(G)},
\end{equation}
where the second sum is over all long edge graphs whose edges incident to $0$ have weight one, including the empty graph. 
Recall that $P_\beta(G)$ records the marking data and is zero for nonallowable graphs.
Note that the coefficient 
$[\,z^{\,e-\delta_0}\,]S(\mathbf l,\mathbf r;z)$ is the $(\mathbf l,\mathbf r)$-part of $K_e$.

From equation~\eqref{eq:poly-formal-log}, we obtain
$$
\log S(\mathbf l,\mathbf r;z)
=\sum_{h\geq1}Q_{\beta,h}^0(q)z^h
=\sum_{\substack{G\in\mathcal G(\mathbf l,\mathbf r)\\G\neq\varnothing}}
\operatorname{mult}_q(G)\,\Phi_\beta(G)\,z^{\delta(G)}.
$$ 
To compute $[z^{e-\delta_0}]S$, only terms of cogenus at most $e-\delta_0$ are needed. Lemma~\ref{lem:poly-log} removes non-template graphs, leaving finitely many templates and their shifts.

Every graph $G\in\mathcal G(\lf,\rf)$ of cogenus at most $e-\delta_0$, and every graph in a partition of $G$, is $\beta$-allowable. Indeed, a nonshort edge of length $d$ and weight $w$ satisfies $w\leq2(dw-1)$, so $\lambda_j(G)\leq2(e-\delta_0)$ for $j\geq2$. Lemmas \ref{lem: width1} and \ref{lem:poly-defect}(2) give $W_{j-1}(\lf,\rf)\geq2e-\delta_0\geq2(e-\delta_0)$. At $0$, every incident edge has weight one and length at least two, so $\lambda_1(G)\leq e-\delta_0\leq d_t$. Passing to a subgraph only decreases these bounds.
From Lemma \ref{lem:poly-log}, we have the following:
\begin{enumerate}
\item if $G$ is not a translate of a template, then
$\Phi_{\beta}(G)=0$, so such graphs have no contribution;
\item if $G$ is a template $\Gamma$ of length $\ell$ with vertex set $\{0,\dots,\ell\}$,
the shift $k$ ranges over $a_\Gamma:={1-\epsilon_0(\Gamma)}\le k\le M-\ell$, and the shift preserves cogenus
and multiplicity.
For $\delta(\Gamma)\leq e-\delta_0$, from Lemma \ref{lem:poly-log}(2), we have
$\Phi_\beta(\Gamma^{(k)})=\eta_0(\Gamma)+\sum_{j=1}^{\ell}\eta_j(\Gamma)W_{k+j-1}(\lf,\rf)$.
\end{enumerate}
Hence, $\log S(\mathbf l,\mathbf r;z)$ is given by
\begin{equation}\label{eq:logS}
\begin{aligned}
\log S(\mathbf l,\mathbf r;z)
{\equiv}\sum_\Gamma\sum_{k=a_\Gamma}^{M-\ell}\mult_q(\Gamma)z^{\delta(\Gamma)}
\Bigl({\eta_0(\Gamma)}+\sum_{j=1}^{\ell}{\eta_j(\Gamma)}W_{k+j-1}(\lf,\rf)\Bigr)\mod z^{e-\delta_0+1},
\end{aligned}
\end{equation}
where the first sum is over the finite set of templates with
${1\leq\delta(\Gamma)\leq e-\delta_0}$.
For fixed $j$, the index $k+j-1$ runs over
$[a_\Gamma+j-1,\ M-\ell+j-1]$, so we have
$$
\sum_{k=a_{\Gamma}}^{M-\ell}W_{k+j-1}(\mathbf l,\mathbf r)
=\sum_{m=0}^{M-1}W_m(\mathbf l,\mathbf r)
-\sum_{m=0}^{a_\Gamma+j-2}W_m(\mathbf l,\mathbf r)
-\sum_{m=1}^{\ell-j}W_{M-m}(\mathbf l,\mathbf r).
$$

By Lemma \ref{lem: width1},
$\sum_{m=0}^{M-1}W_m(\mathbf l,\mathbf r)=\sum_{j=0}^{M-1}W_j-\delta_0$. Note that $\sum_{j=0}^{M-1}W_j=\sum_{j=0}^{M-1}(d_t+\sum_{v=1}^j(r_v^0-l_v^0))$,
where $\lf^0,\rf^0$ are the naturally ordered permutations.
An entry $c_r^i$ in the $v$-th position of the sequence $\rf^0$ occurs in exactly $M-v$ of the widths $W_0,\ldots,W_{M-1}$.
The entries equal to $c_r^i$ in $\rf^0$, starting after the $H$-th position up to the $(H+d)$-th position, contribute $c_r^i\sum_{v=1}^d(M-H-v)=c_r^i\left(d(M-H)-\binom{d+1}{2}\right)$. The same computation applies to the sequence $\lf^0$.
Hence, we have
\begingroup\small
\begin{equation}\label{eq:dadi-theta}
	\begin{aligned}
		\sum_{j=0}^{M-1}W_j
		={}Md_t
		+\sum_{i=1}^{p_r} c_r^i\left(d_r^i(M-\sum_{v=1}^{i-1}d_r^v)-\binom{d_r^i+1}{2}\right)
		-\sum_{j=1}^{p_l} c_l^j\left(d_l^j(M-\sum_{v=1}^{j-1}d_l^v)-\binom{d_l^j+1}{2}\right).
	\end{aligned}
\end{equation}
\endgroup
So $\sum_{j=0}^{M-1}W_j$ has length degree at
most two, direction degree at most one, and finitely many terms.

Let $h=\delta(\Gamma)\le e-\delta_0$.
The first and last $h$ entries of $\lf$ and $\rf$ must be the first and last $h$ entries of the {naturally ordered sequences $\lf^0$ and $\rf^0$}, respectively. Indeed, let $a$ be an entry different from the first (resp. last) $h$ entries of the {naturally ordered sequence $\rf^0$}.
From equation \eqref{eq:poly-range}, every entry occurs at least $e+1$ times in $\rf$.
If $a$ is in the first (resp. last) $h$ positions of $\rf$, there must be at least $e+1-h$ entries greater (resp. smaller) than $a$, but sitting after (resp. before) $a$ in $\rf$.
This will create at least $\delta_0+1$ reversal pairs,
contradicting the total reversal weight $\delta(\lf,\rf)=\delta_0$.
A similar argument applies to $\lf$. Hence, for ${0\leq m\leq h}$, we have
$W_m(\mathbf l,\mathbf r)=d_t+m\,(c^1_r-c^1_l)$, 
$W_{M-m}(\mathbf l,\mathbf r)=d_b-m\,(c^{p_r}_r-c^{p_l}_l)$.
Since $1\leq j\leq\ell$ and $\ell-1\leq h$, we have $a_\Gamma+j-2\leq\ell-1\leq h$ and $\ell-j\leq\ell-1\leq h$. The bound $\ell-1\leq h$ follows because the edges of a template strictly cover all its $\ell-1$ interior vertices.
Then,
$$
\begin{aligned}
\sum_{m=0}^{a_\Gamma+j-2}W_m(\mathbf l,\mathbf r)
&=(a_\Gamma+j-1)\,d_t+(c^1_r-c^1_l)\tbinom{a_\Gamma+j-1}{2},\\
\sum_{m=1}^{\ell-j}W_{M-m}(\mathbf l,\mathbf r)
&=(\ell-j)\,d_b-(c^{p_r}_r-c^{p_l}_l)\tbinom{\ell-j+1}{2}.
\end{aligned}
$$ 
Finally, we obtain
\begin{equation}\label{eq:dadi-trans}
	\begin{aligned}
		\sum_{k=a_\Gamma}^{M-\ell}W_{k+j-1}(\lf,\rf)
		={}&\sum_{j=0}^{M-1}W_j-\delta_0-(a_\Gamma+j-1)d_t
		-(c_r^1-c_l^1)\binom{a_\Gamma+j-1}{2}\\
		&-(\ell-j)d_b+(c_r^{p_r}-c_l^{p_l})\binom{\ell-j+1}{2}.
	\end{aligned}
\end{equation}

Substituting equation {\eqref{eq:dadi-trans}} into equation \eqref{eq:logS} {and grouping templates by cogenus, we obtain}
\begingroup\small
$$
\begin{aligned}
&\log S(\mathbf l,\mathbf r;z)\\
&\equiv\sum_{h=1}^{e-\delta_0}z^h\sum_{\Gamma:\,\delta(\Gamma)=h}\mult_q(\Gamma)
\left((M-\ell-a_\Gamma+1){\eta_0(\Gamma)}
+\sum_{j=1}^{\ell}{\eta_j(\Gamma)}\mathcal W_j(\Gamma;\delta_0)\right)
\mod z^{e-\delta_0+1},
\end{aligned}
$$
\endgroup
where $\mathcal W_j(\Gamma;\delta_0)$ 
denotes the right-hand side of equation~\eqref{eq:dadi-trans}.
Since $S(\mathbf l,\mathbf r;z)=\exp(\log S(\mathbf l,\mathbf r;z))$, by extracting the coefficient of $z^{e-\delta_0}$, 
we obtain the finite coefficient formula:
\begin{equation}\label{eq:poly-interior}
\begin{aligned}
\mathcal K^{\mathbf c}_e(\mathbf d;q)
=\sum_{\delta_0=0}^{e}m_{\mathbf c}(\delta_0)&[z^{e-\delta_0}]
\exp\Bigl(\sum_{h=1}^{e}z^h\sum_{\Gamma:\,\delta(\Gamma)=h}\mult_q(\Gamma)\\
&\cdot{\Bigl((M-\ell-a_\Gamma+1)\eta_0(\Gamma)
+\sum_{j=1}^{\ell}\eta_j(\Gamma)\mathcal W_j(\Gamma;\delta_0)\Bigr)}\Bigr),
\end{aligned}
\end{equation}
where $m_{\mathbf c}(\delta_0)$ is the number of $\lf,\rf$ with reversal
weight $\delta_0$. By Lemma \ref{lem:poly-swaps}, this number is independent of $\df$.

It remains to justify the identification with $K_e$. From \eqref{eq:S}, the coefficient $[\,z^{\,e-\delta_0}\,]S(\mathbf l,\mathbf r;z)$ is the $(\mathbf l,\mathbf r)$-part of $K_e$.
{From the equation for $\log S(\mathbf l,\mathbf r;z)$ above}, the coefficient $[\,z^{\,e-\delta_0}\,]S(\mathbf l,\mathbf r;z)$ depends on $(\mathbf l,\mathbf r)$ only through $\delta_0$.
Hence, all $m_{\mathbf c}(\delta_0)$ permutations with the
same reversal weight $\delta_0$ contribute identically, 
and summing over $\delta_0$ with the multiplicity $m_{\mathbf c}(\delta_0)$ reproduces the full count:
$$
K_e(\Delta;q)={\sum_{\substack{(\lf,\rf)\\\delta(\lf,\rf)\leq e}}}
[\,z^{\,e-\delta(\mathbf l,\mathbf r)}\,]\,S(\mathbf l,\mathbf r;z)
=\mathcal K^{\mathbf c}_e(\mathbf d;q).
$$

In the right-hand side of equation \eqref{eq:poly-interior}, ${\eta_j(\Gamma)}$ are
{rational constants, and} $M-\ell-a_\Gamma+1$ is 
linear in the length data, since $M=\sum_i d_r^i$.
$\sum_{j=0}^{M-1}W_j-\delta_0$ and the right-hand side of equation \eqref{eq:dadi-trans} have length degree
at most two and direction degree at most one. 
Since the set of templates with
$\delta(\Gamma)\le e$ is finite and the coefficient extraction is finite, the
right-hand side of {\eqref{eq:poly-interior}} lies in $\mathbb Q[q^{\pm1}][\mathbf d]$.
Each $z^h$-coefficient of the exponent has length degree at most two.
A product contributing to $z^{e-\delta_0}$ has at most $e-\delta_0$
factors, since every factor has positive $z$-degree. This gives the
length-degree bound $2e$ {in statement (1)}.

(2).
Under the direction-gap hypotheses {\eqref{eq:poly-gap-range}},
{Lemma~\ref{lem:poly-swaps}} gives $m_{\mathbf c}(0)=1$ and
$m_{\mathbf c}(\delta_0)=0$ for $\delta_0>0$, and the count lies in
$\qb[q^{\pm1}][\cf,\df]$.
Each $z^h$-coefficient of the exponent has length degree at most two, direction degree at most one, and joint degree at most three. A product contributing to $z^{e-\delta_0}$ has at most $e-\delta_0$ factors, since every factor has positive $z$-degree. This gives the bounds $2e$, $e$, and $3e$.

(3). We calculate the highest length-degree contribution to the interior
polynomial of statement (1) and statement (2). In equation \eqref{eq:poly-interior}, every coefficient in
the exponent has length degree at most two. The terms with positive
reversal weight $\delta_0$ have length degree at most $2(e-\delta_0)$.
For the natural side orders, let $Q_1$ be the coefficient of $z$ in
the exponent. Only $\delta_0=0$ and a product of $e$ cogenus-one
contributions can have length degree $2e$.

The two templates of cogenus one are a length-one edge of weight two
and a length-two edge of weight one.  Computing these two contributions gives
$$
\begin{aligned}
Q_1&=(q+2+q^{-1})\sum_{k=1}^{M-1}(W_k-1)
+\sum_{k=0}^{M-2}(W_k+W_{k+1})\\
&=(q+2+q^{-1})\bigl(\sum_{k=0}^{M-1}W_k-d_t-(M-1)\bigr)
+2\sum_{k=0}^{M-1}W_k-d_t-d_b+c_r^{p_r}-c_l^{p_l}.
\end{aligned}
$$
Decompose the polygon $\Delta$ into horizontal strips of height one, 
and compute the strip areas using the trapezoid formula. Then we have
$\area(\Delta)=2\sum_{k=0}^{M-1}W_k-d_t+d_b$. Therefore,
\[
[Q_1]_{\df,2}
=([2]_q^2+2)\frac{\area(\Delta)}2
=(q+4+q^{-1})\frac{\area(\Delta)}2.
\]
Thus the above coefficient implies equation \eqref{eq:poly-interior-leading}. Finally, $e=0$ forces {naturally ordered permutations} and the empty long edge
graph, whose extended ordering is unique, so $K_0=1$.
\end{proof}

\subsubsection{Relative decomposition and boundary markings}

We separate a relative marked floor diagram into an interior marked floor diagram and an extended template.
The construction follows the extended template decomposition of
\cite[Section~3]{Block-2012} closely. 

Let $\Delta=\Delta(\cf,\df)$ be an $h$-transverse polygon where $\cf,\df$ are sufficiently large and spread out for $\delta>0$. For $0\leq\ell\leq\delta$, let $\Delta^{(\ell)}$ be the polygon obtained by
deleting the bottom $\ell$ rows of $\Delta$. Its height is $M-\ell$, its bottom length is $d_b^{(\ell)}=W_{M-\ell}=d_b-\ell(c_r^{p_r}-c_l^{p_l})$, and its last right and left side lengths are $d_r^{p_r}-\ell$ and $d_l^{p_l}-\ell$, respectively.

Fix two partitions $\eta,\mu$ with $|\eta|+|\mu|=d_b$.
We put
$$
a_w=m_w(\eta),\qquad b_w=m_w(\mu).
$$
Let $(\Lambda,A,B)$ be an extended template with length $\ell=l(\Lambda)$.
The extended template $(\Lambda,A,B)$ is called \emph{$(\eta,\mu)$-admissible} if
$\sum_{i\geq1}a_{iw}\leq a_w$, $\sum_{i\geq1}b_{iw}\leq b_w$ for any $w\geq1$.
We use the weighted lower sums introduced in \cite[Section~3]{Block-2012}
$$
\operatorname{wls}(A)_i=\sum_{i'\geq i,\,w\geq1}w\,a_{i'w},\qquad
\operatorname{wls}(B)_i=\sum_{i'\geq i,\,w\geq1}w\,b_{i'w}.
$$
Given an extended template $(\Lambda,A,B)$ with length $\ell$ and cogenus $\delta(\Lambda,A,B)\leq\delta$,  we define its \emph{short edge capacities} by
\begin{equation}\label{Eqt:def-Sj}
\begin{aligned}
S_j={}d_b-(\ell-j+1)(c_r^{p_r}-c_l^{p_l})
-\lambda_j(\Lambda)
-\operatorname{wls}(A)_{\ell+1-j}
-\operatorname{wls}(B)_{\ell+1-j},
\end{aligned}
\end{equation}
for any $1\leq j\leq\ell$.
Note that the short edge capacities $S_j$ are nonnegative.
Indeed, for $\ell>0$, we have $1\leq M-\ell+j-1<M$. By Lemma~\ref{lem:poly-defect}\textup{(2)},
$W_{M-\ell+j-1}\geq2\delta$. Each nonshort edge of length $d$ and weight $w$ satisfies $w\leq2(dw-1)$, 
so $\lambda_j(\Lambda)\leq2\delta(\Lambda)$. Moreover,
$\operatorname{wls}(A)_{\ell+1-j}
+\operatorname{wls}(B)_{\ell+1-j}
\leq\delta(A)+\delta(B)$.
It follows that
$S_j\geq2\delta-2\delta(\Lambda)-\delta(A)-\delta(B)\geq0$.

For an $(\eta,\mu)$-admissible template $(\Lambda,A,B)$ and a polygon $\Delta$, 
{we construct the associated \emph{terminal diagram}} as follows.
Translate each edge $i\to j$ of $\Lambda$ to $M-\ell+i\to M-\ell+j$,
keeping its weight, and insert $S_j$ short edges of weight one from
$M-\ell+j-1$ to $M-\ell+j$ for $1\leq j\leq\ell$.
Attach $a_{iw}$ (resp. $b_{iw}$) edges of weight $w$ to the 
floor $M-i$, pointing from the floor $M-i$ to $a_{iw}$ (resp. $b_{iw}$) different $A$-vertices (resp. $B$-vertices). 
Attach $(a_w-\sum_{i\geq1}a_{iw})$ (resp. $(b_w-\sum_{i\geq1}b_{iw})$)
edges of weight $w$ to floor $M$, pointing from the floor $M$ to $(a_w-\sum_{i\geq1}a_{iw})$ (resp. $(b_w-\sum_{i\geq1}b_{iw})$) different $A$-vertices (resp. $B$-vertices). 
Subdivide every edge connecting two floors by a subdivision vertex.

A \emph{compatible marking} of the $(\eta,\mu)$-admissible template $(\Lambda,A,B)$ and the polygon $\Delta$ is a total order on the set consisting of the floors
$M-\ell,\ldots,M$, the subdivision vertices, and {the $A$-vertices and $B$-vertices} of the terminal diagram, compatible with the orientation of the terminal diagram. {It is also required that} the $A$-vertices are the largest among all vertices, and that for every pair of $A$-vertices $i'>i$, the weight of the edge adjacent to $i'$ is greater than or equal to the weight of the edge adjacent to $i$.
Two such markings are equivalent if they are related by an order-preserving isomorphism
fixing the floors and preserving weights and vertex types.

\begin{lemma}\label{lem:poly-decomp}
Let $\Delta=\Delta(\cf,\df)$ be an $h$-transverse polygon where $\cf,\df$ are sufficiently large and spread out for $\delta>0$. 
Then there is a bijection from the set of $(\eta,\mu)$-marked $\Delta$-floor diagrams of cogenus $\delta$ 
to the set of triples consisting of
\begin{enumerate}[$(i)$]
\item an $(\eta,\mu)$-admissible extended template
$T=(\Lambda,A,B)$ with $\delta(T)\leq\delta$, of length
$\ell=l(\Lambda)$;
\item an interior marked $\Delta^{(\ell)}$-floor diagram of
cogenus $e=\delta-\delta(T)$;
\item a compatible marking of $T$ and $\Delta$.
\end{enumerate}
In the bijection, the interior diagrams are precisely those counted
by $K_e(\Delta^{(\ell)};q)$. Moreover, if $e>0$, the truncated polygon
$\Delta^{(\ell)}$ satisfies condition \eqref{eq:poly-range} with $\delta$ replaced
by $e$.
\end{lemma}

\begin{proof}
Let $\wt\dl$ be an $(\eta,\mu)$-marked $\Delta$-floor diagram
of cogenus $\delta$. Let $\dl'$ be the graph at the end of Step~2 of
Definition~\ref{def:relative-marking}.
Adjoin a vertex $0$ and connect it to the $s_j$ indistinguishable vertices by edges of weight one
so that these edges become directed edges $0\to j$. The original $s_j$ indistinguishable vertices
become the subdivision vertices of these auxiliary edges.
Omit the edges adjacent to the $\eta$-vertices and $\mu$-vertices and all short edges. 
The result is a long edge graph $G$ on $\{0,\ldots,M\}$.
For $1\leq i<M$ and $w\geq1$, let
$$
a_{iw}=m_w(\eta^{M-i}),\qquad
b_{iw}=m_w(\mu^{M-i}),
$$
and set the remaining rows of $A$ and $B$ to zero. Thus
$M-\max\{l(A),l(B)\}$ is the earliest floor of $\dl'$ connecting to an $\eta$-vertex or $\mu$-vertex, unless $A=B=0$. Let
$$
b=\max\bigl\{v\in\{0,\ldots,M-\max(l(A),l(B))\}:
\text{$v$ is not strictly covered by an edge of $G$}\bigr\}.
$$
This set contains $0$. No edge of $G$ strictly crosses $b$, and no
$\eta$-vertex or $\mu$-vertex is connected to a floor less than $b$ in $\dl'$. 
Let $\Lambda$ be the graph consisting of the edges of $G$ with endpoints in $\{b,\ldots,M\}$,
translated by $-b$, and put $\ell=M-b$.
By maximality of $b$, each vertex
$1\leq j\leq\ell-\max\{l(A),l(B)\}$ is strictly covered by an edge of
$\Lambda$. Hence $T=(\Lambda,A,B)$ is an extended template. The
cogenus formula \eqref{eq:poly-defect} and the fact $l(\Lambda)\leq\delta(\Lambda,A,B)$ imply
$\ell\leq\delta(T)\leq\delta$.
If $\delta>0$, condition \eqref{eq:poly-range} implies
$M\geq\delta+1$, so $b\geq1$. If $\delta=0$, the template is empty
and $b=M\geq1$. A top edge $0\to v$ with $v>b$ would strictly cover $b$. Consequently $s_j=0$ for every $j>b$.

We claim that the last $\ell$ entries of $\rf$ and $\lf$ equal
$c_r^{p_r}$ and $c_l^{p_l}$, respectively. This is vacuous for
$\ell=0$. Otherwise, if one of the last $\ell$ right entries
differed from $c_r^{p_r}$, at least
$d_r^{p_r}-(\ell-1)\geq\delta+2-\ell$
copies of the smallest right direction would precede it. Each
would contribute at least one to the reversal cogenus $\delta(\lf,\rf)$, whereas
equation \eqref{eq:poly-defect} implies that
$\delta(\lf,\rf)\leq\delta-\delta(T)\leq\delta-\ell$.
This is a contradiction. The left side follows in the same way, since
$c_l^{p_l}$ is the largest left direction. Thus deleting the final
$\ell$ entries of $\lf,\rf$ produces side permutations for $\Delta^{(\ell)}$ without
changing the reversal cogenus. In particular, we have $W_b(\lf,\rf)=W_b=d_b^{(\ell)}$.

Let $\wt\dl'$ be the graph consisting of the floors $1,\ldots,b$ in $\wt\dl$, 
all edges of $\wt\dl$ between these floors together with their subdivision vertices, and all edges corresponding to the sequence $\sff$ of $\wt\dl$. 
Here $\di$ denotes divergence in the underlying floor graph, before adding the top marking vertices. For $v<b$, this divergence in the restricted graph equals $r_v-l_v+s_v$. Since every edge of the restricted floor graph has both endpoints in $\{1,\ldots,b\}$, the sum of its divergences vanishes, giving
$\di(b)=-\sum_{v=1}^{b-1}(r_v-l_v+s_v)$.
Hence, 
$$
r_b-l_b+s_b-\di (b)=d_t+\sum_{v=1}^{b}(r_v-l_v)
=W_b(\lf,\rf)=d_b^{(\ell)}.
$$
Add $d_b^{(\ell)}$ edges of weight one to the floor $b$ in $\wt\dl'$ {pointing away from $b$,}
then we have a $(\emptyset,1^{d_b^{(\ell)}})$-marked $\Delta^{(\ell)}$-floor
diagram. Hence, $\wt\dl'$ is an interior marked $\Delta^{(\ell)}$-floor diagram.  
From equation \eqref{eq:poly-defect}, the cogenus of $\wt\dl'$ is
$e=\delta-\delta(\Lambda,A,B)$, since the reversal and top-contact contributions
are unchanged, and the removed contributions are exactly
$\delta(\Lambda)$, $\delta(A)$, and $\delta(B)$.
The restriction of the original total order to the remaining vertices gives the compatible marking {in item} $(iii)$.

Conversely, we obtain an $(\eta,\mu)$-marked $\Delta$-floor diagram from an extended template $T=(\Lambda,A,B)$ in \textup{(i)}, an interior marked floor diagram $\wt\dl'$ in \textup{(ii)} and a compatible marking in $(iii)$ as follows.
We identify the last floor $b$ of the interior marked floor diagram $\wt\dl'$ with the vertex $M-\ell$ of the terminal diagram of $(\Lambda,A,B)$, and obtain a diagram $\wt\dl$. 
The order of vertices of the terminal diagram of $(\Lambda,A,B)$ is given by the compatible marking.
Hence, all vertices in $\wt\dl$ can be ordered as in step 4 of Definition \ref{def:relative-marking}.
Add the last $\ell$ entries of $\lf^0,\rf^0$ to the end of the two side permutations of $\Delta^{(\ell)}$ obtaining two permutations $\lf,\rf$ for an $h$-transverse polygon $\Delta$. The sequence $\sff$ is obtained by appending $\ell$ zeros to that of the interior marked floor diagram $\wt\dl'$.
For every $1\leq j\leq M$ and $w\geq1$, we
define the partitions $\eta^j$ and $\mu^j$ by
$$
 m_w(\eta^j)=
 \begin{cases}
 0,&1\leq j<b,\\
 a_{M-j,w},&b\leq j<M,\\
 a_w-\displaystyle\sum_{i=1}^{\ell}a_{iw},&j=M,
 \end{cases}
 \qquad
 m_w(\mu^j)=
 \begin{cases}
 0,&1\leq j<b,\\
 b_{M-j,w},&b\leq j<M,\\
 b_w-\displaystyle\sum_{i=1}^{\ell}b_{iw},&j=M.
 \end{cases}
$$
These multiplicities give
$\eta=\bigcup_{v=1}^M\eta^v$ and
$\mu=\bigcup_{v=1}^M\mu^v$, with the property $|\eta|+|\mu|=d_b$.

We check the divergence compatibility required in
Definition~\ref{def:relative-marking}. Suppose first that $\ell>0$.
For $1\leq j\leq\ell$, put
$$
 F_j=\lambda_j(\Lambda)+S_j=W_{b+j-1}-(\operatorname{wls}(A)_{\ell+1-j}
     +\operatorname{wls}(B)_{\ell+1-j}).
$$
Note that $F_j$ is the total weight of the floor-to-floor edges crossing $b+j-1/2$. 
On floors $j<b$, the interior
divergence conditions are unchanged, and
$|\eta^j|+|\mu^j|=0$. 
At the floor $b$, every newly added
outgoing floor-to-floor edge crosses $b+1/2$, while
no new floor-to-floor edge enters $b$. Thus its divergence increases by $F_1$, giving
$$
 r_b-l_b+s_b-\di(b)
 =W_b-F_1
 =\sum_{w\geq1}w(a_{\ell,w}+b_{\ell,w})
  =|\eta^b|+|\mu^b|.
$$
For a floor $v=b+j$, with $1\leq j<\ell$,
we have $s_v=0$ and
$\di(v)=F_{j+1}-F_j$. Indeed, an edge crossing both $b+j-1/2$ and $b+j+1/2$ makes no contribution to this difference. From the construction of $\lf,\rf$, we have
$$
\begin{aligned}
 r_v-l_v+s_v-\di(v)
 &=(W_v-W_{v-1})-(F_{j+1}-F_j)\\
&=(\operatorname{wls}(A)_{\ell-j}
     +\operatorname{wls}(B)_{\ell-j})-(\operatorname{wls}(A)_{\ell+1-j}
     +\operatorname{wls}(B)_{\ell+1-j})\\
 &=\sum_{w\geq1}w(a_{\ell-j,w}+b_{\ell-j,w})
  =|\eta^v|+|\mu^v|.
\end{aligned}
$$
At the last floor, $s_M=0$ and $\di(M)=-F_\ell$; hence
\[
\begin{aligned}
 r_M-l_M+s_M-\di(M)
 &=(W_M-W_{M-1})+F_\ell\\
 &=d_b-(\operatorname{wls}(A)_{1}
     +\operatorname{wls}(B)_{1})\\
 &=\sum_{w\geq1}w\left(a_w-\sum_{i=1}^{\ell}a_{iw}
                         +b_w-\sum_{i=1}^{\ell}b_{iw}\right)\\
 &=|\eta^M|+|\mu^M|.
\end{aligned}
\]
If $\ell=0$, the interior diagram already has bottom capacity
$d_b$ at floor $M$, and the compatibility is immediate.
This reconstructs an
$(\eta,\mu)$-marked $\Delta$-floor diagram, whose cogenus is
$e+\delta(T)=\delta$.
It is obvious that the constructions are inverse on equivalence classes.

Assume that $e>0$ and $\ell>0$. Then the last side lengths of $\Delta^{(\ell)}$
satisfy
$$
d_r^{p_r}-\ell,\ d_l^{p_l}-\ell
\geq\delta+1-\ell\geq e+1.
$$
The top width, all other side lengths, and the differences between the internal corner heights are the same as those of $\Delta$. By
Lemma~\ref{lem:poly-defect}\textup{(2)},
$d_b^{(\ell)}=W_b\geq2\delta\geq e$.
Put $s=c_r^{p_r}-c_l^{p_l}$. If $c_r^{p_r}-c_l^{p_l}\leq0$, then
$$
d_b^{(\ell)}+c_r^{p_r}-c_l^{p_l}=d_b+(1-\ell)(c_r^{p_r}-c_l^{p_l})\geq d_b+c_r^{p_r}-c_l^{p_l}\geq2\delta.
$$
If $c_r^{p_r}-c_l^{p_l}\geq0$, then
$d_b^{(\ell)}+c_r^{p_r}-c_l^{p_l}\geq d_b^{(\ell)}\geq2\delta$.
The top-width inequality is unchanged, and all the inequalities in condition
\eqref{eq:poly-range} hold with $e$ in place of $\delta$.
\end{proof}
Now we count the compatible markings of a fixed $(\eta,\mu)$-admissible extended template $(\Lambda,A,B)$ and $\Delta$. Let $G$ be the terminal diagram of $(\Lambda,A,B)$ and $\Delta$.
Let $\Omega_{\Lambda,A,B}$ be the set of total orders on the set consisting of floors $M-\ell,\ldots,M$, the subdivision vertices of nonshort edges of $G$, and the $B$-vertices attached to floors $M-\ell,\ldots,M-1$ of $G$, compatible with the orientation of the terminal diagram $G$.
In particular, this finite set depends only on the translated extended template. In defining $\Omega_{\Lambda,A,B}$, we temporarily label the subdivision vertices of the nonshort edges and the displaced $B$-vertices, including vertices with identical incidence and weight data. These labels will be removed by the symmetry factor below.

For $\omega\in\Omega_{\Lambda,A,B}$, let $h_j(\omega)$ denote the number of its labelled vertices
between floors $M-\ell+j-1$ and $M-\ell+j$, and let $h_+(\omega)$ count those after floor $M$.
Let $m_\Lambda(x,y,w)$ be the number of edges $x\to y$ of weight $w$,
and put
$$
\rho_{\Lambda,A,B}
=\prod_{x<y,\,w}m_\Lambda(x,y,w)!\prod_{i,w}a_{iw}!b_{iw}!.
$$
For $r\geq0$, we use the falling and rising factorials
$$
(x)_{\underline r}=\prod_{k=0}^{r-1}(x-k),\qquad
(x)^{\overline r}=\prod_{k=0}^{r-1}(x+k),
$$
with empty products equal to $1$. 
The enumeration of terminal orders for relative plane curves and its
normalized polynomiality are given in
\cite[Section~3 and Lemma~1]{Block-2012}.
The following is a generalization of \cite[Section~3 and Lemma~1]{Block-2012}.

\begin{lemma}\label{lem:poly-boundary}
Let $\Delta=\Delta(\cf,\df)$ be an $h$-transverse polygon such that $\cf,\df$ are sufficiently large and spread out for $\delta>0$. 
Let $(\Lambda,A,B)$ be an $(\eta,\mu)$-admissible extended template with $\delta(\Lambda,A,B)\leq\delta$ {and length $\ell$, let $L=l(\mu)$, and let} $S_j$ be given by equation \eqref{Eqt:def-Sj}.
Define
\begin{equation}\label{eq:poly-boundary}
\begin{aligned}
\mathcal B_{(\Lambda,A,B),\delta}
={}&\frac{\mult_q(\Lambda)}{\rho_{\Lambda,A,B}}
\prod_{w\geq1}(a_w)_{\underline{\sum_{i\geq1}a_{iw}}}(b_w)_{\underline{\sum_{i\geq1}b_{iw}}}\\
&\times\sum_{\omega\in\Omega_{\Lambda,A,B}}
\frac{(L-\delta+1)^{\overline{\delta-\sum_{i,w\geq1}b_{iw}+h_+(\omega)}}}{h_+(\omega)!}
\prod_{j=1}^{\ell}\binom{S_j+h_j(\omega)}{h_j(\omega)}.
\end{aligned}
\end{equation}
Then the following statements hold.
\begin{enumerate}[$(i)$]
\item
Let $\operatorname{ord}_{(\Lambda,A,B)}$ be the number of compatible markings of $(\Lambda,A,B)$ and a fixed interior marked floor diagram, up to equivalence.
Then
for $L\geq\delta$, we have
$
\mult_q(\Lambda)\operatorname{ord}_{(\Lambda,A,B)}
=\frac{(L-\delta)!}{\prod_{w\geq1}b_w!}\,
\mathcal B_{(\Lambda,A,B),\delta}.
$
\item
Regarding $L,a_w,b_w$ and the polygon coordinates as independent indeterminates, $\mathcal B_{(\Lambda,A,B),\delta}$ is a polynomial over
$\qb[q^{\pm1}]$ in the variables
$
d_b, c_r^{p_r}-c_l^{p_l},  L, 
a_1, \ldots, a_\delta,  b_1,\ldots,b_\delta$.
Its degree in $L$ is at most $\delta$.
Moreover, after substituting the polygon closure relation for $d_b$, while keeping
$L$ and the contact variables independent of $\cf,\df$, it satisfies
$\deg_{\df}\mathcal B_{(\Lambda,A,B),\delta}\leq\delta(\Lambda,A,B)$,
$\deg_{\cf}\mathcal B_{(\Lambda,A,B),\delta}\leq\delta(\Lambda,A,B)$,
and $\deg_{(\cf,\df)}\mathcal B_{(\Lambda,A,B),\delta}\leq2\delta(\Lambda,A,B)$.
For the empty extended template,
$\mathcal B_{(\Lambda,A,B),\delta}=(L)_{\underline\delta}$.
\end{enumerate}
\end{lemma}

\begin{proof}
We adapt the terminal-order method of
\cite[Section~3 and Lemma~1]{Block-2012} to the case of an $h$-transverse polygon $\Delta$.

$(i)$.
Let $G$ be the terminal diagram of $(\Lambda,A,B)$, and fix $\omega\in\Omega_{\Lambda,A,B}$.
Then we have a total order on the set of floors $M-\ell,\ldots,M$, the subdivision vertices of nonshort edges of $G$, and the $B$-vertices attached to floors $M-\ell,\ldots,M-1$ of $G$.
To get a compatible marking, we only need to insert the remaining vertices into $\omega$ in a compatible way.
There are $S_j$ short edges connecting the floor $M-\ell+j-1$ to $M-\ell+j$ of $G$. 
There are $h_j(\omega)$ vertices between $M-\ell+j-1$ and $M-\ell+j$ according to $\omega$.
Distributing the $S_j$ short edge subdivision vertices among these $h_j(\omega)$ vertices gives us
$\binom{S_j+h_j(\omega)}{h_j(\omega)}$ choices, up to equivalence.

There are $L-\sum_{i,w\geq1}b_{iw}$ $B$-vertices attached to floor $M$, and
$b_w-\sum_{i\geq1}b_{iw}$ of them have weight $w$.
Their multiset orders can be interleaved with the given $h_+(\omega)$
ordered vertices in
\begin{center}
$
\binom{L-\sum_{i,w\geq1}b_{iw}+h_+(\omega)}{h_+(\omega)}
\frac{(L-\sum_{i,w\geq1}b_{iw})!}{\prod_{w\geq1}(b_w-\sum_{i\geq1}b_{iw})!}
$
\end{center} ways, up to equivalence.
All of these vertices precede the $A$-vertices.
Note that the $A$-vertices occupy the final positions, and the weights of the edges adjacent to them are non-decreasing.
For each weight $w$, we have $\frac{a_w!}{(a_w-\sum_{i\geq1}a_{iw})!\prod_{i\geq1}a_{iw}!}$ ways to arrange the $a_w$ $A$-vertices, up to equivalence.

The permutations of equally weighted parallel edges and of equally weighted $B$-vertices attached to the same floor act freely on $\Omega_{\Lambda,A,B}$. 
An orbit of this action on $\Omega_{\Lambda,A,B}$ has
$\prod_{x<y,\,w}m_\Lambda(x,y,w)!\prod_{i,w}b_{iw}!$ elements.
Summing the product of the above factors over the equivalence classes of $\omega$, we obtain the formula for $\operatorname{ord}_{(\Lambda,A,B)}$:
$$
\begin{aligned}
\operatorname{ord}_{(\Lambda,A,B)}
={}&\frac{1}{\rho_{\Lambda,A,B}}
\prod_{w\geq1}\frac{a_w!}{(a_w-\sum_{i\geq1}a_{iw})!}\\
&\times\sum_{\omega\in\Omega_{\Lambda,A,B}}
\frac{(L-\sum_{i,w\geq1}b_{iw}+h_+(\omega))!}
{h_+(\omega)!\prod_{w\geq1}(b_w-\sum_{i\geq1}b_{iw})!}
\prod_{j=1}^{\ell}\binom{S_j+h_j(\omega)}{h_j(\omega)}.
\end{aligned}
$$

{Let $L\geq\delta$. The equality} $\mult_q(\Lambda)\operatorname{ord}_{(\Lambda,A,B)}
=\frac{(L-\delta)!}{\prod_{w\geq1}b_w!}\,
\mathcal B_{(\Lambda,A,B),\delta}$ follows from the formula for $\operatorname{ord}_{(\Lambda,A,B)}$.
Indeed,
$$
\frac{a_w!}{(a_w-\sum_{i\geq1}a_{iw})!}=(a_w)_{\underline{\sum_{i\geq1}a_{iw}}},\qquad
\frac{b_w!}{(b_w-\sum_{i\geq1}b_{iw})!}=(b_w)_{\underline{\sum_{i\geq1}b_{iw}}},
$$
and, writing $h=h_+(\omega)$, 
$$
\frac{(L-\sum_{i,w\geq1}b_{iw}+h)!}{(L-\delta)!}
=(L-\delta+1)^{\overline{\delta-\sum_{i,w\geq1}b_{iw}+h}}.
$$
The order $\delta-\sum_{i,w\geq1}b_{iw}+h$ is nonnegative since $\sum_{i,w\geq1}b_{iw}\leq\delta(B)\leq\delta$.
Hence, 
$$
\mult_q(\Lambda)\frac{\prod_{w\geq1}b_w!}{(L-\delta)!}\operatorname{ord}_{(\Lambda,A,B)}=\mathcal B_{(\Lambda,A,B),\delta}.
$$

$(ii)$. We first give some basic combinatorial estimates that will be used.
Recall that the labelled vertices in $\omega$ are the disjoint union of the
$|E(\Lambda)|$ edge-subdivision vertices and the $\sum_{i,w\geq1}b_{iw}$ displaced
$B$-vertices.
Moreover, none of them is before floor $M-\ell$, and only $B$-vertices can be after floor $M$.
Hence, we have
$$
\begin{gathered}
\sum_{j=1}^{\ell}h_j(\omega)+h_+(\omega)
=|E(\Lambda)|+\sum_{i,w\geq1}b_{iw}
\leq\delta(\Lambda)+\delta(B)
\leq\delta(\Lambda,A,B),\\
0\leq h_+(\omega)\leq\sum_{i,w\geq1}b_{iw}\leq\delta(B)\leq\delta.
\end{gathered}
$$
Here each edge of $\Lambda$ contributes at least one to its cogenus because
it is not short, and $iw\geq1$ for every entry $b_{iw}$.

Let $r=\delta-\sum_{i,w\geq1}b_{iw}+h_+(\omega)$, and we have $0\leq r\leq\delta$.
The only $L$-dependent factor in the summand of $\mathcal B_{(\Lambda,A,B),\delta}$ in equation
\eqref{eq:poly-boundary} is the rising factorial
$(L-\delta+1)^{\overline r}$. This factor is a polynomial of degree
$r$ in $L$. Hence the $L$-degree is at most $\delta$.

For a fixed extended template $(\Lambda,A,B)$, the orders $\sum_{i\geq1}a_{iw},\sum_{i\geq1}b_{iw}$ of the falling factorials,
the finite set $\Omega_{\Lambda,A,B}$, and the factor $\rho_{\Lambda,A,B}$ are
constant.
Each binomial factor $\binom{S_j+h_j(\omega)}{h_j(\omega)}$ is a polynomial of degree 
$h_j(\omega)$ in $S_j$.
From equation \eqref{Eqt:def-Sj}, $S_j$ is affine linear in $d_b$ and $c_r^{p_r}-c_l^{p_l}$.
Moreover, a nonzero entry $a_{iw}$ or $b_{iw}$ contributes at least $w$ to
the template cogenus, so all columns with $w>\delta$ vanish.
Since $\mult_q(\Lambda)\in\zb[q^{\pm1}]$, we obtain the stated
polynomiality over $\qb[q^{\pm1}]$ and the bound on the $L$-degree.

Using the equality $d_b=d_t+\sum_i c_r^i d_r^i-\sum_j c_l^j d_l^j$,
we know that $S_j$ has degree at most one in each of $\cf$ and $\df$, and at most two jointly. 
All other factors are independent of these polygon variables.
In particular, the degree of $\prod_{j=1}^{\ell}\binom{S_j+h_j(\omega)}{h_j(\omega)}$ in each of $\cf$ and $\df$ is at most $\sum_{j=1}^{\ell}h_j(\omega)$, and the degree in both $\cf,\df$ is at most $2\sum_{j=1}^{\ell}h_j(\omega)$.
The preceding vertex-count inequality proves the asserted bounds for
their finite sum.
These are bounds before imposing the enumerative relations
$L=\sum_w b_w$ and $\sum_w w(a_w+b_w)=d_b$.

For the empty template, $\ell=0$, $A=B=0$, and there is a single structural
order with $h_+=0$.
Both its symmetry factor and refined multiplicity equal $1$.
Formula~\eqref{eq:poly-boundary} becomes
$\mathcal B_{(\varnothing,0,0),\delta}
=(L-\delta+1)^{\overline\delta}=(L)_{\underline\delta}$.
\end{proof}

\subsubsection{Relative polynomiality}
\label{subsec:relative-poly}
Let $\Delta=\Delta(\cf,\df)$ be an $h$-transverse polygon where $\cf,\df$ are sufficiently large and spread out for $\delta>0$. Fix two partitions $\eta,\mu$ with $|\eta|+|\mu|=d_b$.
Note that we have proved $\area(\Delta)=2\sum_{j=0}^{M-1}W_j-d_t+d_b$ before. {Thus the genus is given by}
$$
\begin{aligned}
g_\delta(\cf,\df)
&:=|\Delta^\circ\cap\zb^2|-\delta
=\sum_{j=0}^{M-1}W_j-d_t-M+1-\delta\\
&=\frac{\area(\Delta)-2M-d_t-d_b}{2}+1-\delta.
\end{aligned}
$$
From equation \eqref{eq:dadi-theta}, the area is homogeneous of length
degree two and has direction degree at most one.

\begin{theorem}\label{thm:poly-fixed}
Fix $\cf$ and $\delta\geq0$. For any $\Delta=\Delta(\cf,\df)$ satisfying condition
\eqref{eq:poly-range} with $\delta>0$ and any pair of partitions $\eta,\mu$ satisfying
$|\eta|+|\mu|=d_b$, $l(\mu)\geq\delta$, 
there is a polynomial
$$
R_\delta^\cf\in\qb[q^{\pm1}][\df,L,a_1,\ldots,a_\delta,b_1,\ldots,b_\delta],
$$
where $L=l(\mu)$, $a_w=m_w(\eta)$ and $b_w=m_w(\mu)$, 
such that one has
\begin{equation}\label{eq:poly-relative}
N_{\floor}^{\Delta,g_\delta}(\eta,\mu;q)
=\left(\prod_{w\geq1}[w]_q^{b_w}\right)
\frac{(l(\mu)-\delta)!}{\prod_{w\geq1}b_w!}\,
R_\delta^\cf(\df,L,a_1,\ldots,a_\delta,b_1,\ldots,b_\delta;q).
\end{equation}
Moreover, the polynomial $R_\delta^\cf$ has degree $2\delta$ in $\df$ when we regard the contact variables $a_w,b_w$ as independent coefficient parameters, and its leading coefficient is
\begin{equation}\label{eq:poly-relative-leading}
[R_\delta^\cf]_{\df,2\delta}
=\frac{(l(\mu))_{\underline\delta}}{\delta!}
\left((q+4+q^{-1})\frac{\area(\Delta)}2\right)^\delta.
\end{equation}
For $\delta=0$, one has $R_0^\cf=1$. 
\end{theorem}

\begin{proof}
For an extended template $T=(\Lambda,A,B)$, abbreviate its cogenus by
$\delta_T=\delta(\Lambda,A,B)$ and its length by $\ell_T=l(\Lambda)$. Recall that
$\Delta^{(\ell_T)}$ is the polygon obtained by deleting the bottom $\ell_T$ rows of $\Delta$.
$\df^{(\ell_T)}$ is the length sequence of the polygon $\Delta^{(\ell_T)}$.
Let
$$
R_\delta^\cf
=\sum_{\substack{T=(\Lambda,A,B)\\\delta_T\leq\delta}}
\mathcal K_{\delta-\delta_T}^\cf(\df^{(\ell_T)};q)
\mathcal B_{T,\delta},
$$
where $\mathcal K_{\delta-\delta_T}^\cf(\df^{(\ell_T)};q)$ is the polynomial in Lemma \ref{lem:poly-interior}(1), and $\mathcal B_{T,\delta}$ is given in equation \eqref{eq:poly-boundary}.
The set of extended templates $T=(\Lambda,A,B)$ with $\delta_T\leq\delta$ is finite and depends only on $\delta$. Indeed,
from equation \eqref{eq:poly-term-bound}, $\ell_T\leq\delta$, so the length of $\Lambda$ is bounded.
Since each edge of $T$ contributes at least one to the cogenus and has weight at most $\delta+1$, 
the number of edges in $\Lambda$ is bounded.
Nonzero matrix entries in $A$ and $B$ satisfy $iw(a_{iw}+b_{iw})\leq\delta$.
The truncation $\df\mapsto\df^{(\ell_T)}$ is affine.
The defining polynomial $\mathcal B_{T,\delta}$ also depends on $L=l(\mu)$.
Thus Lemmas~\ref{lem:poly-interior}
and~\ref{lem:poly-boundary} show that $R_\delta^\cf$ belongs
to the asserted ring.

From Lemma~\ref{lem:poly-decomp}, there is a bijection between 
the set of $(\eta,\mu)$-marked $\Delta$-floor diagrams and the set of 
triples consisting of {an extended template $T$, an interior marked floor diagram, and a compatible marking}. 
Moreover, the interior marked floor diagrams in the triples are counted by 
$\mathcal K_{\delta-\delta_T}^\cf$ on the truncated polygon.
After factoring out the common free-contact product $\prod_{w\geq1}[w]_q^{b_w}$, 
the refined edge multiplicity is the product of the multiplicities of the two pieces. 
From Lemma \ref{lem:poly-boundary}(i), the remaining contribution of $T$ is
$(l(\mu)-\delta)!\mathcal K_{\delta-\delta_T}^\cf
\mathcal B_{T,\delta}/(\prod_{w\geq1}b_w!)$.
Summing over $T$, we obtain the equation \eqref{eq:poly-relative}.
When $\delta=0$, the only terminal template is empty and
$\mathcal K_0=\mathcal B_{T,0}=1$, giving $R_0^\cf=1$.

For the empty extended template, $\delta_T=\ell_T=0$ and
$\mathcal B_{T,\delta}=(l(\mu))_{\underline\delta}$.
By Lemma~\ref{lem:poly-interior}\textup{(3)}, its contribution to
$R_\delta^\cf$ {has a part of length degree $2\delta$ equal to} $\frac{(l(\mu))_{\underline\delta}}{\delta!}
\left((q+4+q^{-1})\frac{\area(\Delta)}2\right)^\delta$.
For any nonempty $T$, one has $\delta_T\geq1$, and from
Lemmas~\ref{lem:poly-interior} and~\ref{lem:poly-boundary} we obtain
$$
\deg_{\df}\bigl(\mathcal K_{\delta-\delta_T}^\cf
(\df^{(\ell_T)};q)\,\mathcal B_{T,\delta}\bigr)
\leq 2(\delta-\delta_T)+\delta_T
=2\delta-\delta_T<2\delta.
$$
Thus no nonempty extended template contributes to the degree $2\delta$ homogeneous part. The area is a nonzero homogeneous quadratic
polynomial in the length variables,
and the coefficient $(l(\mu))_{\underline\delta}(q+4+q^{-1})^\delta$
is nonzero. Hence, $[R_\delta^\cf]_{\df,2\delta}$ is nonzero and the polynomial $R_\delta^\cf$ has length degree exactly $2\delta$.
\end{proof}

\begin{theorem}\label{thm:poly-varying}
Fix $p_r,p_l\geq1$ and $\delta\geq0$.
For any $\Delta=\Delta(\cf,\df)$ {satisfying conditions}
\eqref{eq:poly-range} and \eqref{eq:poly-gap-range} with $\delta>0$ and any pair of partitions $\eta,\mu$ satisfying
$|\eta|+|\mu|=d_b$, $l(\mu)\geq\delta$, 
there is a polynomial
$$
R_\delta\in\qb[q^{\pm1}][\cf,\df,L,a_1,\ldots,a_\delta,b_1,\ldots,b_\delta]
$$
where $L=l(\mu)$, $a_w=m_w(\eta)$ and $b_w=m_w(\mu)$,
such that one has
$$
N_{\floor}^{\Delta,g_\delta}(\eta,\mu;q)
=\left(\prod_{w\geq1}[w]_q^{b_w}\right)
\frac{(l(\mu)-\delta)!}{\prod_{w\geq1}b_w!}\,
R_\delta(\cf,\df,L,a_1,\ldots,a_\delta,b_1,\ldots,b_\delta;q).
$$
One may choose $R_\delta$ with
$$
\deg_{\cf}R_\delta=\delta,\qquad
\deg_{\df}R_\delta=2\delta,\qquad
\deg_{(\cf,\df)}R_\delta=3\delta,
$$
and {with the highest length degree part given in equation} \eqref{eq:poly-relative-leading}.
For $\delta=0$, $R_0=1$.
\end{theorem}

\begin{proof}
For an extended template $T=(\Lambda,A,B)$, denote its cogenus by
$\delta_T=\delta(\Lambda,A,B)$ and its length by $\ell_T=l(\Lambda)$.
Let
$$
R_\delta
=\sum_{\substack{T=(\Lambda,A,B)\\\delta_T\leq\delta}}
\mathcal K_{\delta-\delta_T}(\cf,\df^{(\ell_T)};q)
\mathcal B_{T,\delta},
$$
where $\mathcal K_{\delta-\delta_T}(\cf,\df^{(\ell_T)};q)$ is the universal polynomial in Lemma \ref{lem:poly-interior}(2), and $\mathcal B_{T,\delta}$ is given in equation \eqref{eq:poly-boundary}.
The gap hypothesis for $\delta$ in equation \eqref{eq:poly-gap-range} implies that
$$
c_r^i-c_r^{i+1}\geq\delta-\delta_T+1\quad(i<p_r),\qquad
c_l^{j+1}-c_l^j\geq \delta-\delta_T+1\quad(j<p_l).
$$
The truncation leaves the direction vector $\cf$ unchanged.
Thus, the same argument as that used in the proof of equation \eqref{eq:poly-relative} implies the required equality involving $N_{\floor}^{\Delta,g_\delta}(\eta,\mu;q)$.
For an extended template of cogenus $\delta_T$, the separate degree bounds are
$\delta$ in $\cf$ and $2\delta-\delta_T$ in $\df$, while the joint
degree is at most $3\delta-\delta_T$.
Only the empty template contributes in length degree $2\delta$ or
joint degree $3\delta$.
Its leading term in equation \eqref{eq:poly-relative-leading} has a nonzero
component of bidegree $(\delta,2\delta)$, proving all three equalities.
\end{proof}

\begin{proposition}[{\cite[Theorem 4.2 and Remark 4.3]{bg2016}}]
\label{prop:poly-ordinary}
Let $\eta=\varnothing$ and {$\mu=(1^{d_b})$, then} the following statements hold.
\begin{itemize}
\item Fix $\cf$ and $\delta\geq0$. For any $\Delta=\Delta(\cf,\df)$ satisfying condition
\eqref{eq:poly-range} with $\delta>0$, 
there is a polynomial $P_\delta^\cf\in\qb[q^{\pm1}][\df]$
such that
$$
N_{\floor}^{\Delta,g_\delta}(\varnothing,1^{d_b};q)=P_\delta^\cf(\df;q).
$$
\item Fix $p_r,p_l\geq1$ and $\delta\geq0$.
For any $\Delta=\Delta(\cf,\df)$ {satisfying conditions}
\eqref{eq:poly-range} and \eqref{eq:poly-gap-range} with $\delta>0$,
there is a polynomial $P_\delta\in\qb[q^{\pm1}][\cf,\df]$ such that
$$
N_{\floor}^{\Delta,g_\delta}(\varnothing,1^{d_b};q)=P_\delta(\cf,\df;q).
$$
\item $P_\delta^\cf(\df;q)$ and $P_\delta(\cf,\df;q)$ have degree $2\delta$ in $\df$. The polynomial $P_\delta(\cf,\df;q)$ has degree $\delta$ in $\cf$ and joint degree $3\delta$.
Their highest length degree part is
\begin{equation}\label{eq:poly-ordinary-leading}
[P_\delta]_{\df,2\delta}
=\frac1{\delta!}
\left((q+4+q^{-1})\frac{\area(\Delta)}2\right)^\delta.
\end{equation}
{The same formula holds for fixed directions, and} $P_0^\cf=P_0=1$.
\end{itemize}
\end{proposition}

\begin{proof}
Take $\eta=\varnothing$, $\mu=(1^{d_b})$, and hence
$a_w=0$ for every $w$, $b_1=L=d_b\geq\delta$, and $b_w=0$ for $w>1$.
From Lemma \ref{lem:poly-decomp}, after splitting the $(\eta,\mu)$-marked $\Delta$-floor diagrams, only extended templates $T=(\Lambda,A,B)$ satisfying $A=0$ and with $B$ supported on the first column remain.
In this case, the factor in the right-hand side of equation \eqref{eq:poly-relative} becomes
$$
\left(\prod_{j=1}^{l(\mu)}[\mu_j]_q\right)
	\frac{(L-\delta)!}{\prod_{w\geq1}b_w!}=\frac{(d_b-\delta)!}{d_b!}=\frac{1}{(d_b)_{\underline\delta}}.
$$

In the following, we show that $R_\delta^\cf$ is divisible by $(d_b)_{\underline\delta}$.
Let $k=\sum_{i\geq1}b_{i1}\leq\delta$.
Recall that 
$$
R_\delta^\cf
=\sum_{\substack{T=(\Lambda,A,B)\\\delta_T\leq\delta}}
\mathcal K_{\delta-\delta_T}^\cf(\df^{(\ell_T)};q)
\mathcal B_{T,\delta},
$$
where $T=(\Lambda,A,B)$ is an extended template with $A=0$ and $B$ supported on the first column.
From equation \eqref{eq:poly-boundary}, every polynomial {$\mathcal B_{T,\delta}$ that contributes to $R_\delta^\cf$ has} a factor $(d_b)_{\underline k}
\frac{(d_b-\delta+1)^{\overline{\delta-k+h_+(\omega)}}}{h_+(\omega)!}$.
Note that
$$
(d_b)_{\underline k}
\frac{(d_b-\delta+1)^{\overline{\delta-k+h_+(\omega)}}}{h_+(\omega)!}
=(d_b)_{\underline\delta}\binom{d_b-k+h_+(\omega)}{h_+(\omega)}.
$$
Therefore, $R_\delta^\cf$ is divisible termwise by $(d_b)_{\underline\delta}$, and we have
$P_\delta^\cf=R_\delta^\cf/(d_b)_{\underline\delta}
\in\qb[q^{\pm1}][\df]$, 
$P_\delta=R_\delta/(d_b)_{\underline\delta}
\in\qb[q^{\pm1}][\cf,\df]$.
From Theorem \ref{thm:poly-fixed} and Theorem \ref{thm:poly-varying},
$N_{\floor}^{\Delta,g_\delta}(\varnothing,1^{d_b};q)=P_\delta^\cf(\df;q)$
under condition \eqref{eq:poly-range}, and
$N_{\floor}^{\Delta,g_\delta}(\varnothing,1^{d_b};q)=P_\delta(\cf,\df;q)$
under conditions \eqref{eq:poly-range} and \eqref{eq:poly-gap-range}.

After cancellation, the count of markings of extended templates contributes, up to its
	constant multiplicity and symmetry factor,
	\[
	\binom{d_b-\sum_{i\geq1}b_{i1}+h_+(\omega)}{h_+(\omega)}
	\prod_{j=1}^{\ell}\binom{S_j+h_j(\omega)}{h_j(\omega)}.
	\]
	Its degree in $\df$ or in $\cf$ is at most $h_+(\omega)+\sum_{j=1}^\ell h_j(\omega)=|E(\Lambda)|+\sum_{i\geq1}b_{i1}\leq \delta(\Lambda,A,B)$, and its joint degree is at most $2\delta(\Lambda,A,B)$.
    Combining this with the interior bounds from Lemma \ref{lem:poly-interior}, we find that the term in the polynomials \(P_\delta^\cf\) and \(P_\delta\) indexed by the extended template \((\Lambda,A,B)\) has length degree at most \(2\delta-\delta(\Lambda,A,B)\) and joint degree at most \(3\delta-\delta(\Lambda,A,B)\). 
Only the empty extended template contributes to degree \(2\delta\) or \(3\delta\), so the leading term lies solely in \(K^\cf_\delta\) and \(K_\delta\). The leading terms $[P_\delta]_{\df,2\delta}$ and $[P_\delta^\cf]_{\df,2\delta}$ follow from equation \eqref{eq:poly-interior-leading}.
\end{proof}

\section{Polynomiality for Gromov-Witten invariants}\label{sec:ordinary-polynomiality}
In this section, we prove the polynomiality of the combined invariants.

\subsection{Polynomiality for logarithmic invariants}
\label{sec:poly-log}

We show the polynomiality of logarithmic Gromov-Witten invariants using the polynomiality results in Section \ref{subsec:relative-poly}.

Fix $\delta,m\geq0$.  
Let $\Delta=\Delta(\cf,\df)$ be an $h$-transverse polygon with $d_b>0$ and
let $\eta,\mu$ be two partitions satisfying $|\eta|+|\mu|=d_b$.
Put
$$
\begin{aligned}
g'&=|\Delta^\circ\cap\zb^2|-\delta,&g(\cf,\df)&=g'+m,\\
\beta(\cf,\df)&=\beta_\Delta,&
n(\cf,\df;l(\mu))&=2M+d_t+l(\mu)+g'-1.
\end{aligned}
$$
The polygon class $\beta_\Delta$ is specified on each toric surface by
$\beta_\Delta\cdot D_\rho=d_\rho$.  When $\cf$ is fixed, this class varies linearly with the edge lengths of $\Delta$.
For fixed $\cf$, we suppress $\cf$ from $g,\beta,n$.

Let $s(u)=\frac{2\sin(u/2)}u$. Its Taylor expansion is
$$
s(u)=1-\frac{u^2}{24}+\frac{u^4}{1920}+O(u^6)
\in 1+u^2\qb[[u^2]].
$$
The constant term in this Taylor expansion is $1$, so the formal logarithm
$$
\log s(u)
=\sum_{k\geq1}\frac{(-1)^{k+1}}k\bigl(s(u)-1\bigr)^k
\in u^2\qb[[u^2]]
$$
is well-defined. Since $s(u)$ is even, there are unique rational
numbers $\gamma_r$ such that
$
\log s(u)=-\sum_{r\geq1}\gamma_r u^{2r}$.
In particular, $\gamma_1=1/24$ and $\gamma_2=1/2880$.
For a parameter $C$, we use the formal
convention $\exp\bigl(C\log s(u)\bigr)=s(u)^C$.
Let
$$
\begin{aligned}
S_{2r}&=\sum_{w\geq1}w^{2r}(a_w+b_w),\qquad r\geq1,\\
C_\delta&=2g'-2+2M+d_t=\area(\Delta)-d_b-2\delta.
\end{aligned}
$$
As in Section~\ref{sec:log-polynomiality}, $l(\mu),a_w,b_w,S_{2r}$ are
coefficient parameters for polynomiality statements.
Their relations to the contact partitions are imposed only upon evaluation.

\begin{theorem}\label{thm:log-poly-fix}\label{thm:log-poly-var}
Fix $\delta,m\geq0$. Let $L=l(\mu)$,
$\mathbf a=(a_1,\ldots,a_\delta)$,
$\mathbf b=(b_1,\ldots,b_\delta)$, and
$\mathbf S=(S_2,\ldots,S_{2m})$, where $a_i=m_i(\eta)$ and $b_j=m_j(\mu)$.
Regarding $L,\af,\bbf,\Sff$ as parameters, we have the following results.
\begin{itemize}
\item
Fix $p_r,p_l\geq1$ and $\cf\in\zb^{p_r+p_l}$. There is a polynomial
\[
\rl_{\delta,m}^\cf\in
\qb[\df,L,a_1,\ldots,a_\delta,b_1,\ldots,b_\delta,S_2,\ldots,S_{2m}]
\]
that satisfies the following identity for every $h$-transverse polygon $\Delta=\Delta(\cf,\df)$ satisfying condition \eqref{eq:poly-range} when $\delta>0$, and every pair of partitions $\eta,\mu$ with $|\eta|+|\mu|=d_b$ and $L\geq\delta$:

\begin{equation}\label{eq:log-poly}
N_{g(\df),\beta(\df),n(\df;L)}^{X_\cf|\partial X_\cf,m}(\eta,\mu)
=\left(\prod_{w\geq1}w^{b_w}\right)
\frac{(L-\delta)!}{\prod_{w\geq1}b_w!}\,
\rl_{\delta,m}^\cf(\df,L,\mathbf a,\mathbf b,\mathbf S).
\end{equation}
Moreover, the degree of $\rl_{\delta,m}^\cf$ in $\df$ is
$2(\delta+m)$, and its highest length degree part is
$$
[\rl_{\delta,m}^\cf]_{\df,2(\delta+m)}
=\frac{(-1)^m3^\delta(L)_{\underline\delta}}
{24^m m!\delta!}\,\area(\Delta)^{\delta+m}.
$$
\item
Fix $p_r,p_l\geq1$. There is a polynomial
\[
\rl_{\delta,m}\in
\qb[\cf,\df,L,a_1,\ldots,a_\delta,b_1,\ldots,b_\delta,S_2,\ldots,S_{2m}]
\]
that satisfies the following identity for every $h$-transverse polygon $\Delta=\Delta(\cf,\df)$ {satisfying conditions} \eqref{eq:poly-range} and \eqref{eq:poly-gap-range} when $\delta>0$, and every pair of partitions $\eta,\mu$ with $|\eta|+|\mu|=d_b$ and $L\geq\delta$:

\begin{equation}\label{eq:log-poly-var}
N_{g(\cf,\df),\beta(\cf,\df),n(\cf,\df;L)}^{X_\cf|\partial X_\cf,m}(\eta,\mu)
=\left(\prod_{w\geq1}w^{b_w}\right)
\frac{(L-\delta)!}{\prod_{w\geq1}b_w!}\,
\rl_{\delta,m}(\cf,\df,L,\mathbf a,\mathbf b,\mathbf S).
\end{equation}
One may choose $\rl_{\delta,m}$ with degree $\delta+m$ in
$\cf$, degree $2(\delta+m)$ in $\df$, and joint degree
$3(\delta+m)$.
Its homogeneous part of length degree $2(\delta+m)$ is given by the formula in the first statement, with $\cf$ regarded as variables.
\end{itemize}
\end{theorem}

\begin{proof}
Let $R$ be the polynomial $R_\delta^\cf$ constructed in Theorem~\ref{thm:poly-fixed}, or $R_\delta$ constructed in Theorem~\ref{thm:poly-varying}. Put $n=n(\cf,\df;L)$.
Let
$$
\ell=|\partial\Delta\cap\zb^2|-d_b+l(\eta)+L
=2M+d_t+l(\eta)+L,
\qquad \kappa=2g'-2+\ell.
$$
For $\delta>0$, condition \eqref{eq:poly-range} gives $M\geq\delta+1$, and Lemma~\ref{lem:poly-defect}(2) gives
$$
g'=\sum_{j=1}^{M-1}(W_j-1)-\delta
\geq\delta(2\delta-1)-\delta\geq0.
$$
For $\delta=0$, the polygon has $M>0$ and $g'=|\Delta^\circ\cap\zb^2|\geq0$. Hence $\delta<M$ and $n=2M+d_t+L+g'-1\geq1$ in both cases. By Theorem~\ref{thm:log-trop-total}, we have
$$
\begin{aligned}
u^\kappa\sum_{v\geq0}
N_{g'+v,\beta_\Delta,n}^{X_\cf|\partial X_\cf,v}(\eta,\mu)
u^{2v}=\prod_{w\geq1}
\left(\frac{[w]_{q}}{w}\right)^{a_w}\cdot
N_{\mathrm{trop}}^{\Delta,g'}(\eta,\mu;q)
\bigl(2\sin(u/2)\bigr)^\kappa,
\end{aligned}
$$
where $q=e^{iu}$.
Note that $\delta<M$ from the assumption,
and $n=2M+d_t+L+g'-1\geq1$.

Theorem \ref{thm:floor-trop} identifies the
tropical count with $N_{\floor}^{\Delta,g'}(\eta,\mu;q)$.
From the corresponding formula in Theorem \ref{thm:poly-fixed} or Theorem~\ref{thm:poly-varying},
the right-hand side of the above equation is
$$
\begin{aligned}
&\prod_{w\geq1}
\left(\frac{[w]_{q}}{w}\right)^{a_w}
\prod_{w\geq1}[w]_{q}^{b_w}
\frac{(L-\delta)!}{\prod_{w\geq1}b_w!}\,
R(q)\bigl(2\sin(u/2)\bigr)^\kappa.
\end{aligned}
$$
The quantum integer satisfies
\[
[w]_{q}
=\frac{\sin(wu/2)}{\sin(u/2)}
=w\,\frac{s(wu)}{s(u)},
\qquad 2\sin(u/2)=u\,s(u).
\]
Substituting both identities and collecting all powers of $s(u)$
gives
$$
\begin{aligned}
&u^\kappa\sum_{v\geq0}
N_{g'+v,\beta_\Delta,n}^{X_\cf|\partial X_\cf,v}(\eta,\mu)
u^{2v}\\
&=u^\kappa\left(\prod_{w\geq1}w^{b_w}\right)
\frac{(L-\delta)!}{\prod_{w\geq1}b_w!}\,
s(u)^{\kappa-\sum_w(a_w+b_w)}
\prod_{w\geq1}s(wu)^{a_w+b_w}R(e^{iu}).
\end{aligned}
$$
Since $\sum_w a_w=l(\eta)$ and $\sum_w b_w=L$, we have
$\kappa-l(\eta)-L=2g'-2+2M+d_t=C_\delta$.

Cancelling the common monomial $u^\kappa$ yields
\begin{equation}\label{eq:poly-log-series}
\begin{aligned}
\sum_{v\geq0}
N_{g'+v,\beta_\Delta,n}^{X_\cf|\partial X_\cf,v}(\eta,\mu)
u^{2v}
=\left(\prod_{w\geq1}w^{b_w}\right)
\frac{(L-\delta)!}{\prod_{w\geq1}b_w!}\,
 s(u)^{C_\delta}
\prod_{w\geq1}s(wu)^{a_w+b_w}R(e^{iu}).
\end{aligned}
\end{equation}
Equation~\eqref{eq:poly-log-series} is an identity of formal power series. 
To extract the coefficient of $u^{2m}$ on its right-hand side, we use $\log s(u)=-\sum_{r\geq1}\gamma_r u^{2r}$.

Since $\eta,\mu$ have finitely many parts, $\prod_{w\geq1}s(wu)^{a_w+b_w}$ is a finite product.
We have
$$
\begin{aligned}
\log\left(s(u)^{C_\delta}
\prod_{w\geq1}s(wu)^{a_w+b_w}\right)
&=C_\delta\log s(u)
 +\sum_{w\geq1}(a_w+b_w)\log s(wu)\\
&=-\sum_{r\geq1}\gamma_r\left(
 C_\delta+\sum_{w\geq1}w^{2r}(a_w+b_w)\right)u^{2r}\\
&=-\sum_{r\geq1}\gamma_r(C_\delta+S_{2r})u^{2r}.
\end{aligned}
$$
Hence,
$
 s(u)^{C_\delta}\prod_{w\geq1}s(wu)^{a_w+b_w}
 =\exp\left(-\sum_{r\geq1}
 \gamma_r(C_\delta+S_{2r})u^{2r}\right)$.
Define the polynomials by coefficient extraction:
$$
\begin{aligned}
\rl_{\delta,m}^\cf
&:=[u^{2m}]\left[
\exp\left(-\sum_{r\geq1}\gamma_r(C_\delta+S_{2r})u^{2r}\right)
R_\delta^\cf(\df,L,\mathbf a,\mathbf b;e^{iu})\right],\\
\rl_{\delta,m}
&:=[u^{2m}]\left[
\exp\left(-\sum_{r\geq1}\gamma_r(C_\delta+S_{2r})u^{2r}\right)
R_\delta(\cf,\df,L,\mathbf a,\mathbf b;e^{iu})\right].
\end{aligned}
$$
From equation \eqref{eq:poly-log-series}, we have the equalities \eqref{eq:log-poly} and \eqref{eq:log-poly-var}.

We now show the polynomiality of $\rl_{\delta,m}^\cf$ and $\rl_{\delta,m}$. 
The template constructions of $R_\delta^\cf$ and $R_\delta$ use symmetric refined multiplicities, so $R(q)=R(q^{-1})$. Consequently, $R(e^{iu})$ is an even formal power series in $u$.
 Write
$$
\begin{aligned}
E(u)=&\exp\left(-\sum_{r\geq1}
\gamma_r(C_\delta+S_{2r})u^{2r}\right)
=\sum_{k\geq0}e_k u^{2k},\\
&R(e^{iu})=\sum_{j\geq0}\rho_j u^{2j}.
\end{aligned}
$$
If $R=R_\delta^\cf$ (resp. $R=R_\delta$), each $\rho_j$ is a polynomial
in the variables $\df$, $L$, $\af$, $\bbf$ (resp. $\cf$, $\df$, $L$, $\af$, $\bbf$) with the same {upper bounds on the degree} in $\df$ (resp. $\cf$, $\df$) as $R$.
In particular, $\rho_0=R(1)$.

The coefficient $e_k$ is a polynomial in $C_\delta,S_2,\ldots,S_{2k}$ of degree at most $k$ in $C_\delta$.
By its definition and the area formula in Section~\ref{subsec:relative-poly}, $C_\delta$ is polynomial in $\cf,\df$. Since $\gamma_1=1/24$, the term of $e_k$ of highest degree in $C_\delta$ is
$
\frac{1}{k!}\left(-\gamma_1C_\delta\right)^k
=\frac{(-1)^kC_\delta^k}{24^kk!}$.
Hence, $[u^{2m}]\bigl(E(u)R(e^{iu})\bigr)=\sum_{j=0}^m e_{m-j}\rho_j$
lies in $\qb[\df,L,\af,\bbf,\Sff]$ when $R=R_\delta^\cf$, and in $\qb[\cf,\df,L,\af,\bbf,\Sff]$ when $R=R_\delta$.

The polynomial $C_\delta$ has degree at most $2$ in $\df$, 
degree at most $1$ in $\cf$, and joint degree at most $3$.
For $R$, the degree bounds in $\df$, in $\cf$, and in $\cf,\df$  are $2\delta$, $\delta$, and $3\delta$, respectively.
Hence, the summand $e_{m-j}\rho_j$ in $[u^{2m}]\bigl(E(u)R(e^{iu})\bigr)$ has degree at most $2\delta+2(m-j)$ in $\df$, degree at most $\delta+m-j$ in $\cf$, and joint
degree at most $3(\delta+m-j)$.
Only the summand with $j=0$ can contribute to length degree $2(\delta+m)$ or, when the directions vary, to joint degree $3(\delta+m)$.
The homogeneous part of $C_\delta$ of length degree two is $\area(\Delta)$, so
$[e_m]_{\df,2m}=\frac{(-1)^m}{24^m m!}\area(\Delta)^m$.
By equation \eqref{eq:poly-relative-leading}, the term of $\rho_0=R(1)$ with highest degree in $\df$ is
$[R]_{\df,2\delta}\big|_{q=1}
=\frac{3^\delta(L)_{\underline\delta}}{\delta!}
\area(\Delta)^\delta$.
Hence, the coefficient of the term with highest degree in $\df$ of 
$[u^{2m}]\bigl(E(u)R(e^{iu})\bigr)$ is
$$
[\rl_{\delta,m}]_{\df,2(\delta+m)}=[\rl_{\delta,m}^\cf]_{\df,2(\delta+m)}
=\frac{(-1)^m3^\delta(L)_{\underline\delta}}
{24^m m!\delta!}\,\area(\Delta)^{\delta+m}.
$$
This homogeneous part is nonzero. With variable directions, 
it also has a nonzero component of bidegree $(\delta+m,2(\delta+m))$ in $(\cf,\df)$. 
No summand with $j>0$ can have this bidegree, so the required exact degree statements hold.
\end{proof}

\begin{corollary}\label{thm:log-simple-fix}\label{thm:log-simple-var}
{Fix $\delta,m\geq0$, then} the following statements hold.
\begin{itemize}
\item 
Fix $p_r,p_l\geq1$ and $\cf\in\zb^{p_r+p_l}$. There is a polynomial $\mathcal Q_{\delta,m}^\cf\in\qb[\df]$ such that, for every $h$-transverse polygon $\Delta=\Delta(\cf,\df)$ satisfying condition \eqref{eq:poly-range} when $\delta>0$,
$$
N_{g(\df),\beta(\df),n(\df)}^{X_\cf|\partial X_\cf,m}
=\mathcal Q_{\delta,m}^\cf(\df),
$$
where $g(\cf,\df)=|\Delta^\circ\cap\zb^2|-\delta+m$,$
\beta(\cf,\df)=\beta_\Delta$, $n(\cf,\df)=n(\cf,\df;d_b)
=2M+d_t+d_b+|\Delta^\circ\cap\zb^2|-\delta-1$.
Moreover, the degree of $\mathcal Q_{\delta,m}^\cf$ in $\df$ is
$2(\delta+m)$, and its highest length-degree part is
$$
[\mathcal Q_{\delta,m}^\cf]_{\df,2(\delta+m)}
=\frac{(-1)^m3^\delta}{24^m m!\delta!}
\area(\Delta)^{\delta+m}.
$$

\item
Fix $p_r,p_l\geq1$. There is a polynomial $\mathcal Q_{\delta,m}\in\qb[\cf,\df]$ such that, for every $h$-transverse polygon $\Delta=\Delta(\cf,\df)$ {satisfying conditions} \eqref{eq:poly-range} and \eqref{eq:poly-gap-range} when $\delta>0$,
\[
N_{g(\cf,\df),\beta(\cf,\df),n(\cf,\df)}^{X_\cf|\partial X_\cf,m}
=\mathcal Q_{\delta,m}(\cf,\df).
\]
One may choose $\mathcal Q_{\delta,m}$ with degree $\delta+m$
in $\cf$, degree $2(\delta+m)$ in $\df$, and joint degree
$3(\delta+m)$. 
Its homogeneous part of length degree $2(\delta+m)$ is given by the formula in the first statement, with $\cf$ regarded as variables.
\end{itemize}
\end{corollary}

\begin{proof}
For the partitions $(\eta,\mu)=(\varnothing,1^{d_b})$,
$L=b_1=d_b$, $a_w=0$ for all $w$,
$b_w=0$ for $w>1$, and $S_{2r}=d_b$ for every $r\geq1$.
Let $B_\delta=C_\delta+d_b$. Since $C_\delta=\area(\Delta)-2\delta-d_b$, we know $B_\delta=\area(\Delta)-2\delta$.
Let $R$ be the polynomial in equation \eqref{eq:poly-log-series}, and put $n=n(\cf,\df;d_b)$.

By Proposition~\ref{prop:poly-ordinary},
$R=(d_b)_{\underline\delta}P$, where $P=P_\delta^\cf$ for fixed
$\cf$ and $P=P_\delta$ for varying $\cf$. 
Since the factor $\left(\prod_{w\geq1}w^{b_w}\right)
\frac{(L-\delta)!}{\prod_{w\geq1}b_w!}$ in equation \eqref{eq:poly-log-series} specializes to
$(d_b-\delta)!/d_b!=1/(d_b)_{\underline\delta}$, we obtain
\begin{equation*}
\sum_{v\geq0}
N_{|\Delta^\circ\cap\zb^2|-\delta+v,\beta_\Delta,n}^{X_\cf|\partial X_\cf,v}
u^{2v}
=s(u)^{B_\delta}P(e^{iu}).
\end{equation*}

Define $\mathcal Q_{\delta,m}^\cf(\df)$ and $\mathcal Q_{\delta,m}(\cf,\df)$ as follows.
$$
\mathcal Q_{\delta,m}^\cf(\df)
:=[u^{2m}]\bigl[s(u)^{B_\delta}
P_\delta^\cf(\df;e^{iu})\bigr],
$$
and
$$
\mathcal Q_{\delta,m}(\cf,\df)
:=[u^{2m}]\bigl[s(u)^{B_\delta}
P_\delta(\cf,\df;e^{iu})\bigr].
$$
By comparing coefficients of $u^{2m}$ in the above equation, we get the required equalities in the two statements.
Polynomiality and the degree statements follow by the coefficient extraction argument in the proof of Theorem~\ref{thm:log-poly-var}, with $C_\delta+S_{2r}$ replaced by $B_\delta=\area(\Delta)-2\delta$ for every $r\geq1$, and $R$ replaced by $P$.
\end{proof}

\begin{remark}
The first two coefficient polynomials for fixed directions are
\[
\begin{aligned}
\mathcal Q_{\delta,0}^\cf
&=P_\delta^\cf(\df;1),\\
\mathcal Q_{\delta,1}^\cf
&=[u^2]P_\delta^\cf(\df;e^{iu})
-\frac{\area(\Delta)-2\delta}{24}P_\delta^\cf(\df;1).
\end{aligned}
\]
In particular, since $P_0^\cf=1$, for $\delta=0$ we obtain
\[
\mathcal Q_{0,m}^\cf
=[u^{2m}]s(u)^{\area(\Delta)}.
\]
The same formulas hold for varying directions, with
$P_\delta^\cf$ and $\mathcal Q_{\delta,m}^\cf$ replaced by
$P_\delta$ and $\mathcal Q_{\delta,m}$, respectively.
\end{remark}

\subsection{Comparison of ordinary and logarithmic invariants}\label{sec:poly-ordinary-comparison}
We use the degeneration strategy of \cite[Section 3]{arguz-bousseau-2025} to give a comparison of Gromov-Witten invariants and log Gromov-Witten invariants on smooth toric surfaces.

Let $\Delta=\Delta(\cf,\df)$ be a lattice polygon with smooth normal fan and $d_b>0$.
Let $X=X_\cf$ be the corresponding smooth projective toric surface, {and let $L_\Delta$ be} the ample toric line bundle on $X$ determined by $\Delta$.
Denote by $\beta=\beta_\Delta$ the cohomology class corresponding to $L_\Delta$, and by $\rho$ a ray of the normal fan of $\Delta$.
The boundary toric divisor of $X$ is $\partial X=\bigcup_\rho D_\rho$ , where $D_\rho$ corresponds to an edge of lattice length $d_\rho$. Note that $d_\rho=\beta\cdot D_\rho$.

\begin{lemma}\label{lem:poly-boundary-exclusion}
Let $\Delta=\Delta(\cf,\df)$ be a lattice polygon with smooth normal fan and $d_b>0$.
Fix an integer $0\leq\delta\leq|\Delta^\circ\cap\zb^2|$, and
let $n=|\Delta\cap\zb^2|-\delta-1$. Suppose that for every
boundary divisor $D_\rho$ we have $d_\rho\geq\delta$.
Then every effective cycle of class $\beta$ passing through $n$ general points in the dense torus of $X$ contains no toric boundary component.
\end{lemma}

\begin{proof}
In the smooth toric surface $X$, {the boundary divisor satisfies $D_\rho\simeq\pb^1$} and $\deg(L_\Delta|_{D_\rho})=d_\rho$,
so we have
$
H^0(D_\rho,L_\Delta|_{D_\rho})
=H^0(\pb^1,\ol_{\pb^1}(d_\rho))$.
The toric monomial sections indexed by the lattice points on the corresponding edge restrict to a basis of $H^0(D_\rho,L_\Delta|_{D_\rho})$ (see \cite[Section 4.3]{cls-2011}). Thus the restriction map is surjective, and there is an exact sequence
$$
0\to H^0(X,L_{\Delta}(-D_\rho))\to H^0(X,L_\Delta)\to H^0(D_\rho,L_\Delta|_{D_\rho})\to 0.
$$
We get $h^0(X,L_{\Delta}(-D_\rho))=h^0(X,L_\Delta)-d_\rho-1$.
Therefore, $\dim|L_\Delta-D_\rho|
=\dim|L_\Delta|-d_\rho-1$.
Note that $d_\rho\geq\delta$, so we obtain $\dim|L_\Delta-D_\rho|\leq n-1$.

On the toric surface $X$, divisor classes are determined by their integral cohomology classes.
An effective cycle of class $\beta$ containing $D_\rho$ has the form $D_\rho+C'$ with $C'\in|L_\Delta-D_\rho|$.
Since the $n$ general points lie outside the boundary $\partial X$, $C'$ contains every one of them, which is a contradiction.
\end{proof}

\begin{proposition}\label{prop:poly-absolute-log}
Let $\Delta=\Delta(\cf,\df)$ be a lattice polygon with smooth normal fan and $d_b>0$.
Fix an integer $0\leq\delta\leq|\Delta^\circ\cap\zb^2|$, 
and suppose that for every
boundary divisor $D_\rho$ we have $d_\rho\geq\delta$.
For $m\geq0$, let $g=|\Delta^\circ\cap\zb^2|-\delta+m$, and $n=c_1(X_\cf)\cdot\beta_\Delta+g-m-1$.
Then, we have
$$
N^\cf_{g,\beta_\Delta,n}(m)
=N^{X_\cf|\partial X_\cf}_{g,\beta_\Delta,n}
  (m,\emptyset,1^{d_b}).
$$
\end{proposition}

\begin{proof}
The proof of this proposition follows the ideas in \cite[Section 3.3]{arguz-bousseau-2025} (see also \cite[Theorem 4.7]{dh2026} for the case $m<g$).

Let $X=X_\cf$ be the toric surface determined by the normal fan $\Sigma_X$ of $\Delta$.
Denote by $\rho_1,\ldots,\rho_j$ the rays of the fan $\Sigma_X$ that are labeled clockwise.
Let $u_i$ be the primitive integral point on $\rho_i$.
For convenience, we put $u_{j+1}=u_1$.
We first add edges connecting $u_i$ and $u_{i+1}$ for $i\in\{1,\ldots,j\}$ to the fan $\Sigma_X$, then add a ray with primitive vector $u_{k+1}$ beginning at $u_k$ for $k=1,\ldots,j-1$ and a ray with primitive vector $u_{j}$ at $u_1$. Denote by $\pl_X$ the resulting polyhedral decomposition of $\rb^2$ after the above addition of edges and rays to the fan $\Sigma_X$.
The fan $\wt\Sigma_{\pl_X}$ over the polyhedral decomposition $\pl_X$ defines a toric variety $\xk$,
and the projection from $\wt\Sigma_{\pl_X}$ to the last coordinate defines a flat toric morphism $\pi:\xk\to\cb$ according to \cite[Section 3]{ns-2006}.
Hence, we obtain a toric degeneration $\pi:\xk\to\cb$.
Geometrically, the toric variety $\xk$ is obtained from $X\times\cb$ by successively blowing up the curves $D_{\rho_i}\times\{0\}$ (the proper transformations of $D_{\rho_i}\times\{0\}$ when $i>1$) contained in the distinguished central component, in the order $i=1,\ldots,j$.
Let $P\subset\rb^2$ be the bounded subcomplex of $\pl_X$ with vertices the points $0,u_1,\ldots,u_j$.
The central fiber $\xk_0:=\pi^{-1}(0)$ has dual intersection complex $P$, and the fiber $\xk_t:=\pi^{-1}(t)$, $t\neq0$, is isomorphic to the surface $X$.
Moreover, the central fiber $\xk_0$ can be decomposed into irreducible components as
$$
\xk_0=X\bigcup\left(\cup_{i=1}^jX_{\rho_i}\right),
$$
where $X_{\rho_1}$ is the blow-up of the toric $\pb^1$-bundle over the divisor $D_{\rho_1}$ in $X$ at its two toric corners,
$X_{\rho_j}$ is the toric $\pb^1$-bundle over the divisor $D_{\rho_j}$ in $X$, and $X_{\rho_i}$ is the blow-up of the toric $\pb^1$-bundle over the divisor $D_{\rho_i}$ in $X$ at a toric corner for $i=2,\ldots,j-1$.
There is a natural projection $\nu':\xk_0\to X$.

We equip the total space $\xk$ with the divisorial log structure defined by $\xk_0$, and equip the base $\cb$ with the divisorial log structure defined by $\{0\}\subset\cb$. Then the morphism $\pi:\xk\to\cb$ can be naturally lifted to a log smooth morphism.
The tropicalization of $\xk$ with the divisorial log structure is the cone over the bounded subcomplex $P$, where $P$ need not be convex.
We choose $n:=c_1(X_\cf)\cdot \beta_{\Delta}+g-m-1$ generic sections $\undl p(t)=(p_1(t),\ldots,p_n(t))$ of the family $\pi:\xk\to\cb$ such that the points $p_i(0)$ lie in the dense torus of the irreducible component of the central fiber isomorphic to $X$. Denote by $\undl p(t)$ the $n$ point constraints.
Note that the log structure on a generic fiber $\xk_t$ is the trivial log structure. By applying the decomposition formula in log Gromov-Witten theory of \cite[Theorem 5.4]{acgs-2020} (see also \cite[Theorem B]{ranganathan2022}) to the log smooth degeneration $\pi:\xk\to\cb$, one obtains that
\begin{equation}\label{eq:decomp}
N_{g,\beta_\Delta,n}^\cf(m)=\sum_{h:\Gamma\to P}\frac{n_h}{|\aut(h)|}N_{h,\log}^{\xk_0,m},
\end{equation}
where the sum is over the rigid decorated tropical curves $h:\Gamma\to P$ of total genus $g$, total class $\beta_\Delta$, and with $n$ labelled legs.
The integer $n_h$ is the smallest positive integer by which one must dilate the tropical map so that its vertex positions and the lengths of its source edges are integral, 
and $|\aut(h)|$ is the order of the group of automorphisms of $h$. 
$N_{h,\log}^{\xk_0,m}$ is the log Gromov-Witten invariant of the central fiber $\xk_0$ endowed with the restricted log structure from $\xk$.
Let $\overline{M}_h(\xk_0)$ denote the moduli space of $h$-marked stable log maps to $\xk_0$ passing through the $n$ point constraints imposed at the marked points corresponding to the $n$ legs of $\Gamma$ (see \cite[Definition 2.31]{acgs-2020}).
The invariant $N_{h,\log}^{\xk_0,m}$ is defined as
$$
N_{h,\log}^{\xk_0,m}=\int_{[\overline{M}_h(\xk_0)]^{vir}}(-1)^{m}\lambda_{m}.
$$
From the vanishing property of $\lambda_m$ introduced in \cite[Lemma 8]{bou2019}, we know that $N_{h,\log}^{\xk_0,m}=0$ unless the genus $g_\Gamma$ of the graph $\Gamma$ satisfies $g_\Gamma\leq g-m$.

In the following, we derive a gluing formula for the invariant $N_{h,\log}^{\xk_0,m}$. 
Refine the polyhedral decomposition of the bounded complex $P$ so that $h(\Gamma)$ is contained in its one-skeleton and the images of vertices are vertices. After a finite base change and further toric refinement, choose the components used below to be smooth. Denote by $\wt\xk_0$ the central fiber of the corresponding log modification of the toric degeneration $\pi:\xk\to\cb$.
We write $\nu:\wt\xk_0\to X$ for the composite of the modification map with $\nu'$. For the stars with edges $[0,u_i]$, we use the original model without further refinement.
Let $\wt X^w$ be the irreducible component of $\wt\xk_0$ {labeled by the vertex $w$} of the refined polyhedral decomposition of $P$.
We equip every irreducible component $\wt X^w$ with the divisorial log structure defined by the intersection of $\partial\wt X^w$ with the singular locus of $\wt\xk_0$. In particular, the log structure on $\wt X^0$ is defined by its full toric boundary.
The irreducible component $\wt X^0$ corresponding to the origin of the refined polyhedral decomposition is a toric blow-up of the irreducible component $X$ of $\xk_0$.

Given a vertex $v\in\Gamma$, let $h_v:\Gamma_v\to P$ be the rigid decorated tropical curve with one vertex obtained as the star of $v$ in $h:\Gamma\to P$. Each cut bounded edge gives a labelled relative marking.
The image $h(v)$ is a vertex of the refined polyhedral decomposition of $P$, so $h(v)$ corresponds to an irreducible component $\wt X^{h(v)}$ of $\wt\xk_0$. 
Let $\overline{M}_{h_v}(\wt X^{h(v)})$ denote the moduli space of $h_v$-marked stable log maps to $\wt X^{h(v)}$ passing through the point constraints imposed at the marked points corresponding to the legs of $\Gamma$ adjacent to $v$.
Suppose that $e\subset\Gamma$ is a bounded edge of $\Gamma$.
The image $h(e)$ is also contained in an edge of the refined polyhedral decomposition of $P$, so it corresponds to an irreducible component $\wt D^{h(e)}$ of the singular locus of the modified central fiber $\wt\xk_0$.
Denote by $1_{h(e)}\in H^0(\wt D^{h(e)};\qb)$ and $pt_{h(e)}\in H^2(\wt D^{h(e)};\qb)$ the unit and the class of a point in $\wt D^{h(e)}$, respectively. 
A {\it splitting datum} $\sigma$ assigns to every half-edge $(e,v)$ 
of a bounded edge of $\Gamma$ a cohomology class $\sigma_{e,v}\in H^*(\wt D^{h(e)};\qb)$ such that either $\sigma_{e,v}=1_{h(e)}$ and $\sigma_{e,v'}=pt_{h(e)}$, or $\sigma_{e,v}=pt_{h(e)}$ and $\sigma_{e,v'}=1_{h(e)}$, where $v$ and $v'$ are the two endpoints of the edge $e$.
Suppose that $v\in\Gamma$ is a vertex with genus decoration $g_v$ and class decoration $d_v$, and $\sigma$ is a splitting datum of $\Gamma$. 
Given a non-negative integer $g_v'\leq g_v$, let $N_{h,\sigma,\log}^{v,g_v'}$ be the log Gromov-Witten invariant defined as
\begin{equation}\label{eq:log-gw1}
N_{h,\sigma,\log}^{v,g_v'}:=\int_{[\overline M_{h_v}(\wt X^{h(v)})]^{vir}}(-1)^{g_v'}\lambda_{g_v'}\prod_{e\ni v}ev_e^*(\sigma_{e,v}),
\end{equation}
where $ev_e:\overline M_{h_v}(\wt X^{h(v)})\to \wt D^{h(e)}$ are the evaluation maps at the corresponding marked points, and the product is taken over the edges of $\Gamma$ adjacent to $v$ that are considered as legs of $\Gamma_v$.

From Proposition \ref{prop:poly-expanded-gluing}, we obtain the following gluing formula:
\begin{equation}\label{eq:glue}
N_{h,\log}^{\xk_0,m}=\sum_{\sigma,g_v'}\prod_ew_e\prod_{v}N_{h,\sigma,\log}^{v,g_v'},
\end{equation}
where the sum is taken over the splitting data $\sigma$ and integers $g_v'$ such that $0\leq g_v'\leq g_v$ and $\sum g_v'=m$, and $w_e$ are weights of the edges of $\Gamma$.
Proposition \ref{prop:poly-expanded-gluing} also proves that both sides of equation $(\ref{eq:glue})$ vanish unless $\Gamma$ is a star with a unique vertex $v_0$ of class $\beta_\Delta$ satisfying $h(v_0)=0$. It has $d_\rho$ leaves at $u_\rho$, each of fiber class, connected to the vertex $v_0$ by an edge of weight one.
For such a star the only potentially nonzero splittings have
$\sigma_{e,v_0}=1$, $\sigma_{e,v}=pt$ for $v\ne v_0$,
$g_{v_0}'=g_{v_0}-(g-m)$, $g_v'=g_v$ for $v\ne v_0$.

The invariants corresponding to leaves are computed by \cite[Lemma 15]{bou2019} and its proof.
We give a brief calculation here.
The relative point insertion fixes a fiber over $D_\rho^\circ$, where its normal bundle is trivial.
The obstruction theory in the normal direction contributes
$c_{g_v}(\eb_v^\vee)=(-1)^{g_v}\lambda_{g_v}$.
Together with the prescribed top Hodge insertion, this gives $\lambda_{g_v}^2$, which vanishes for $g_v>0$ by Mumford's relation \cite{mumford-1983}.
For $g_v=0$, the degree one relative fiber map is unique and has trivial automorphism group, so its contribution is one.
Thus every surviving leaf has genus zero, while $g_{v_0}=g$ and $g_{v_0}'=m$.

For this star, $n_h=1$ and $|\aut(h)|=\prod_\rho d_\rho!$.
The contribution of vertex $v_0$ is the logarithmic invariant with all simple boundary contacts labelled.
Equations (\ref{eq:decomp}) and (\ref{eq:glue}) give
$$
\begin{aligned}
N_{g,\beta_\Delta,n}^\cf(m)
&=\frac{1}{\prod_\rho d_\rho!}
\int_{[\overline M^{\log}_{g,n}
(X|\partial X,\beta_\Delta,\emptyset,\vec{1}^{\,d_b})]^{vir}}
(-1)^m\lambda_m\prod_{i=1}^n ev_i^*(pt)\\
&=N^{X_\cf|\partial X_\cf}_{g,\beta_\Delta,n}
(m,\emptyset,1^{d_b}).
\end{aligned}
$$

\end{proof}

The following vanishing result is a combination of \cite[Proposition 2, Lemma 8]{bou2019}.

\begin{lemma}[{\cite[Proposition 2, Lemma 8]{bou2019}}]\label{lem:poly-incidence-hodge}
Let $S$ be a smooth projective toric surface with toric boundary $\partial S$.
Let $\mathbf p=(p_1,\ldots,p_n)$ be $n$ general points in the dense torus of $S$.
Fix a nonzero curve class $\beta\in H_2(S;\zb)$, an integer $g\geq0$, and $s$ labelled relative markings with prescribed nonzero contact vectors along the boundary.
Let $\mathcal M(\mathbf p)$ be the proper moduli space of connected stable log maps of genus $g$ and class $\beta$, with $n$ ordinary markings mapping to the general points $\mathbf p=(p_1,\ldots,p_n)$ and all boundary contacts specified by the $s$ labelled relative markings.
{Let $q=\max\{0,n-s+1\}$, then} we have
\begin{itemize}
\item If $q\leq g$, then
$
\lambda_j|_{\mathcal M(\mathbf p)}=0
\text{ in }H^{2j}(\mathcal M(\mathbf p),\qb)
\text{ for }j>g-q.
$
\item If $q>g$, then $\mathcal M(\mathbf p)$ is empty.
\end{itemize}
\end{lemma}

\begin{proof}
Denote by $\Delta$ the balanced collection formed by the prescribed $s$ contact vectors.
For $q=0$, the assertion follows from the rank bound of the Hodge bundle. 
In the following we consider the case $q>0$. 
Let $P_1,\ldots,P_n\in\rb^2$ be $n$ rational points in tropical general position for $\Delta$.
Choose a toric degeneration $\mathscr S\to\cb^1$ with a rational polyhedral decomposition containing the $P_i$ as vertices as in \cite[Section 4.2]{bou2019}. Choose general point sections specializing to the dense tori of the components indexed by the $P_i$. 
Equip the family with its full toric log structure.
Let $\mathscr M\to\cb^1$ be the proper moduli space of stable log maps to the fibers of $\mathscr S$ with the fixed data and these point constraints.
Let $\Gamma$ be a connected parametrized tropical curve of degree $\Delta$ passing through $P_1,\ldots,P_n$.
From \cite[Proposition 2]{bou2019}, using its genus bound and footnote 7 after discarding vertex genera and contracting contracted edges, we know
$
b_1(\Gamma)\geq n-s+1=q.
$
Hence every source curve in the special fiber of $\mathscr M\to\cb^1$ belongs to the locus of curves whose dual graph has at least $q$ independent cycles. By \cite[Lemma 8]{bou2019}, if $q\leq g$,
$\lambda_j=0$ for any $j>g-q$ in rational cohomology on the point-constrained special fiber.
If $q>g$, the special fiber is empty.

Proper base change gives a specialization map from the rational cohomology of the special fiber to that of a nearby fiber, compatible with the restriction of classes on the total family.
The Hodge bundle is defined on the total family, so its Chern classes specialize compatibly. Their vanishing on the special fiber therefore implies the asserted vanishing on a general nearby fiber. If the special fiber is empty, properness implies that the nearby fibers are empty.
The point constraints $\mathbf p$ can be chosen so that these conclusions hold simultaneously. 
\end{proof}

Proposition \ref{prop:poly-expanded-gluing} can be obtained as in the proof of \cite[Proposition 13]{bou2019}.
In Proposition \ref{prop:poly-expanded-gluing}, we use the notations in the proof of Proposition \ref{prop:poly-absolute-log} and assume its hypotheses.
\begin{proposition}\label{prop:poly-expanded-gluing}
Let $h:\Gamma\to P$ be a rigid decorated tropical curve that contributes to the sum in the right-hand side of equation $(\ref{eq:decomp})$, and let $V_0=\{v:h(v)=0\}$.
Then, for the integer $m$ and the corresponding number $n$ of point constraints fixed above, the gluing formula in equation $(\ref{eq:glue})$ holds.
Moreover, we have the following statements.
\begin{itemize}
\item Both sides of equation $(\ref{eq:glue})$ vanish unless the curve $\Gamma$ is a star with one vertex $v_0$ such that $h(v_0)=0$ and which has class decoration $\beta_\Delta$. Furthermore, for each ray $\rho$, $\Gamma$ has exactly $d_\rho$ leaves at $u_\rho$ of fiber class, with all edge weights one.
\item For a surviving star, a nonzero vertex product on the right-hand side of equation $(\ref{eq:glue})$ has the unit insertion at each flag adjacent to $v_0$ and the point insertion at the opposite flag, with
$g_{v_0}'=g_{v_0}-(g-m)$, and $g_v'=g_v$ for $v\ne v_0$.
\end{itemize}
\end{proposition}

\begin{proof}

For $v\in V_0$, let $m_v$ be the number of legs of $\Gamma$ corresponding to point markings and $s_v$ the number of incident bounded edge flags. Let $s=\sum_{v\in V_0}s_v$.
All legs of $\Gamma$ corresponding to point markings are adjacent to vertices in $V_0$, so $\sum_{v\in V_0}m_v=n$.
A rigid type has no contracted bounded edge, so no bounded edge joins two vertices in $V_0$.

If $\overline M_h(\xk_0)$ or $\prod_{v}\overline M_{h_v}(\wt X^{h(v)})$ is nonempty, its maps determine an effective image cycle of class $\beta_\Delta$ through the $n$ general torus points.
For the product $\prod_{v}\overline M_{h_v}(\wt X^{h(v)})$, take the sum of the pushforward cycles of its factors.
For a vertex $v\in V_0$, the class $d_v$ with zero pushforward class cannot carry a nonzero boundary contact. 
Such a vertex has no incident bounded edge, so it cannot occur in a connected rigid type {of nonzero total degree}.
Thus for every $v\in V_0$ the class $d_v$ has nonzero pushforward class, and its point markings also lie on the effective image cycle.
From Lemma \ref{lem:poly-boundary-exclusion} we have
$$
\nu_*d_v=0 \text{ for } v\notin V_0,\text{ and }
\sum_{v\in V_0}\nu_*d_v=\beta_\Delta.
$$

For $v\in V_0$, the virtual dimension of $\overline M_{h_v}(\wt X^0)$ is
\begin{equation}\label{eq:dim-eq1}
\vdim\overline M_{h_v}(\wt X^0)
=g_v-1+s_v-m_v.
\end{equation}
For a flag $(e,v)$ with $v\in V_0$, write its contact vector as $w_e\wt u_e$.
In a smooth cone of the original fan, express the primitive vector $\wt u_e$ as
$\wt u_e=a_eu_i+b_eu_{i+1}$, $a_e,b_e\in\zb_{\geq0}$, and let $c_e=w_e(a_e+b_e)$.
The pullback of $\partial X$ to $\wt X^0$ has coefficient $a_e+b_e$ along the toric divisor of $\wt X^0$ corresponding to $\wt u_e$.
Note that $c_1(X)=\partial X$ and $\sum_{v\in V_0}\nu_*d_v=\beta_\Delta$.
Summing over the vertices in $V_0$, we have
\begin{equation}\label{eq:dim-count1}
c_1(X)\cdot\beta_\Delta=\sum_{v\in V_0}\sum_{e\ni v}c_e,\qquad
c_1(X)\cdot\beta_\Delta-s=\sum_{v\in V_0}\sum_{e\ni v}(c_e-1)\geq0.
\end{equation}
Here $c_e\geq1$, and $c_e=1$ precisely for a weight-one contact along a ray of the original fan.

The construction in \cite[Section 7.3, especially Lemma 7.3.1]{ranganathan2022} provides cutting morphisms to ordinary logarithmic vertex spaces, after proper virtual modifications.
Contraction of expansions preserves the Hodge bundles and ordinary evaluations, and the virtual pushforward property allows us to compute the original integral there.
Therefore, by \cite[Lemma 7]{bou2019}, we have
$$
(-1)^m\lambda_m
=\sum_{\substack{0\leq g_v'\leq g_v\\\sum_vg_v'=m}}
\prod_v(-1)^{g_v'}\lambda_{g_v'}.
$$
Let $a_v=g_v-g_v'$. From Lemma \ref{lem:poly-incidence-hodge}, the class $\lambda_{g_v'}$ vanishes on the entire {point-constrained vertex space} for $v\in V_0$ unless
$a_v\geq m_v-s_v+1$. 
If $m_v-s_v+1$ is negative, the inequality follows from $a_v\geq0$.
The same necessary inequalities hold for every nonzero product of the vertex invariants in equation (\ref{eq:glue}).
Since $g=b_1(\Gamma)+\sum_vg_v$, a potentially nonzero summand satisfies
$$
g-m=b_1(\Gamma)+\sum_va_v\geq b_1(\Gamma)+\sum_{v\in V_0}(m_v-s_v+1)=b_1(\Gamma)+n-s+|V_0|.
$$
Since $n=c_1(X)\cdot\beta_\Delta+g-m-1$, we have
\begin{equation}\label{eq:dim-count2}
1\geq |V_0|+b_1(\Gamma)+(c_1(X)\cdot\beta_\Delta-s).
\end{equation}
All three terms on the right are nonnegative and $|V_0|\geq1$, so we get
$$
|V_0|=1, b_1(\Gamma)=0, s=c_1(X)\cdot\beta_\Delta.
$$
From equation (\ref{eq:dim-count1}), $c_e=1$ at every flag $(e,v)$ with $v\in V_0$.
Let $v_0$ be the vertex in $V_0$.
The equality also implies that
$a_{v_0}=n-s+|V_0|=g-m$, and $a_v=0$ for $v\ne v_0$.
If $|V_0|\ne1$, $b_1(\Gamma)\ne0$, or $s\ne c_1(X)\cdot\beta_\Delta$, the insertion $\lambda_m$ vanishes on the expanded global space, and
every vertex product on the right of equation (\ref{eq:glue}) vanishes.
Thus both sides are zero.

We now use expanded gluing for a tree as in \cite[Section 7.5.3]{ranganathan2022}.
Cutting a bounded edge $e$ separates it into two connected pieces, with moduli spaces $K_1,K_2$ and independent evaluations to the logarithmic curve $\wt D^{h(e)}$.
The product evaluation to $\wt D^{h(e)}\times\wt D^{h(e)}$ is combinatorially flat, since each factor has target cones of dimension at most one.
After a common subdivision, the strict transform of the relative diagonal pushes forward to the ordinary diagonal.
Compatibility of the virtual classes with this pushforward gives ordinary diagonal gluing across $e$.
By induction over the bounded edges, we have the ordinary diagonal gluing for every bounded edge.

The ordinary diagonal decomposition
$$
[\Delta_{\wt D^{h(e)}}]
=1_{h(e)}\otimes pt_{h(e)}+pt_{h(e)}\otimes1_{h(e)}
$$
and tree gluing give 
\begin{equation}\label{eq:gluing11}
N_{h,\log}^{\xk_0,m}=A_h\sum_{\substack{\sigma,(g_v')\\0\leq g_v'\leq g_v,\ \sum_vg_v'=m}}
\prod_vN_{h,\sigma,\log}^{v,g_v'},
\end{equation}
where $A_h$ is the rational gluing multiplicity. 
For the vertex $v_0$, equation (\ref{eq:dim-eq1}) becomes
$$
\vdim\overline M_{h_{v_0}}(\wt X^0)
=g_{v_0}-1+c_1(X)\cdot\beta_\Delta-n=g_{v_0}-g+m=g_{v_0}'.
$$
Thus every nonzero term of the {right-hand side} of equation (\ref{eq:gluing11}) has the unit insertion at all flags adjacent to $v_0$ and the point class insertion at the flags adjacent to leaves.

Represent the point insertions at the flags adjacent to leaves by general points of the open toric orbits $D_\rho^\circ$.
Remove $v_0$ from the tree and let $T$ be one of the remaining connected subtrees.
Every component of the resulting connected map has zero pushforward class under $\nu$, so its image in $X$ is a single point.
The insertion at its attaching flag forces this point to be the chosen general point $p\in D_\rho^\circ$.
Over $p$, the original central fiber is the single ruled-surface fiber $\pb^1\subset X_\rho$.
The toric corner blowups do not meet this fiber, and further target modifications replace it only by a chain. Since the incoming contact order is one, the degree along this chain is one.
It follows that, in the original complex, $T$ is a single leaf at $u_\rho$, with fiber class and {a weight-one attaching edge}.
There is no other type of $T$ with non-trivial contribution to equation (\ref{eq:gluing11}).
Consequently, $N_{h,\log}^{\xk_0,m}$ vanishes unless $\Gamma$ is a star with exactly $d_\rho$ leaves in each direction.
For these stars we use the original reduced model. The contact orders are all one, so logarithmic enhancement introduces no nontrivial choice for $v_0$. 
All tropical vertices of $\Gamma$ are integral and every source edge has integral length one.
Thus the primitive basic ray has index one and the labelled gluing multiplicity is $A_h=1$.
Since all weights are one, $A_h=\prod_ew_e$ for the surviving stars.
For every other type both sides of equation (\ref{eq:glue}) vanish, so the equality holds.
\end{proof}

\begin{corollary}\label{cor:poly-absolute-log-combined}
Let $\Delta=\Delta(\cf,\df)$ be a lattice polygon with smooth normal fan and $d_b>0$.
Fix an integer $\delta\in\{0,\ldots,|\Delta^\circ\cap\zb^2|\}$, 
and suppose that for every
boundary divisor $D_\rho$ we have $d_\rho\geq\delta$.
For $m\geq0$, let $g=|\Delta^\circ\cap\zb^2|-\delta+m$, and $n=c_1(X_\cf)\cdot\beta_\Delta+g-m-1$.
Then, we have the following equality of the combined invariants:
$$
N^{\cf,m}_{g,\beta_\Delta,n}
=N^{X_\cf|\partial X_\cf,m}_{g,\beta_\Delta,n}.
$$
\end{corollary}

\begin{proof}
The combined Gromov-Witten invariants $N^{\cf,m}_{g,\beta_\Delta,n}$ are defined in equation \eqref{eq:comb-gw}.
Fix a partition $\Lambda=\Lambda_1\cup\cdots\cup\Lambda_t$ and $t$ triples $(\Delta_i,g_i,m_i)$ 
satisfying the condition \eqref{eq:comb-gw-cond}.
Let $\beta_{i},n_i$ be the data determined by the partition and the triple $(\Delta_i,g_i,m_i)$.
Take a summand $\prod_{i=1}^tN_{g_i,\beta_{i},n_i}^{\cf}(m_i)$ in equation \eqref{eq:comb-gw}.
Apply the degeneration in the proof of Proposition \ref{prop:poly-absolute-log} simultaneously 
to these factors $N_{g_i,\beta_{i},n_i}^{\cf}(m_i)$ with data $(g_i,\beta_{i},n_i,m_i)$. 
All ordinary point constraints are chosen generally on $X$. 
Apply Lemma \ref{lem:poly-boundary-exclusion} to the sum of the pushforward cycles of the $t$ components, 
whose class is $\beta_\Delta$ and which passes through all $n$ points. 
Thus no boundary component occurs in this total cycle. 
Note that this does not require the individual classes of the $t$ components to satisfy the edge-length hypothesis.
The component classes may include nef segment classes: each is nonzero and carries $n_i\geq1$ ordinary point conditions.

Use the notation $\Gamma,V_0,s,m_v,g_v,g_v'$ of Proposition \ref{prop:poly-expanded-gluing} 
for the disjoint union of the $t$ decorated graphs. 
Since $\Lambda_i\neq\emptyset$, every connected component carries an 
ordinary point condition and $|V_0|\geq t$. The arithmetic genus of $\Gamma$ is 
$g=b_1(\Gamma)+\sum_vg_v+1-t$. 
The same Hodge-vanishing inequalities at the vertices in $V_0$ give
$$
g-m\geq b_1(\Gamma)+n-s+|V_0|+1-t.
$$
Using $n=c_1(X)\cdot\beta_\Delta+g-m-1$, we obtain
$$
t\geq |V_0|+b_1(\Gamma)+\bigl(c_1(X)\cdot\beta_\Delta-s\bigr).
$$
The last two terms are nonnegative by equation \eqref{eq:dim-count1}. 
Hence, $|V_0|=t$, $b_1(\Gamma)=0$, and $c_e=1$ for every flag $(e,v)$ with $v\in V_0$. 
Each connected component of $\Gamma$ has one vertex in $V_0$. 
The separating-edge gluing and fiber calculation in the proof of Proposition 
\ref{prop:poly-absolute-log} now apply to each component, 
giving a star with genus-zero leaves.

The factors $\prod_\rho d_{\rho,i}!$ arising from leaf automorphisms 
give exactly the normalization of the connected logarithmic factors. 
Finally, the unordered partitions of the ordinary point labels, 
the component classes, and the Hodge classes agree in the two definitions 
when $\eta=\emptyset$ and $\mu=1^{d_b}$. Summing over all the summands, we obtain the required equality.
\end{proof}

\begin{remark}\label{rem:poly-delpezzo-comparison}
For a smooth toric del Pezzo surface, the known connected correspondences \cite[Theorem 3.9]{arguz-bousseau-2025} and \cite[Theorems 4.6--4.7]{dh2026} also give the equalities in Proposition \ref{prop:poly-absolute-log} and Corollary \ref{cor:poly-absolute-log-combined} without the edge-length hypothesis $d_\rho\geq\delta$.
\end{remark}

\subsection{Polynomiality for Gromov-Witten invariants}
\label{sec:poly-ordinary-gw}
We prove the polynomiality of combined Gromov--Witten invariants and the asymptotics of connected invariants.
Let $\Delta=\Delta(\cf,\df)$ be a lattice polygon with smooth normal fan and $d_b>0$.
For fixed $\delta,m$, we use the following notations:
$$
\begin{aligned}
g(\cf,\df)&=|\Delta^\circ\cap\zb^2|-\delta+m,~
\beta(\cf,\df)=\beta_\cf(\df)=\beta_\Delta,\\
n(\cf,\df)&=\ell_0+|\Delta^\circ\cap\zb^2|-\delta-1
 =\frac{\area(\Delta)+\ell_0}{2}-\delta, ~\ell_0=2M+d_t+d_b.
\end{aligned}
$$
The numbers $g,n$ have degree at most two in $\df$ and at most one in $\cf$.

\begin{theorem}\label{thm:ordinary-poly-fixed}
Fix $\cf$ with smooth normal fan and integers $\delta,m\geq0$. There exists a polynomial $\ql_{\delta,m}^\cf(\df)$ such that, for every sufficiently large and spread out $\df$ (see equation $(\ref{eq:poly-range})$),
$$
N^{\cf,m}_{g(\df),\beta(\df),n(\df)}
=\mathcal Q_{\delta,m}^\cf(\df).
$$
The polynomial $\mathcal Q_{\delta,m}^\cf$ is given in Corollary \ref{thm:log-simple-var}.
Moreover, its degree in $\df$ is $2(\delta+m)$, and its highest homogeneous part is given in Corollary \ref{thm:log-simple-var}.
For $\delta=0$, we only need $d_b>0$ to obtain the polynomial.
\end{theorem}

\begin{proof}
When $\delta>0$, from equation (\ref{eq:poly-range}), the length of every edge of $\Delta(\cf,\df)$ is at least $\delta$.
Lemma \ref{lem:poly-defect}(2) gives $W_j\geq2\delta$ and equation (\ref{eq:poly-range}) gives $M\geq\delta+1$, hence
$$
|\Delta^\circ\cap\zb^2|=\sum_{j=1}^{M-1}(W_j-1)
\geq\delta(2\delta-1)\geq\delta.
$$
For $\delta=0$, $|\Delta^\circ\cap\zb^2|\geq0$ and the boundary length condition is automatic.
Therefore, the statement follows from Corollary \ref{cor:poly-absolute-log-combined} and 
Corollary \ref{thm:log-simple-fix}.

\end{proof}

Let $\Delta=\Delta(\cf,\df)$ be an $h$-transverse lattice polygon with smooth normal fan and $d_b>0$, and let $X=X_\cf$.
Fix an integer $0\leq g'\leq |\Delta^\circ\cap\zb^2|$, let $\delta=|\Delta^\circ\cap\zb^2|-g'$, and assume $d_\rho\geq\delta$ for every toric boundary divisor $D_\rho$.
Note that
$c_1(X)\cdot\beta_\Delta=\sum_\rho d_\rho
  =2M+d_t+d_b$, $n=c_1(X)\cdot\beta_\Delta+g'-1$.
In the notation of \cite[equation (1)]{dh2026}, we consider
\begin{equation}\label{eq:ding-hu-series}
P^{X}_{g',\beta_\Delta}(u)
:=\bigl(2\sin(u/2)\bigr)^{2-2g'-c_1(X)\cdot\beta_\Delta}
\sum_{g\geq g'}N^\cf_{g,\beta_\Delta,n}(g-g')
u^{2g-2+c_1(X)\cdot\beta_\Delta}.
\end{equation}
The invariants in this series are the connected ordinary invariants of Section \ref{sec:corr-gw}.

\begin{theorem}\label{thm:ding-hu-bps}
Let $\Delta=\Delta(\cf,\df)$ be an $h$-transverse lattice polygon with smooth normal fan and $d_b>0$.
Fix an integer $0\leq\delta< |\Delta^\circ\cap\zb^2|$ and suppose that $d_\rho\geq\delta$ for every toric boundary divisor $D_\rho$.
Let $g'=|\Delta^\circ\cap\zb^2|-\delta$ and $n=c_1(X_\cf)\cdot\beta_\Delta+g'-1$.
Then the series $P^{X_\cf}_{g',\beta_\Delta}(u)$ in $u$ with rational coefficients is a Laurent polynomial in $q$ with integer coefficients under the change of variables $q=e^{iu}$.
\end{theorem}

\begin{proof}
For every $g\geq g'$, set $m=g-g'$.
Then $0\leq m\leq g$, $g=|\Delta^\circ\cap\zb^2|-\delta+m$, and $n=c_1(X_\cf)\cdot\beta_\Delta+g-m-1$.
From Proposition \ref{prop:poly-absolute-log}, we get
$$
N^\cf_{g,\beta_\Delta,n}(g-g')
=N^{X_\cf|\partial X_\cf}_{g,\beta_\Delta,n}
(g-g',\emptyset,1^{d_b}).
$$
Since $(-i)(q^{1/2}-q^{-1/2})=2\sin(u/2)$, it follows from Theorem \ref{thm:log-trop} that
$$
\sum_{g\geq g'}N^\cf_{g,\beta_\Delta,n}(g-g')u^{2g-2+c_1(X)\cdot\beta_\Delta}
=N_{c,\mathrm{trop}}^{\Delta,g'}(\emptyset,1^{d_b};q)
\bigl(2\sin(u/2)\bigr)^{2g'-2+c_1(X)\cdot\beta_\Delta}.
$$
Therefore, $P^X_{g',\beta_\Delta}(u)=N_{c,\mathrm{trop}}^{\Delta,g'}(\emptyset,1^{d_b};q)$, under the change of variables $q=e^{iu}$.
From Theorem \ref{thm:floor-trop}, we obtain that
$$
\begin{aligned}
N_{c,\mathrm{trop}}^{\Delta,g'}(\emptyset,1^{d_b};q)
&=N_{c,\floor}^{\Delta,g'}(\emptyset,1^{d_b};q)\\
&=\sum_{\dl\in\fdf_c(\Delta,g')}
\nu_{\emptyset,1^{d_b}}(\dl)
\prod_{e\in E(\dl)}[w(e)]_q^2.
\end{aligned}
$$
Hence, $P^X_{g',\beta_\Delta}(u)$ is a Laurent polynomial in $q$ with integer coefficients 
under the change of variables $q=e^{iu}$.
\end{proof}

\begin{corollary}\label{cor:ordinary-connected-asymptotics}
Fix $\cf$ with smooth normal fan and integers $\delta,m\geq0$. Let $\df$ be sufficiently large and spread out with $\cf$.
Suppose that, for every decomposition
$\Delta=\Delta_1+\cdots+\Delta_t$, $t\geq2$, one has
$\delta<\sum_{i<j}\ml(\Delta_i,\Delta_j)$. 
Then
$$
N^\cf_{g(\df),\beta(\df),n(\df)}(m)
=N^{\cf,m}_{g(\df),\beta(\df),n(\df)}
=\ql_{\delta,m}^\cf(\df).
$$
In particular, fix a polygon $\Delta=\Delta(\cf,\df)$ such that $k\Delta=\Delta(\cf,k\df)$ satisfies
the condition \eqref{eq:poly-range} for all sufficiently large positive integers $k$. 
If $\delta+m>0$, the connected invariants satisfy
$$
N^\cf_{g(k\df),\beta(k\df),n(k\df)}(m)
={}\frac{(-1)^m3^\delta}{24^m m!\delta!}
\area(\Delta)^{\delta+m}k^{2(\delta+m)}
\quad+O\bigl(k^{2(\delta+m)-1}\bigr).
$$
In particular, for $m=0$ and fixed $\delta>0$, the leading term is $3^\delta\area(\Delta)^\delta
k^{2\delta}/\delta!$.
\end{corollary}

\begin{proof}
Let $g'=|\Delta^\circ\cap\zb^2|-\delta$. A reducible nodal tropical curve of degree $\Delta$ and
genus $g'$ has component degrees $\Delta_i$ and nonnegative component cogenera $\delta_i$, with
$$
\delta=\sum_i\delta_i+\sum_{i<j}\ml(\Delta_i,\Delta_j).
$$
Then the hypothesis $\delta<\sum_{i<j}\ml(\Delta_i,\Delta_j)$ implies that {there are no reducible tropical curves},
and $N^\cf_{g(\df),\beta(\df),n(\df)}(m)=N^{\cf,m}_{g(\df),\beta(\df),n(\df)}$.
From Theorem \ref{thm:ordinary-poly-fixed}, $N^\cf_{g(\df),\beta(\df),n(\df)}(m)=\ql_{\delta,m}^\cf$.

We next verify that the strict mixed-intersection condition is automatic along a sufficiently large
dilation.  
Suppose, for a contradiction, that $k\Delta=\sum_{i=1}^t\Delta_i$ is a nontrivial Minkowski decomposition of $k\Delta$ with $\sum_{i<j}\ml(\Delta_i,\Delta_j)\leq\delta$. 
Let $\beta_1,\ldots,\beta_t$ be the nonzero nef integral divisor classes corresponding to
$\Delta_1,\ldots,\Delta_t$, respectively. 
Note that $\beta_1^2=\area(\Delta_1)$, $(\sum_{i=2}^t\beta_i)^2=\area(\sum_{i=2}^t\Delta_i)$ 
and $\beta_i\cdot\beta_j=\ml(\Delta_i,\Delta_j)$,  so we have
$$
\beta_1\cdot\sum_{i=2}^t\beta_i=\sum_{i=2}^t\ml(\Delta_1,\Delta_i)\leq\delta.
$$
If $\beta_1^2=0$, we have $\delta\geq\beta_1\cdot\sum_{i=2}^t\beta_i=k\beta_{\Delta}\cdot\beta_1\geq k$, since $\beta_\Delta$ is ample and $\beta_1$ is nonzero. 
The same argument applies when $(\sum_{i=2}^t\beta_i)^2=0$.
If $\beta_1^2>0$ and $(\sum_{i=2}^t\beta_i)^2>0$, 
the Hodge index theorem \cite[Chapter V, Section 1]{hartshorne1977} gives
$$
\beta_1^2(\sum_{i=2}^t\beta_i)^2\leq(\beta_1\cdot\sum_{i=2}^t\beta_i)^2\leq\delta^2.
$$
Consequently, 
$$
k^2\area(\Delta)=(\beta_1+\sum_{i=2}^t\beta_i)^2\leq2\delta^2+2\delta.
$$
Thus no such decomposition can exist once $k>\delta$ and $k^2\area(\Delta)>2\delta^2+2\delta$. This
proves the required eventual condition uniformly over all decompositions.
For sufficiently large $k$, Corollary \ref{thm:log-simple-fix} identifies the highest homogeneous part of this polynomial. 
Since $\area(k\Delta)=k^2\area(\Delta)$, we obtain the displayed expansion. 
\end{proof}

\section*{Acknowledgements}
Y. Ding would like to thank Jianxun Hu for valuable discussions and continuous support.
Y. Ding is partially supported by the Natural Science Foundation of Henan Province (No.262300421872).

\bibliographystyle{abbrv}
\bibliography{sg}

\end{document}